\documentclass[11pt]{amsart}
\usepackage{graphicx}
\usepackage{amssymb,amscd,color,latexsym,epsfig,xypic,hyperref, float}
\usepackage[mathscr]{euscript}

\usepackage{tikz-cd}
\usepackage{tikz}
\usetikzlibrary{calc,trees,positioning,arrows,chains,shapes.geometric,%
    decorations.pathreplacing,decorations.pathmorphing,shapes,%
    matrix,shapes.symbols,arrows.meta}

\tikzcdset{tikzcd left hook/.tip = {cm left hook[scale=2]},
                 tikzcd right hook/.tip = {cm left hook[scale=2,swap]}}

\tikzset{
>=stealth',
  punktchain/.style={
    rectangle,
    rounded corners,
    draw=black, thick,
    minimum height=3em,
    text centered,
    on chain},
  line/.style={draw, thick, <-},
  element/.style={
    tape,
    top color=white,
    bottom color=blue!50!black!60!,
    minimum width=8em,
    draw=blue!40!black!90, very thick,
    text width=10em,
    minimum height=3.5em,
    text centered,
    on chain},
  every join/.style={->, thick,shorten >=1pt},
  decoration={brace},
  tuborg/.style={decorate},
  tubnode/.style={midway, right=2pt},
}

\DeclareFontFamily{OT1}{rsfs}{}
\DeclareFontShape{OT1}{rsfs}{n}{it}{<-> rsfs10}{}
\DeclareMathAlphabet{\curly}{OT1}{rsfs}{n}{it}

\DeclareFontFamily{U}{mathb}{\hyphenchar\font45}
\DeclareFontShape{U}{mathb}{m}{n}{
      <5> <6> <7> <8> <9> <10> gen * mathb
      <10.95> mathb10 <12> <14.4> <17.28> <20.74> <24.88> mathb12
      }{}
\DeclareSymbolFont{mathb}{U}{mathb}{m}{n}

\makeatletter
\newcommand{\eqnum}{\refstepcounter{equation}\textup{\tagform@{\theequation}}}
\makeatother
\definecolor{darkgreen}{RGB}{10,130,0}

\renewcommand\;{\hspace{.6pt}}

\newcommand\C{\mathbb C}
\newcommand\Q{\mathbb Q}
\newcommand\R{\mathbb R}

\newcommand\Z{\mathbb Z}
\renewcommand\AA{\mathbb A}

\newcommand\EE{\mathbb E}

\newcommand\HH{\mathbb H}

\newcommand\LL{\mathbb L}

\newcommand\PP{\mathbb P}

\newcommand\TT{\mathbb T}

\newcommand\cO{\mathcal O}

\newcommand\cC{\mathcal C}

\renewcommand\cH{\mathcal H}
\newcommand\cI{\mathcal I}

\newcommand\ft{\mathfrak t}
\newcommand\fg{\mathfrak g}
\newcommand\fh{\mathfrak h}

\newcommand{\fs}{\mathfrak{s}}

\newcommand{\sH}{\mathsf{H}}
\newcommand{\sG}{\mathsf{G}}

\newcommand{\sL}{\mathsf{L}}
\newcommand{\sS}{\mathsf{S}}
\newcommand{\sT}{\mathsf{T}}
\newcommand{\sX}{\mathsf{X}}

\newcommand{\sZ}{\mathsf{Z}}

\newcommand{\Bd}{\theta}

\newcommand{\Bz}{\mathbf{z}}

\newcommand\eps{\epsilon}

\renewcommand\({\big(}
\renewcommand\){\big)}

\makeatletter
\newcommand{\so}{\ \ext@arrow 0359\Rightarrowfill@{}{\hspace{3mm}}\ }
\makeatother
\newcommand{\rt}[1]{\xrightarrow{\ #1\ }}
\newcommand\To{\longrightarrow}

\newcommand\into{\hookrightarrow}
\newcommand\INTO{\ \ar@{^(->}[r]<-.2ex>}
\newcommand{\Into}{\ensuremath{\lhook\joinrel\relbar\joinrel\rightarrow}}
\newcommand\onto{\to\hspace{-3mm}\to}
\newcommand\Onto{\longrightarrow\hspace{-5.5mm}\longrightarrow}
\newcommand\Mapsto{\ \longmapsto\ }

\newcommand\wt{\widetilde}
\renewcommand\_{^{}_}

\newcommand\<{\langle}
\renewcommand\>{\rangle}

\newfont{\bigtimesfont}{cmsy10 scaled \magstep5}
\newcommand{\bigtimes}{\mathop{\lower0.9ex\hbox{\bigtimesfont\symbol2}}}
\renewcommand\={\ =\ }

\newcommand\udot{^{\bullet}}
\newcommand{\Wedge}{\mbox{\Large $\wedge$}}
\newcommand\acts{\curvearrowright}

\DeclareMathSymbol{\lefttorightarrow}{3}{mathb}{"FC}
\DeclareMathSymbol{\righttoleftarrow}{3}{mathb}{"FD}

\newcommand\op{\operatorname}
\newcommand\rank{\op{rank}}
\newcommand\codim{\op{codim}}
\newcommand\pt{\op{pt}}
\renewcommand\ss{\op{ss}}

\newcommand\rk{\op{rank}}
\newcommand\vir{\op{vir}}

\newcommand\res{\op{res}}
\newcommand\Res{\op{Res}}
\newcommand\red{\op{red}}
\newcommand\sr{\op{sr}}
\newcommand\vd{\op{vd}}

\newcommand\im{\op{im}}
\newcommand\id{\op{id}}

\newcommand\Hom{\op{Hom}}

\newcommand\Pic{\op{Pic}}

\newcommand\Spec{\op{Spec}}
\newcommand\Hilb{\op{Hilb}}
\newcommand\cone{\op{Cone}}
\newcommand\rspan{\op{span}}
\newcommand\JK{\op{JK}}
\newcommand\TJK{\op{TJK}}
\newcommand{\CT}{\mathrm{CT}}

\newcommand\Sym{\op{Sym}}

\renewcommand{\div}{/\!\!/}

\newcommand{\Faces}{\mathrm{Faces}(\Sigma)}

\newcommand{\sm}{\;\mathrm{sm}}
\newcommand\Perf{\op{Perf}}
\newcommand\coh{\op{coh}}

\newcommand{\dT}{{\sT_{\!\mathrm{diag}}}}
\newcommand{\wsigma}{\widetilde\sT_{\!\sigma^\perp}}
\newcommand{\wSigma}{\widetilde\sT_{\!\Sigma^\perp}}
\newcommand{\dsigma}{\widetilde\sT_{\!\sigma^\perp}^{\;\diag}}
\newcommand{\sSigma}{\mathsf{\Sigma_\C}}
\newcommand{\diag}{{\mathrm{diag}}}

\newcommand{\Xaff}{\Spec H^0(X,\cO_X)}

\newcommand\commentOn{
	\newcommand{\commentO}[1]{\noindent\textcolor{blue}{\texttt{RO:}##1}}
    \newcommand{\commentT}[1]{\noindent\textcolor{red}{\texttt{RT:}##1}}}

\newcommand\beq[1]{\begin{equation}\label{#1}}
\newcommand\eeq{\end{equation}}
\newcommand\beqa{\begin{eqnarray*}}
\newcommand\eeqa{\end{eqnarray*}}

\newcommand\arXiv[1]{\href{http://arxiv.org/abs/#1}{arXiv:#1}}
\newcommand\mathAG[1]{\href{http://arxiv.org/abs/math/#1}{math.AG/#1}}
\newcommand\mathSG[1]{\href{http://arxiv.org/abs/math/#1}{math.SG/#1}}

\newtheorem{theorem}{Theorem}[section]

\theoremstyle{definition}
\newtheorem{definition}[theorem]{Definition}
\newtheorem{remark}[theorem]{Remark}
\newtheorem{remarks}[theorem]{Remarks}
\newtheorem{example}[theorem]{Example}

\newtheorem{lemma}[theorem]{Lemma}
\newtheorem{corollary}[theorem]{Corollary}
\newtheorem{proposition}[theorem]{Proposition}

\title{Virtual Jeffrey--Kirwan localisation}
\author{Riccardo Ontani}
\author{Richard P. Thomas}

\begin{document}
\commentOn
\begin{abstract} We express integrals over virtual cycles of GIT quotients $X\div\sG$ in terms of integrals over virtual cycles of fixed loci $X^\sT$.  The results hold for both perfect obstruction theories and $(-2)$-shifted symplectic structures, in cohomology and in $K$-theory, and for noncompact Deligne-Mumford stacks acted on by a reductive group with compact quotient.
\end{abstract}
\maketitle

\section{Introduction}

Let $\sG \curvearrowright(X,L)$ be the linearised action of a reductive group on a polarised smooth projective variety. Let $\sT\subset \sG$ be a maximal torus.

\medskip

\noindent\textbf{ABBV localisation.} The Atiyah--Bott--Berline--Vergne localisation formula \cite{AtiyahBott,BerlineVergne} expresses integrals over $X$ via integrals over the components $F$ of the fixed locus $X^\sT$\!,
$$
\int_X\alpha\=\sum_{F \subseteq X^\sT} \int_F \frac{\alpha|_F}{e^\sT(N_{F/X})}\ \text{ for }\,\alpha\,\in\,H^*_\sT(X).
$$

\noindent\textbf{JK localisation.}
By contrast the Jeffrey--Kirwan localisation formula \cite{JeffreyKirwan} expresses integrals over the \emph{GIT quotient} $(X\div\sG,L_0)$ of $(X,L)$ in terms of integrals over the fixed locus $X^\sT$. For instance in the abelian case $\sG=\sT$ with no strictly semistable points it gives
$$
\int_{X\div\;\sT}\alpha\_0\=\sum_{F \subseteq X^\sT}\JK^{\eta}_{\mu(F)+\eps}\int_F \frac{\alpha|_F}{e^\sT(N_{F/X})}\ \text{ for }\,\alpha\,\in\,H^*_\sT(X).
$$
Here $\alpha\_0$ is the image of $\alpha|_{X^{\ss}}\in H^*_\sT(X^{\ss})\cong H^*(X^{\ss}/\;\sT)=H^*(X\div\;\sT)$ and $\mu:X\to\ft^\vee$ is the moment map. The residue operation $\JK$ will be explained in Section \ref{section:JK_res}; it depends on a choice of $\eta\in\ft$, and only $F$ with $\langle\mu(F),\eta\rangle>0$ contribute to the formula. (The perturbation $\epsilon$ is explained in Section \ref{section:JK_loc}; it is not necessary if the points $\mu(F)\in\ft^\vee$ avoid a certain hyperplane arrangement.) There are generalisations to nonabelian $\sG$, noncompact $X$ with compact $X\div\sG$ and a perhaps more familiar version with $\alpha$ replaced by $\alpha\cdot\exp(c_1^\sT(L))$. The result is particularly dramatic when applied to a vector space $V$, computing integrals over $V\div\sG$ from a single contribution at the origin $\{0\}=V^\sT\subset V$.\medskip

\noindent\textbf{Virtual ABBV localisation.} There are two generalisations of ABBV localisation to projective schemes or Deligne-Mumford stacks with $\sT$-equivariant obstruction theory $\EE\to\LL_X$.
Graber--Pandharipande's \cite{GraberPandharipande} applies when $\EE$ is \emph{perfect} and 2-term \cite{BehrendFantechi},
$$
\int_{[X]^{\vir}}\alpha\=\sum_{F \subseteq X^\sT} \int_{[F]^{\vir}}\frac{\alpha|_F}{e^\sT(N^{\vir}_{F/X})}\,.
$$
There is a similar formula \cite{OhThomas} for the virtual cycles of \cite{BorisovJoyce} when $\EE$ is \emph{symmetric} 3-term\footnote{In fact we need a slight derived enhancement of $(X,\EE)$ called a $(-2)$-shifted symplectic structure.} and \emph{oriented} \cite{CGJ}, using the Edidin-Graham class $\sqrt e$ of \cite[Equation 21]{OhThomas}:
$$
\int_{[X]^{\vir}}\alpha\=\sum_{F \subseteq X^\sT} \int_{[F]^{\vir}}\frac{\alpha|_F}{\sqrt e^\sT\!\(N^{\vir}_{F/X}\)}\,.
$$

\noindent\textbf{Virtual JK localisation.} 
In this paper we complete the square by proving generalisations of JK localisation to Deligne-Mumford stacks $X$ with $\sG$-equivariant obstruction theory $\EE\to\LL_X$ which is either (\emph{i}) perfect and 2-term \cite{BehrendFantechi} or (\emph{ii}) symmetric 3-term and oriented \cite{OhThomas}. We show that the GIT quotient $X\div\sG$ inherits an obstruction theory of the same type if there are no strictly semistable points. It therefore has a virtual cycle $[X\div\sG]^{\mathrm{vir}}$, integrals over which we express in terms of integrals over $[X^\sT]^{\mathrm{vir}}$.
For abelian $\sG=\sT$ and the virtual cycles of \cite{BehrendFantechi, LiTian} the formula is
\beq{vJK1}\tag{JK1}
\int_{[X\div\;\sT]^{\vir}}\alpha\_0\=\sum_{F \subseteq X^\sT}\JK^{\eta}_{\mu(F)+\eps}\int_{[F]^{\vir}}\frac{\alpha|_F}{e^\sT(N^{\vir}_{F/X})}\,.
\eeq
For the virtual cycles of \cite{BorisovJoyce, OhThomas} we require that $X$'s symmetric obstruction theory comes from a $\sG$-equivariant oriented $(-2)$-shifted symplectic derived structure. Then
\beq{vJK2}\tag{JK2}
\int_{[X\div\;\sT]^{\vir}}\alpha\_0\=\sum_{F \subseteq X^\sT}\JK^{\eta}_{\mu(F)+\eps}\int_{[F]^{\vir}}\frac{\alpha|_F}{\sqrt e^\sT\!\(N^{\vir}_{F/X}\)}\,.
\eeq
We may replace $\alpha$ by $\alpha\;e^{c_1^\sT(L)}$ in both formulae.  In Theorems \ref{th:abelian_JK} and \ref{th:abelian_JK_general} below we generalise \eqref{vJK1} to noncompact $X$ with compact $X\div\,\sT$, and in Theorem \ref{th:JK_cy4} we do the same for \eqref{vJK2}. \medskip

\noindent\textbf{Nonabelian quotients.} 
In Section \ref{section:NA_JK_loc} we adapt the algebraic version \cite{Maddock} of a method of Martin \cite{Martin} to deduce results for nonabelian $\sG$ directly from the abelian formulae \eqref{vJK1},\,\eqref{vJK2}. The most general result is Theorem \ref{th:JK},
$$
            \int_{[X\div \sG]^{\vir}} \alpha\_0 \= \frac{1}{\vert W \vert}\sum_{\substack{F\subseteq X^\sT\\\
            \langle \mu(F), \eta \rangle >0}} \JK^{\eta}_{\mu(F)+\eps}\left(e^\sT(\fg_\C/\ft_\C) \cdot\int_{[F]^{\vir}} \frac{\alpha^\sT\vert_{F}}{e^\sT(N^{\vir}_{F/X})}\right)
$$
for the virtual cycles of \cite{BehrendFantechi, LiTian}. For the virtual cycles of \cite{BorisovJoyce, OhThomas}, Theorem \ref{th:JK_cy4} proves the same formula with $e^\sT$ replaced by $\surd e^\sT$ in the denominator only. Here $W$ is the Weyl group of $\sG$ and $\fg_\C/\ft_\C$ carries the adjoint representation of $\sT$ so $e^\sT(\fg_\C/\ft_\C)$ is the product $\prod\epsilon_i$ of the roots $\eps_i\in\ft^\vee_\Z$ of $\sT\subset\sG$. Replacing $\alpha$ by $\alpha\cdot\exp(c_1^\sT(L))$ gives a formula which is more familiar in some settings.
\medskip

\noindent\textbf{Virtual $K$-theoretic JK localisation.} In Section \ref{KKJK}  we prove a number of $K$-theoretic JK localisation formulae under similar conditions, such as 
$$
            \chi\(\cO^{\vir}_{X\div\sG}\otimes V_0\) \= \frac1{|W|}\sum_{\substack{F\subseteq X^\sT\\ \langle\mu(F), \eta \rangle>0}}\TJK^\eta_{\mu(F)+\eps} \left(\Lambda\udot\op{(\fg_\C/\ft_\C)}^\vee \cdot\,\chi\_\sT\!\left(\!F,\,\frac{V\otimes\cO^{\vir}_{X^\sT}\big|_F}{\Lambda\udot\(N^{\vir}_{F/X}\)^\vee}\right)\right)
$$
for $V\in K^0_\sG(X)$ and $X$ with a $\sG$-invariant perfect obstruction theory. Here $V_0$ is the class $V|_{X^{\ss}}\in K^0_\sG(X^{\ss})\cong K^0(X^{\ss}/\sG)\cong K^0(X\div\sG)$ while $\TJK$ is the toric JK residue of Definition \ref{def:TJK} and $\Lambda\udot(\fg_\C/\ft_\C)^\vee=\prod\nolimits_{i}(1-t^{\eps_i})$ with $\epsilon_i \in \ft_\Z^\vee$ the roots of $\sG$. For $X$ with an oriented $(-2)$-shifted symplectic structure, we use the twisted virtual structure sheaves $\widehat\cO^{\vir}$ of \cite[Definition 5.9]{OhThomas} and replace the $K$-theoretic Euler class $\Lambda\udot(N^{\vir}_{F/X})^\vee$\vspace{-1mm} used above by the square-rooted version $\sqrt{\mathfrak e^\sT}(N^{\vir}_{X/F})$ of \cite[Section 7]{OhThomas}. The result is
$$
            \chi\(\widehat\cO^{\vir}_{X\div\sG}\otimes V_0\) \= \frac1{|W|}\sum_{\substack{F\subseteq X^\sT\\ \langle\mu(F), \eta \rangle>0}}\TJK^\eta_{\mu(F)+\eps} \left(\Lambda\udot(\fg_\C/\ft_\C)^\vee \cdot\,\chi\_\sT\!\left(\!F,\,\frac{V\otimes\widehat\cO^{\vir}_{X^\sT}\big|_F}{\sqrt{\mathfrak e^\sT}\(N^{\vir}_{X/F}\)}\right)\right).
$$
Finally in Section \ref{C*sec} we prove versions of these results when $X\div\sG$ is noncompact but acted on by a further $\C^*$ action.\medskip

\noindent\textbf{Method.} The standard proofs of (cohomological) JK localisation use the Laplace transform for the moment map,\footnote{For symplectic manifolds, abelian JK localisation is just Fubini's theorem $\int_X=\int_{\ft^\vee}\circ\,\mu_*$ for the moment map $\mu:X\to\ft^\vee$, combined with ABBV localisation to evaluate $\int_X$ and the (inverse) Laplace transform to invert $\int_{\ft^\vee}$. In our singular algebro-geometric setting, with no access to $\mu$, we have to work much harder. But we do have $\im\mu$\,---\,the moment polytope\,---\,and we make heavy use of it.} which we do not have access to in algebraic geometry. We do, however, have access to its \emph{image}, the moment polytope, and this will allow us to generalise the different proofs of Lerman \cite{Lerman} and Jeffrey-Kogan \cite{JeffreyKogan}. (By contrast the Fourier series relevant to $K$-theoretic JK localisation are much simpler than Fourier or Laplace transforms, making the proofs in Section \ref{KKJK} much easier.) 
\medskip

Lerman handled the special case of a hamiltonian $S^1$ action on a symplectic manifold $X$, showing how to deduce the JK localisation formula from the ABBV localisation formula applied to the \emph{symplectic cut} of $X$.\medskip

Denoting the moment map by $\mu\,:\,X\to\R$, the symplectic cut of $X$ is formed by taking $\mu^{-1}[0,\infty)$ and dividing the boundary $\mu^{-1}(0)$ by the $S^1$ action. Thus it is the open set $\{\mu>0\}$ of $X$ compactified with a copy of the symplectic reduction $\mu^{-1}(0)/S^1$\,---\,which is the GIT quotient $X\div\C^*$ in the algebraic case.

The $S^1$ action on $X$ induces one on the cut with fixed loci of two types: \begin{enumerate}
\item fixed loci $F$ of $X$ with $\mu(F)>0$, and
\item a new fixed locus $\mu^{-1}(0)/S^1$ over $\mu=0$.
\end{enumerate}
Applying the ABBV localisation formula to an equivariant form of degree $\dim X-2$ expresses its integral (zero!) as a sum of contributions at (1) and one at (2). Taking their residues give the right (respectively left) hand side of the JK formula, thus proving it for $S^1$ actions.\medskip

Jeffrey and Kogan showed how to extend this geometric method\footnote{For a different approach see \cite{GK}.} to the action of a higher rank torus $\sT_{\!\R}$ with moment map $\mu\,:\,X\to\ft_\R^\vee$. Here the analogue of $[0,\infty)$ is a choice of simplicial cone\footnote{The same cone used to define the Jeffrey-Kirwan residue operator appearing in the localisation formula.}  $\Sigma\subset\ft_\R^\vee$, from which we define the cut by quotienting the boundary of $\mu^{-1}(\Sigma)$ by subtori of $\sT_{\!\R}$. (Over the interior $\mathring\sigma$ of the stratum $\sigma\subset\partial\Sigma$ we divide $\mu^{-1}(\mathring\sigma)$ by the subtorus $(\sT_{\!\sigma^\perp})\_\R\subseteq\sT_{\!\R}$ with Lie algebra $\sigma^\perp\subseteq\ft_\R$.) Again this carries a $\sT_{\!\R}$ action and we integrate an appropriate equivariant class to get a function on $\ft_\R$. Again a certain residue (or combination of residues) of this function gives zero. Evaluating this by the ABBV formula instead, we get contributions from
\begin{enumerate}
\item those fixed loci $F$ of $X$ with $\mu(F)\in\mathring\Sigma$, 
\item the symplectic reduction $\mu^{-1}(0)/\;\sT_{\!\R}$ over $\mu=0\in\Sigma$, and 
\item some additional fixed loci created by the cutting process.
\end{enumerate}
For $\Sigma$ sufficiently wide, Jeffrey-Kogan show the loci (3) contribute zero to the residue of the total integral, thus recovering the JK formula for higher rank tori.\medskip

Edidin-Graham put Lerman's method into algebraic geometry in \cite{EG_alg_cut}. This allows $X$ to have singularities (which is important for us since our $X$ or $X\div\sG$ will usually be a moduli space with a virtual cycle). So we begin by generalising \cite{EG_alg_cut} to the higher rank Jeffrey-Kogan setting by defining, in Section \ref{section:alg_cut}, an algebraic cut for the action of an algebraic torus $\sT$ on a quasi-projective Deligne-Mumford stack $X$. The possible noncompactness and reducibility of $X$\,---\,meaning the ``moment polytope'' is in fact a union of polyhedra, possibly intersecting or even overlapping\,---\,and the finite stabilisers makes it much more complicated than we would have liked.\medskip

In Section \ref{BFsec} we show how a perfect obstruction theory on $X$ induces one on the cut. What is much harder is to prove the analogous results for $(-2)$-shifted symplectic structures. This uses derived algebraic geometry and the shifted symplectic reduction of \cite{calaque2015lagrangian, safronov2016quasi,Park}, as we explain in Section \ref{sec:derived_symp_red}.\medskip

Thus we can apply virtual ABBV localisation on these algebraic cuts to try to deduce the abelian virtual JK formula. For perfect obstruction theories and the virtual cycles of \cite{BehrendFantechi, LiTian} we apply Graber-Pandharipande's virtual ABBV formula \cite{GraberPandharipande} in Section \ref{section:loc_cut}. For $(-2)$-shifted symplectic structures and the virtual cycles of \cite{BorisovJoyce, OhThomas} we apply the virtual ABBV formula of \cite{OhThomas} in Section \ref{sec:derived_symp_red}.\medskip

The hardest thing is to show\,---\,for sufficiently wide cone $\Sigma$\,---\,the vanishing of the contributions (3) in our singular, virtual setting. This involves clarifying the (complicated!) method of Jeffrey-Kogan and extending it to more general weights. For smooth manifolds the denominators $e^\sT(N_{F/X})$ in the localisation formula, considered as $H^*(X)$-valued functions on $\ft$, are products of powers of the linear functions on $\ft$ given by the weights of the $\sT$ action on $N_{F/X}$. Their negatives lie along rays of the moment polytope $\mu(X)$ emanating from the vertex $\mu(F)$. (Or, if $\mu(F)$ lies in the interior of $\mu(X)$, they span all the directions in $\mu(X)$.) When $N_{F/X}$ is replaced by $N_{F/X}^{\vir}$ there could be more weights, which we control in Proposition \ref{pro:weights_normal_bundle}. This allows us to generalise Jeffrey-Kogan's vanishing of the contributions (3) to the virtual setting in Section \ref{section:JK_loc}.\medskip

Finally Martin \cite{Martin} showed how to use the rational fibration $\xymatrix@1{X\div\;\sT\, \ar@{-->}[r]& \,X\div\sG}$ to deduce the nonabelian JK formula from the abelian one. He made crucial use of the moment map, but  Maddock \cite{Maddock} found an algebro-geometric argument which we show how to generalise to noncompact and virtual settings in Section \ref{section:NA_JK_loc}.\medskip

So our results all fall squarely inside the convex hull of what was already known.
The main contributions of this paper are in
\begin{itemize}
\item finding algebraic proofs for results in symplectic geometry \emph{without using the moment map},
\item generalising these proofs to singular, reducible, non-reduced and noncompact varieties (with compact quotient), all of which are important to the applications in moduli theory and all of which add significant technical difficulty,
\item generalising to the two virtual settings: perfect obstruction theories and $(-2)$-shifted symplectic structures, and
\item combining ideas and proofs from disparate areas that few people know all of.
\end{itemize}
Along the way
we clarify, simplify and clean up a number of those results which had been written in conflicting and confusing ways\,---\,in a field shared between symplectic geometers, algebraic geometers and physicists some mess seems inevitable.

\subsection*{Acknowledgements} We thank Frances Kirwan, Johan Martens and Jacopo Stoppa for introducing us to\,---\,and for useful conversations about\,---\,cohomological JK localisation. Thanks to Henry Liu and \cite[Appendix A]{AganagicOkounkov} for teaching us the $K$-theoretic analogue. We are most grateful to Hyeonjun Park for generous assistance with the derived algebraic geometry of Section \ref{sec:derived_symp_red}.

There are now a number of sophisticated wall crossing formulae for virtual cycles and virtual structure sheaves such as \cite{bojko2025,halpernL2025,joyce2021,karpovmoreira,kuhnliuthimm}. 
The GIT setting considered here is less general\footnote{The abelian JK formula is a wall crossing formula for the variation of GIT problem with $(X,L)\div\,\sT$ at one end and the empty set $\emptyset=(X,L\otimes\Psi)\div\,\sT$ at the other. (Here $\Psi$ is a character chosen so that the moment polyhedron $\Delta(X,L)\subset\ft^\vee_\Q$ of Section \ref{mpoly}, translated by $\Psi$, misses the origin $0\in\ft^\vee_\Q$.) So the difference of the two JK formulae for two different linearisations gives the general VGIT wall crossing formula.} but more concrete, lending itself to easier explicit calculation.  And the extraordinary number of applications of JK localisation in the physics literature, especially since \cite{benini2014elliptic}, suggest the final result should be a useful tool in enumerative algebraic geometry. 

\tableofcontents

\section*{Notation}
\begin{enumerate}
\makeatletter
    \let\c@enumi\c@equation 
    \makeatother
    \renewcommand{\theenumi}{(\theequation)} 
    \renewcommand{\labelenumi}{\theenumi}
\item\label{weak} Throughout $X$ will always be a projective-over-affine Deligne-Mumford stack over $\C$ whose coarse moduli space $|X|$ has finitely generated ring of global functions $H^0\(\cO_{|X|}\)$. For most of the paper (Sections 1-8 and the Appendices) we use the terms (quasi-)projective in the weak sense that $|X|$ is (quasi-)projective. So our conditions are equivalent to $|X|$ admitting a closed embedding in some $\PP^m\times\AA^n$. In Sections 9-10 we have to assume $X$ is quasi-projective in the strong sense of \cite{Kr2}.
\item\label{irredcpts} We do \emph{not} assume $X$ is irreducible or reduced, because in applications a quotient of $X$ will be a moduli space. We consider an \emph{irreducible component} of a scheme (or Deligne-Mumford stack) to carry its maximalist scheme structure, incorporating all embedded components supported on it. Formally, the component is cut out of $X$ by the intersection of all primary ideals\,---\,taken from a primary decomposition of the zero ideal\,---\,whose associated primes contain the minimal prime corresponding to the component. Thus $X$ is the union of its irreducible components.
\item $(X,L)$ denotes a Deligne-Mumford stack $X$ polarised by an ample line bundle $L$. This could be pulled back from $|X|$; we do \emph{not} require the local orbi-ampleness condition that the stabiliser groups of points $x$ act faithfully on the line $L|_x$.
\item $\sG$ will always denote a reductive algebraic group with maximal torus $\sT$ and Weyl group $W$.
\item $\sG\acts(X,L)$ denotes a (left) action of $\sG$ on $X$ plus the choice of a linearisation of the action on $L$. This defines the open subset of semistable points
$$
X^{L-\ss}\=X^{\ss}\ :=\ \bigcup\nolimits_{n \geq 0}\  \bigcup\nolimits_{s \in H^0(X,L^n)^\sG} \ \big\{x\in X\ :\ s(x) \neq 0\big\}.
$$
We will always assume that stability\,=\,semistability for $\sG\acts(X,L)$, so that all points of $X^{\ss}$ are stable and hence have finite stabilisers. Therefore the stack quotient of $X^{\ss}$ by $\sG$ is a Deligne-Mumford stack, the \emph{GIT quotient stack}
$$
X\div\sG\ :=\ X^{\ss}/\sG.
$$
The underlying GIT quotient \emph{scheme} is the coarse moduli space of $X\div\sG$,
$$
|X\div\sG|\=\mathrm{Proj}\left(\bigoplus\nolimits_{n\geq 0} H^0(X,L^{n})^\sG\right)
$$
because $X\to|X\div\sG|$ is a \emph{geometric quotient}.
\item For $\sG=\sT$ an algebraic torus, the derivative of the action defines a map $\Omega_X\otimes\ft\to\cO_X$ whose image is an ideal $\cI_{X^\sT}\subset\cO_X$ defining the fixed locus $X^\sT\subseteq X$. It is the largest substack on which some \emph{cover} of $\sT$ acts trivially \cite[Theorem 1.1\;(2)]{AJ}.
\item All (Chow) (co)homology groups are taken with \emph{rational coefficients}. For the equivariant group $A_*^\sG(X)$ we use the Chow homology of the Artin stack $A_*(X/\sG)$ \cite{Kr1}.
\item If $\sG\acts X$ has finite stabilisers and $V$ is a $\sG$-representation we denote the associated bundle on the quotient by $X \times_\sG V \to X/\sG$. Its total space is the quotient of $X\times V$ by the diagonal action of $\sG$. If $V$ is a $(\sG \times \sH)$-representation and $(\sG\times\sH)\acts X$ then $X \times_\sG V$ is $\sH$-equivariant.
\item We use $\bigsqcup X_i$ for set-theoretic disjoint union (such as a stratification) and $\coprod X_i$ for a topological disjoint union (i.e. the closures $\overline X_i=X_i$ are disjoint).
\item $\sT$ is a complex algebraic torus of rank $r$ with character and cocharacter lattices
$$
\ft^\vee_\Z\ :=\ \Hom(\sT, \C^\ast) \qquad \text{and} \qquad \ft_\Z := \Hom(\C^\ast, \sT).
$$
For any lattice $\Gamma_\Z$ and field $k$ we write $\Gamma_k := \Gamma_\Z \otimes_\Z k$. So $\ft_\C$ is the Lie algebra of $\sT$ and $\ft_\R$ is the Lie algebra of its maximal compact subgroup $\sT_{\!\R}\subset\sT$.
\item\label{psipsi} We use the same notation for $\psi\in\ft^\vee_\Z$ and the associated character $\psi:\sT\to\C^*$, but on composing with $\C^*\into\C$ to think of it as a \emph{function} we denote it by $t^\psi:\sT\to\C$.
\item Given a convex subset $\tau \subseteq \ft^\vee_\Q$ we denote by $\sT_{\!\tau^\perp} \subseteq \sT$ the subtorus annihilated by the tangent space to $\tau$,
$$
\sT_{\!\tau^\perp}\ :=\ \bigcap\nolimits_{\psi\in\langle\tau-p\rangle\_\Q\,\cap\,\ft^\vee_\Z\,}\ker\!\(\psi:\sT\rightarrow \C^\ast\),
$$
where we intersect over all the integral elements $\psi$ of the $\Q$-linear span of the translation (by any element $p\in\tau$) of $\tau$ to the origin. Although this conflicts with the notation $\sT_x\subset\sT$ for the stabiliser group of a point $x$ of a $\sT$-space $X$, it will always be clear which is meant.
\item\label{kone} Given rational characters $\psi_1, \dots, \psi_k\in\ft^\vee_\Q$\,---\,which without loss of generality we will always assume to be \emph{primitive} and \emph{integral}\,---\,we let
$$
\cone(\psi_1, \dots, \psi_k)\ :=\ \Big\{\sum a_i\psi_i\ :\ a_i\in\Q\_{\ge0}\Big\}\,\subset\,\ft^\vee_\Q\,.
$$
\item\label{Tperptilde}  If the primitive integral $\psi_i$ in \ref{kone} are linearly independent in $\ft^\vee_\Q$  the cone is called \emph{simplicial}. Calling it $\sigma$ we get an extension $\widetilde\sT_{\!\sigma^\perp}$ of $\sT_{\!\sigma^\perp}$ by a finite group, given by
$$
\widetilde\sT_{\!\sigma^\perp}\ :=\ \bigcap\nolimits_{i=1}^k\ker\!\(\psi_i:\sT\rightarrow \C^\ast\)\ \subseteq\ \sT.
$$
\item\label{sigma_rep} In the setting of \ref{kone},\,\ref{Tperptilde} we use $\Psi_i\cong\C$ to denote the 1-dimensional $\sT$ representation induced by $\psi_i$ and $\sigma\_\C$ for the $\C$-linear span of the $\psi_i$,
$$
\sigma\_\C\ :=\ \bigoplus\nolimits_{i=1}^k \Psi_i.
$$
Let $\{z_i\}$ be the dual basis of linear functions on $
\sigma\_\C$. Then we denote by
$$
\sigma_\C^\circ\ :=\ \big\{z_i\ne0\big\}\ \subseteq\ \sigma\_\C
$$
the open subset where the torus $\sT/\,\wsigma$ acts with finite stabilisers, and
$$
    \Bz^\theta\ :=\ \prod\nolimits_{i=1}^k z_i^{\theta_i}\ \in\ \C[z_1, \dots, z_k],
$$
a monomial function on $\sigma\_\C$ of $\sT$-weight $-\theta$, where $\theta=(\theta_1,\dots,\theta_k)$.
\item Given a representation $V$ of $\sT$ we denote its $\psi$-weight space by $V_\psi\subseteq V$ for any character $\psi\in\ft_\Z^\vee$.
Equivalently $V_\psi=(V\otimes\Psi^{-1})^\sT\otimes\Psi$.
\item Given $\sT\acts(X,L)$ and a polyhedron $\sigma\subset\ft_\Q^\vee$ we define
$$
X^{\sigma-\ss}\ :=\ \bigcup\nolimits_{\phi \in \sigma}X^{L\otimes\Phi^{-1}-\ss}\=\big\{x\in X:s(x)\ne0\text{ for some }s\in H^0(L^k)_{k\psi\,\in\,k\sigma}\big\},
\hspace{-1cm}
$$
using only the powers $L^k\otimes\Phi^{-k}$ with $k\phi\in\ft_\Z^\vee$ to define the locus of semistable points.
\item For the moment ``polyhedron" $\Delta=\mu(X)$ of $\sT\acts(X,L)$ and the special substacks $\mu^{-1}(\tau_p)\subset X$ see Definitions \ref{def:moment_polyhedron} and \ref{def:mu_tau} in the next Section.
\end{enumerate}

\section{The moment polyhedron}\label{mpoly}

Let $(X,L)$ be a polarised projective-over-affine Deligne-Mumford stack over $\C$ acted on by an algebraic torus $\sT$.


\begin{definition}\label{def:moment_polyhedron}
    The \textit{moment polyhedron} $\Delta=\Delta^\sT(X,L)\subset\ft_\Q^\vee$ of $(X,L)$ is
$$
    \mu(X)\ :=\ \Delta\ :=\ \Big\{\psi\in\ft_\Q^\vee\ :\ \exists\,k>0\text{ such that }H^0\big(X_{\text{red}},L^k|_{X_{\text{red}}}\big)_{k\psi}\ne0\Big\}.
$$
\end{definition}
\noindent Equivalently, recalling that a closed point $x\in X$ is $\sT$-semistable in $(X,L)$ if there exists $s\in H^0(X,L^k)^\sT$ with $s(x)\ne0$,
\beq{Deltass}
\Delta\,=\,\big\{\psi\in\ft_\Q^\vee\,:\,\text{the semistable locus } X^{L\otimes\Psi^{-1}-\ss} \text{ is nonempty}\big\}.\hspace{-2mm}
\eeq

\begin{remarks}\label{rem:polyhedron}
In Appendix \ref{appendix} we prove various facts about $\Delta$,  for instance  that it is indeed a polyhedron if $X$ is irreducible. (See Lemma \ref{lem:polyhedron_structure}; if $X$ is also projective then $\Delta$ is a polytope, denoted by $P_{\;\sT}(X,L)$ in \cite{Brion}.) In general $\Delta$ is the union of the moment polyhedra of its irreducible components, but we abuse notation by calling it the moment polyhedron anyway.

While the moment map $\mu\colon X\to\ft_\R^\vee$ of symplectic geometry is not available in algebraic geometry, its image $\mu(X)\subset\ft_\R^\vee$ \emph{is} and its rational points are precisely $\Delta\subset\ft_\Q^\vee$. For this reason we will often denote $\Delta$ by $\mu(X)$ or $\mu^\sT(X,L)$. This is compatible with passing to $\sT$-invariant substacks $Y\subset X$, with $\mu^\sT(Y,L|_Y)\subset\mu^\sT(X,L)$.  In particular if $x \in F\subseteq X^\sT$ is an irreducible component of the fixed locus then
$$
    \mu(x)\=\mu(F)\ =\  \mathrm{wt}\_\sT(L|_x)\ =\ c_1^\sT(L|_x)
$$
is a single point in $\ft_\Q^\vee$ (or $\ft_\Z^\vee$ if $X$ is a scheme) which appears in the localisation formulae.
Note these points can be in the interior of $\mu(X)$\,---\,see Example \ref{example:P1P1} for the diagonal action of $\C^*$ on $\PP^1\times\PP^1$.
But if $X$ is irreducible and projective then $\Delta$ is the convex hull of the points $\mu(X^\sT)$ by \cite[Theorem 1;(iii)]{Brion}.

More generally if there is a $\C^\ast \subset \sT$ whose action on $H^0(X,\cO_X)$ has weights all nonzero of the same sign then
each vertex of $\Delta$ is of the form $\mu(F)$ for a \emph{unique} irreducible component $F$ of $X^\sT$. This is a combination of Lemma \ref{lem:polyhedron_structure} and \cite[Theorem 1\,(ii)]{Brion}.

The moment polyhedron of an irreducible $X$ has the obvious scaling property $\Delta(X,L^d)=d\Delta(X,L)$. If $L^d$ descends to the coarse moduli space $|X|$ then these both equal $\Delta(|X|,L^d)$. It quickly follows that we may take $k\gg0$ in the original definition of $\Delta$.

\end{remarks}

\subsection{Stratification}
We stratify $\Delta$ by the dimensions of stabiliser groups.
\begin{figure}[h]
    \centering
    \includegraphics{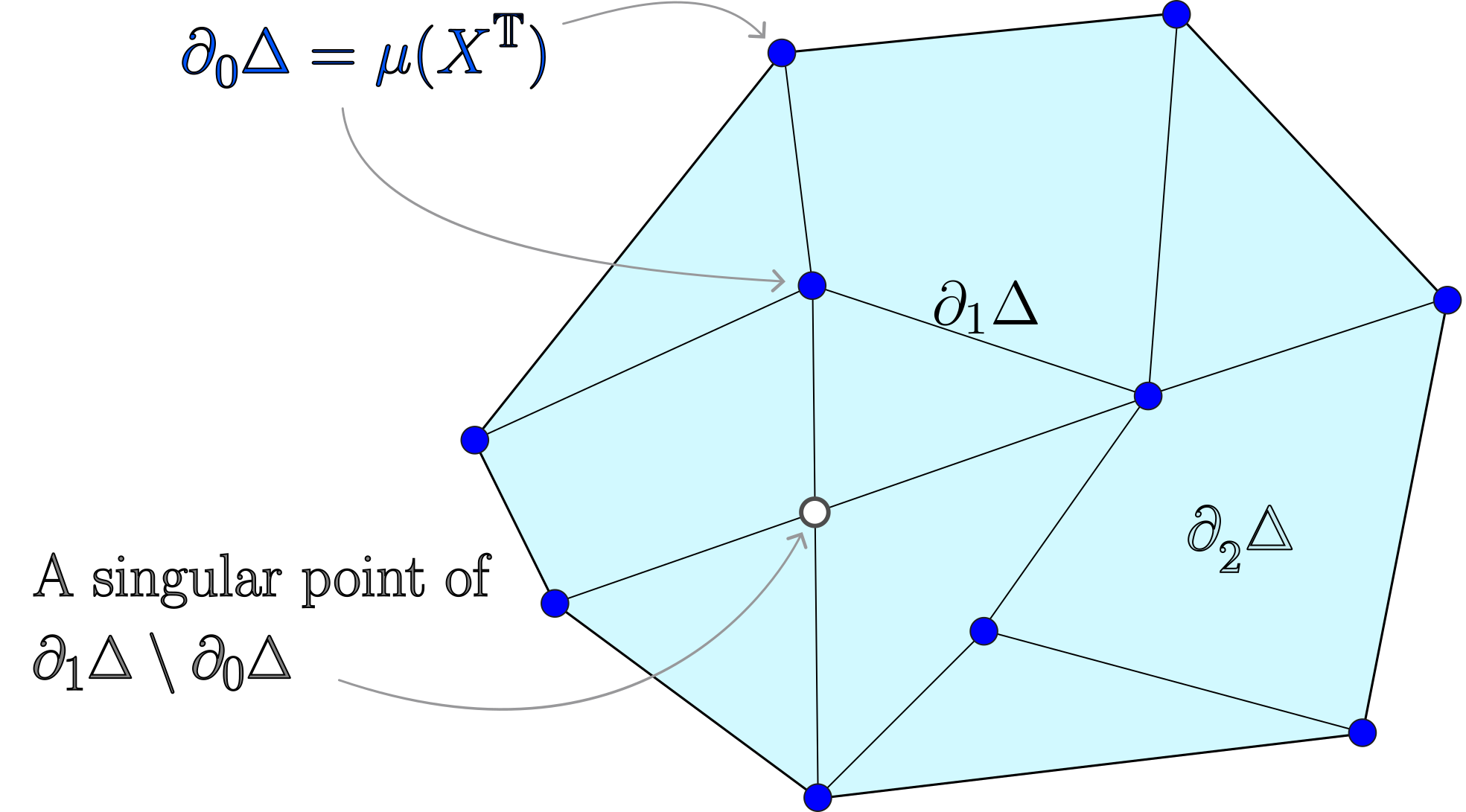}
    \caption{Stratification of a 2-dimensional moment polytope of an irreducible $X$s.}
    \label{fig:polytope_strata}
\end{figure} 
\begin{definition}
    For $0 \leq d \leq \dim(\sT)$ the \textit{dimension-$d$ stratum of $\Delta$} is the closed subset
\begin{align*}
    \partial_{d} \Delta\,:=\,\big \lbrace \psi \in \Delta\,:\,\exists\,x \in X^{L\otimes\Psi^{-1}-\ss} \text{ such that dim}(\sT_x) \geq r-d \big \rbrace\ = \bigcup_{\dim(\sS)=r-d} \mu\big(X^\sS\big).
\end{align*}
The union over all subtori $\sS \hookrightarrow \sT$ of codimension $d$ is finite by Lemma \ref{lem:codimension}. This $\partial_{d} \Delta$ is really a finite union of dimension $\le d$ polyhedra of $\ft^\vee_\Q$ because each of the moment polyhedra $\mu(X^\sS_i)$ of the irreducible components of $X^\sS$ is contained in a translate of $\sS^\perp \subset \ft^\vee_\Q$ by Lemma \ref{lem:trivial_act_from_polyhedron}. Their topological boundaries are all contained in $\partial_{d-1}\Delta$ by Lemma \ref{lem:interior_polytope}.
\end{definition}

Note that these polyhedra can have varying dimension $\le d$, with no constraint on the way they intersect each other: they can be disjoint, overlapping, nested or even equal. 

\begin{example}\label{example:P1P1}
 Let $t\in\C^\ast$ act on $\(\PP^1 \times \PP^1, \cO(1,1)\)$ by $(tx_0,x_1),\,(y_0, t^{-1}y_1)$. The moment polyhedron $\Delta = [-1,1]$ has zero-dimensional stratum $\partial_0\Delta = \lbrace -1, 0,1 \rbrace$ with the interior point $0 = \mu(0 \times \infty) = \mu(\infty \times 0)$ being the moment polyhedron of \emph{two distinct} fixed points.
\end{example}

Thus not all pairwise intersections of the polyhedra comprising $\partial_d\Delta$ lie in $\partial_{d-1}\Delta$ so\,---\,as illustrated in Figure \ref{fig:polytope_strata}\,---\,$\partial_d\Delta\setminus\partial_{d-1}\Delta$ need not be smooth. Removing all the intersections which cause a singularity\footnote{I.e. we remove any intersection which is not an \emph{overlap} where the two components are locally isomorphic.} gives the \emph{smooth relative interior}
$$
\partial^{\sm}_d\Delta\ \subset\ \partial_d\Delta \setminus \partial_{d-1}\Delta.
$$
By Lemma \ref{lem:codimension} $\partial^{\sm}_d\Delta$ is \textit{smooth of dimension $d$}\,---\,about any point it is locally an affine space of dimension $d$. We formalise this in a definition.

\begin{definition}\label{def:tau}
    Given $p \in \partial^{\sm}_d\Delta$, we denote by $\tau_p \subseteq \ft^\vee_\Q$ the affine space coinciding with $\partial_d \Delta$ locally around $p$.
\end{definition}

    By Lemma \ref{lem:codimension}, every connected component of the $\sT_{\!\tau_p^\perp}$-fixed locus $X^{\!\sT_{\!\tau_p^\perp}}$ has moment polyhedron contained in a translate of $\tau_p\subseteq\ft^\vee_\Q$. This allows us to generalise the symplectic submanifold $\mu^{-1}(\tau_p)$ from symplectic geometry (though beware the following definition is in general smaller than the set theoretic preimage of the moment map in symplectic geometry which will contain loci not fixed by $\sT_{\!\tau_p^\perp}$ if $\tau_p\cap\Delta$ is not in the boundary of $\Delta$).
 
\begin{definition}\label{def:mu_tau}
        We \textit{define} $\mu^{-1}(\tau_p)$ to be the union of the connected components $Y \subseteq X^{\!\sT_{\!\tau_p^\perp}}$ such that $\mu(Y) \subseteq \tau_p$.
    \end{definition}

Finally we record for later that
\begin{equation}\label{eq:union_polyhedra}
    \partial_d^{\sm}\Delta \text{ is an open dense subset in a finite union of $d$-dimensional polyhedra}
\end{equation}
obtained by removing their boundaries and finitely many polyhedra of smaller dimension. And that taking $d=r=\dim\sT$, Lemma \ref{lem:interior_polytope} shows that $\partial^{\sm}_r\Delta$ is contained in the interior of $\Delta$ and that for all $\psi\in\partial^{\sm}_r\Delta$, all $L\otimes\Psi^{-1}$-semistable points have finite stabilisers, so they are $L\otimes\Psi^{-1}$-stable by \cite[Remark 8.3]{Kirwan}.




\section{The algebraic cut}\label{section:alg_cut}
Let a compact torus $\sT_{\!\R}$ act on a compact symplectic manifold $X$ with moment map $\mu : X \rightarrow \ft^\vee_\R$ such that the origin is in the interior of the moment polytope $\Delta := \mu(X)$. Let $\Sigma$ be the cone in $\ft^\vee_\R$ generated by an integral basis of $\ft_\Z^\vee$. 

The \textit{symplectic cut} of $X$ by $\Sigma$ is defined in \cite{JeffreyKogan}, building on the definition of \cite{Lerman} for $\sT_{\!\R}=S^1$. It is a symplectic $\sT_{\!\R}$-manifold with moment polytope $\mu(X)\cap\Sigma$ that compactifies the open subset $\mu^{-1}(\mathring\Sigma)$ of $X$. It is constructed from $\mu^{-1}(\Sigma)$ by quotienting each stratum $\mu^{-1}(\mathring\sigma),\ \sigma\subset\partial\Sigma$ of its boundary $\mu^{-1}(\partial\Sigma)$ by the action of $(\sT_{\!\sigma^\perp})\_\R\subset \sT_{\!\R}$, the torus annihilated by the tangent space of $\sigma\subset\ft^\vee_\R$.
It admits a $\sT_{\!\R}$-equivariant stratification whose locally closed strata are indexed by faces of $\Sigma$:
\begin{align}\label{eq:symplectic_cut}
    X_\Sigma\= \bigsqcup_{\sigma \in \mathrm{Faces}(\Sigma)}\left(X_\sigma^\circ\,:=\,\frac{\mu^{-1}(\mathring\sigma)}{(\sT_{\!\sigma^\perp})\_\R}\right)\!.
\end{align}
For $(X,L)$ a polarised projective variety and $\sT= \C^\ast$ a version of Lerman's construction was introduced in \cite{EG_alg_cut}. In generalising this to $\sT$ of rank $r$ we find it convenient to define cuts by cones $\Sigma\subset\ft^\vee_\Q$ of all dimensions $\le r$ since these will describe strata in the cut by a cone of maximal dimension $r$\,---\,in fact the closed strata of the stratification \eqref{eq:algcutstrata} which generalises \eqref{eq:symplectic_cut}.
\begin{figure}[h]
    \centering
    \includegraphics[width=.7\textwidth, keepaspectratio]{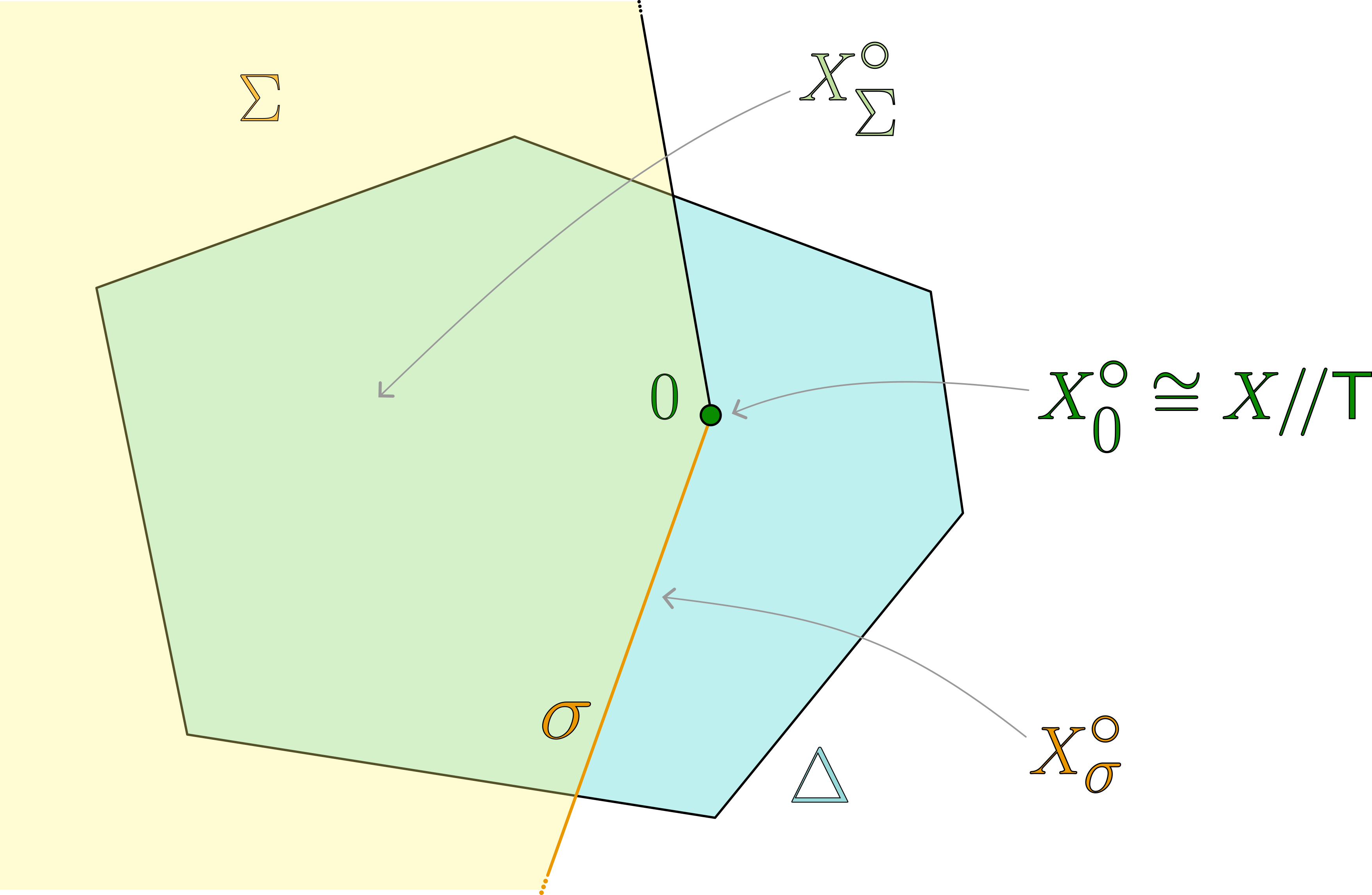}
    \caption{Stratification of the cut $X_\Sigma$ seen from $\Delta \cap \Sigma$.}
    \label{fig:stratification}
\end{figure}

\subsection{Definition of algebraic cut}
Fix an algebraic torus $\sT$ of rank $r$ acting on a polarised projective-over-affine Deligne-Mumford stack $(X,L)$, with moment polyhedron $\Delta$ such that $L$-semistable points are $L$-stable.

Let $\Sigma:=\cone(\psi_1, \dots, \psi_k) \subset \ft^\vee_\Q$ be a simplicial cone of dimension $k:= \dim(\Sigma)\leq r$ spanned by primitive integral elements $\psi_i \in \ft_\Z^\vee$. We denote by $ \Sigma_\C:=\bigoplus_{i=1}^k \Psi_i$ the associated $\sT$-representation of Notation\;\ref{sigma_rep}. The dual basis $\{z_i\}$ defines coordinates $z_1, \dots, z_k$ on $\Sigma_\C$. Let two copies of $\sT$ act on $X \times \Sigma_\C$, the first on $X$ only and the second diagonally,
$$
\sT\times\dT\ \acts\ (X \times \Sigma_\C,\pi_X^\ast L).
$$

\begin{definition}
Suppose $\Sigma$ is transverse to $\Delta(X)$ in the sense of Definition \ref{def:transversal} below. The \textit{algebraic cut} of $X$ by $\Sigma$ is the GIT quotient stack
    \begin{align}\label{eq:alg_cut}
        (X_\Sigma,\,L_\Sigma)\ :=\ \big(X \times \Sigma_\C,\,\pi_X^*L \big)\div\;\dT
    \end{align}
with its residual $\sT$ action.
\end{definition}

\begin{remark}\label{rem:action_factorisation}
Note that the stabiliser subgroup $\wSigma \subset \sT$ of $\Sigma_\C$ defined in Notation\;\ref{Tperptilde} acts trivially on $X_\Sigma$, so the $\sT$ action factors through the action of $\sT/\;\wSigma$.
\end{remark}

The transversality assumption is to ensure that the quotient \eqref{eq:alg_cut} is a Deligne-Mumford stack (non-transversal $\Sigma$ usually produce an Artin stack), as we prove in Lemma \ref{lem:semistable_finite_stabiliser} below. This Section will be devoted to proving the following basic properties of $(X_\Sigma,L_\Sigma)$.

\begin{theorem}\label{thm:structure_cut}
For transverse $\Sigma$ the cut $(X_\Sigma,L_\Sigma)$ has the following properties.
    \begin{enumerate}
    \item[(a)] $X_\Sigma$ is a projective-over-affine Deligne-Mumford stack.
        \item[(b)] The moment polyhedron of $\sT\acts(X_\Sigma, L_\Sigma)$ is $\Delta\cap \Sigma$.
        \item[(c)] If $\,\im\!\big[H^0(X, \cO_X)^\sT\to H^0(X^{\Sigma-\ss}, \cO_{X^{\Sigma-\ss}})^\sT\big]$ is finite dimensional\,\footnote{For instance if $X$ is compact, or if the fixed locus $X^\sT$ is compact and intersects the closure of every $\sT$ orbit so that $H^0(X, \cO_X)^\sT$ is finite dimensional.} and $\Sigma\cap\Delta$ is bounded then $X_\Sigma$ is projective.
        \item[(d)] There is a $\sT$-invariant stratification of $X_\Sigma$ whose closed strata $\lbrace X_\sigma\,:\,\sigma \in \Faces \rbrace$ are algebraic cuts by the faces of $\Sigma$. In particular taking $\sigma=\{0\}$ gives $X\div\;\sT\subset X_\Sigma$.
        \item[(e)] If $\sigma$ is a face of $\Sigma$, the corresponding locally closed stratum $X_\sigma^\circ \subset X_\sigma$ is
\beq{eq:algcutstrata}
        X_\sigma^\circ\ \cong\ X^{\sigma-\ss}/\,\wsigma \qquad \text{where} \qquad X^{\sigma-\ss}\ :=\ \bigcup\nolimits_{\phi \in \sigma} X^{L\otimes\Phi^{-1}-\ss}.
\eeq
    \end{enumerate}
\end{theorem}
        

\begin{remark}
The irreducible components of $X_\Sigma$ are $\sT$-equivariantly birational to irreducible components of $X$ if $\Sigma$ is spanned by an integral basis of $\ft^\vee_\Z$. More generally irreducible components of $X_\Sigma$ are $\sT$-equivariantly birational to irreducible components of $X/\,\wSigma$ by \eqref{eq:algcutstrata}, which shows the open stratum $X^\circ_\Sigma \cong X^{\Sigma-\ss}/\,\wSigma$. Note $\wSigma$ is a finite group when $\Sigma$ is full dimensional.
\end{remark}  

We can also describe the fixed locus of $\sT\curvearrowright X_\Sigma$ in terms of the zero-dimensional intersections between the faces of $\sigma$ and the strata of $\Delta$.
\begin{proposition}\label{pro:fixed_locus}
    If $\Sigma$ is transversal to $\Delta$ in the sense of Definition \ref{def:transversal}, then
    \begin{equation}\label{eq:mu_fixed_loci}
        \mu(X_\Sigma^\sT) \= \partial_0(\Sigma\cap\Delta)  \= \coprod\nolimits_{\sigma \in \Faces} \sigma \cap \partial_{r-\dim(\sigma)}\Delta
    \end{equation}
is a finite set of points $p$. Moreover,
    \begin{enumerate}
        \item Each such $p\in\sigma\cap\partial_{r-\dim(\sigma)}\Delta$ lies in the smooth part $\partial_{r-\dim(\sigma)}^{\sm}\Delta$ of the stratum.
        \item\label{cond:fix} Therefore $\partial_{r-\dim(\sigma)}\Delta$ coincides near $p$ with the affine space $\tau_p \subseteq \ft_\Q^\vee$ of Definition \ref{def:tau}. Every face of $\Sigma$ is transversal to $\tau_p$; in particular $\sigma\cap \tau_p = \lbrace p \rbrace$ and $\Sigma \cap \tau_p$ is a simplicial cone of full dimension in $\tau_p$, generated by the projections of the generators of $\Sigma$ to $\tau_p$ along $\rspan(\sigma)$.
        \item\label{cond:3} The component $F_p$ of $X_\Sigma^\sT$ such that $\mu(F_p) = p$ is the algebraic cut of $\mu^{-1}(\tau_p)$ (from Definition \ref{def:mu_tau}) by $\sigma$,
    \begin{align*}
        F_p \= \(\mu^{-1}(\tau_p)\)_\sigma\ \cong\ \mu^{-1}(\tau_p) \div\,\wsigma
    \end{align*}
    embedded in the stratum $X_\sigma^\circ \cong X^{\sigma-\ss}/\,\wsigma$. In particular $F_0 \cong X\div\;\sT$.
    \end{enumerate}
\end{proposition}

\begin{figure}[ht]
    \centering
    \includegraphics[width=0.9\textwidth, keepaspectratio]{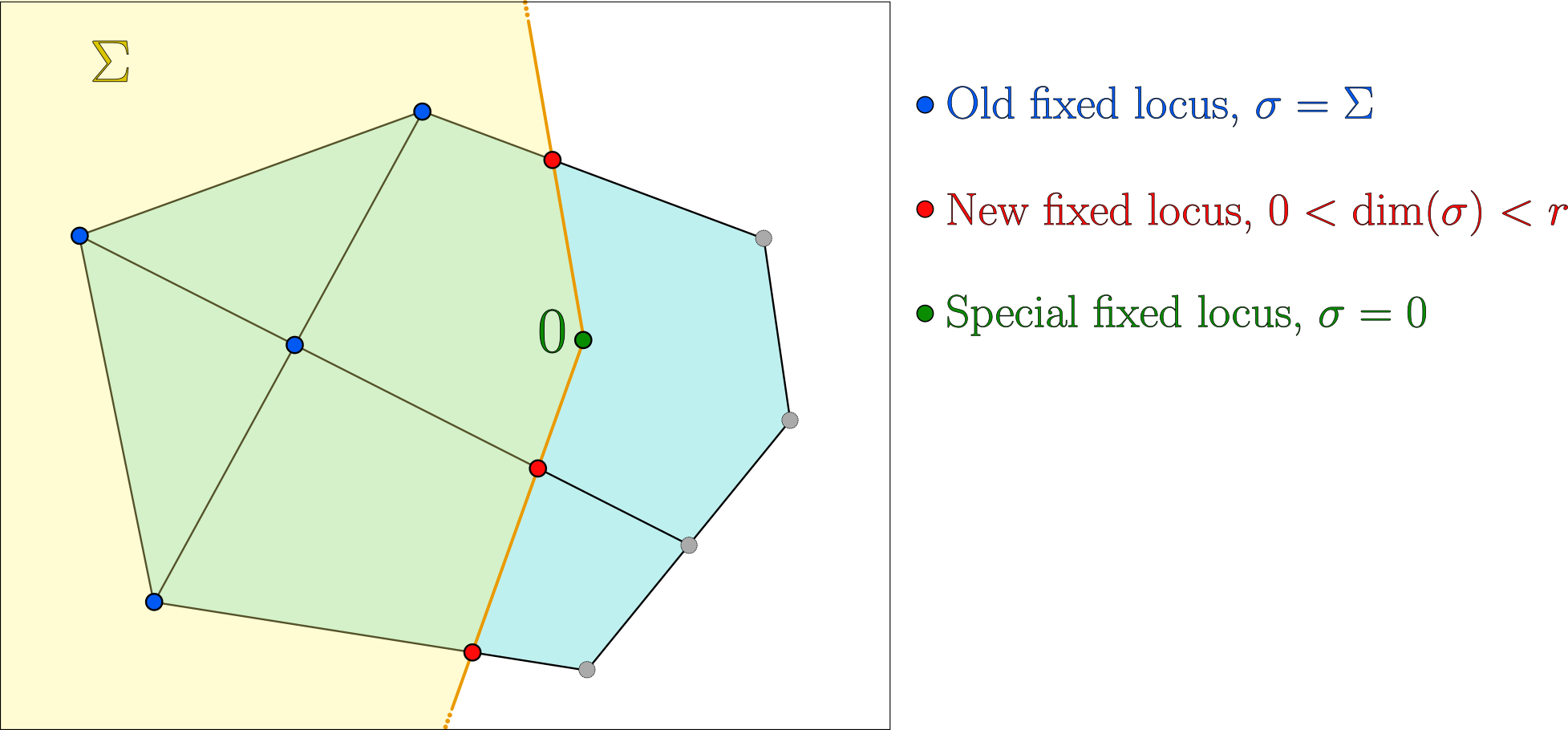}
    \caption{The classification of the components of $X_\Sigma^\sT$ into special, old and new fixed loci, together with the corresponding face $\sigma$ from \eqref{eq:mu_fixed_loci}.}
    \label{fig:classification_fix}
\end{figure}

\begin{remark}\label{rem:classification_fix_loci}
We illustrate Proposition \ref{pro:fixed_locus} for a full dimensional $\Sigma\subset\ft^\vee_\Q$ in Figure \ref{fig:classification_fix}. Components $F_p$ of $X_\Sigma^\sT$ fall into three categories, depending on the face $\sigma$ of $\Sigma$ in whose interior the point $p=\mu(F_p)$ lies: (1) the smallest, (2) the biggest, or (3) another.
\begin{enumerate}
\item \textcolor{darkgreen}{The special fixed locus} $F_0 \cong X\div\;\sT$ with $\sigma=\{0\}=\mu(X\div\;\sT)$.
\item \textcolor{blue}{The old fixed loci} are the components of the fixed locus $(X^\sT \cap X^{\Sigma-\ss})/\,\wSigma$ corresponding to the big face $\sigma=\Sigma$. These are the components of $X^\sT$ about which the birational equivalence between $X/\,\wSigma$ and $X_\Sigma$ is a local isomorphism.
\item  \textcolor{red}{The new fixed loci} are the other components. Their moment polyhedra are points of some proper face of $\Sigma$. We consider them undesired artifacts of the algebraic cut construction.
\end{enumerate}
\end{remark}

The rest of this Section is devoted to proving Theorem \ref{thm:structure_cut} and Proposition \ref{pro:fixed_locus} and could be skipped by a trusting reader.

\subsection{Semistable locus on the product} 
Here we describe a stratification of the $\dT$-semistable locus of $X\times \Sigma_\C$. Since every face $\sigma$ of $\Sigma$ is a simplicial cone in $\ft^\vee_\Q$ we get a subrepresentation $\sigma\_\C \subset \Sigma_\C$ of dimension $\dim\sigma$ by the construction of Notation\;\ref{sigma_rep}. In terms of the expression $\Sigma=\cone(\psi_1,\dots,\psi_k)$ and the dual basis of coordinates $z_i$ \ref{sigma_rep} we have $\sigma=\cone(\psi_i)_{i\in I}$ for some $I$ and so $\sigma\_\C=\{z_j=0\,:\,j\not\in I\}\subseteq\Sigma_\C$. Then we set
$$
\sigma_\C^\circ\ :=\ \big\{z_i\ne0\ :\ i\in I\big\}\ \subseteq\ \sigma\_\C
$$
to be the open subset where the torus $\sT/\,\wsigma$ acts with finite stabilisers, as in \ref{sigma_rep}.

\begin{lemma}\label{lem:stratification_stable_locus} There is a stratification of the $\dT$-semistable locus
\beq{eq:stratification_ss}
(X \times \Sigma_\C)^{\ss}\ = \bigsqcup_{\sigma \in \Faces} X^{\sigma-\ss} \times \sigma_\C^\circ \qquad \text{where} \qquad X^{\sigma-\ss}\ :=\ \bigcup_{\psi \in \sigma} X^{L\otimes\Psi^{-1}-{\ss}}
\eeq
into locally closed substacks.
\end{lemma} 

\begin{proof}
    A point $(x,\Bz_0)$ belongs to $(X\times \Sigma_\C)^{\ss}$ if and only if there is a $\ell>0$ and a $\dT$-invariant section of $\pi_X^\ast L^\ell$ which does not vanish on it. These sections decompose as
    \begin{equation}\label{eq:section_decomposition}
        H^0(X\times \Sigma_\C, \pi_X^\ast L^\ell)^{\dT}\ \cong\ \bigoplus\nolimits_{\theta \in \Sigma} H^0(X,L^\ell)\_{\theta} \otimes \Sym(\Sigma_\C^\vee)\_{-\theta}
    \end{equation}
and $\Sym(\Sigma^\vee_\C)\_{-\theta}$ is spanned by the monomials ${\bf z}^\theta$ of Notation\;\ref{sigma_rep}. Therefore $(x,\Bz_0)\in(X\times \Sigma_\C)^{\ss}$ if and only if there is $\ell>0,\,\theta\in\Sigma$ and a section
$$
s\otimes{\bf z}^\theta\ \in\ H^0(X,L^\ell)\_{\theta} \otimes \Sym(\Sigma_\C^\vee)\_{-\theta}
$$
such that $(s\otimes\Bz^\theta)(x,\Bz_0)\ne0$. This holds if and only if
$$
s(x)\,\ne\,0\ \text{and}\ \Bz_0^\theta\,\ne\,0\ \,\iff\ \,\exists\,\sigma\subset\Sigma\ \text{such that}\ x\,\in\,X^{\sigma-\ss} \text{ and } \Bz_0\,\in\,\sigma_\C^\circ,
$$
where $\sigma=\cone\!\(\psi_i:(\Bz_0)_i\ne0\)$ is the unique face of $\Sigma$ such that $\Bz_0\in\sigma^\circ_\C$.
\end{proof}

We now consider conditions under which the algebraic cut is a Deligne-Mumford stack. Recall from \eqref{eq:union_polyhedra} that, given $0\leq d \leq r$, the smooth locus $\partial_d^{\sm}\Delta$ is a union of dense open subsets of polyhedra of dimension $d$.

\begin{definition}\label{def:transversal}
        If $X$ is irreducible we say that $\Sigma$ \textit{is transverse} to $\Delta$ if, for every face $\sigma$ of $\Sigma$ and every $d\leq \dim\Delta$,
		\begin{enumerate}
            \item\label{cond:first} $\sigma \cap \partial^{\sm}_{d} \Delta$ is either empty or has pure dimension $\dim\sigma+d-r$, and
            \item\label{cond:second} $\sigma \cap \(\partial_{d}\Delta\setminus\partial_{d}^{\sm}\Delta\)$ is smaller, i.e. of dimension $<\dim\sigma+d-r$.
		\end{enumerate}
		For $X$ with irreducible components $X_i$ we say that $\Sigma$ is transversal to $\Delta(X)$ if $\Sigma$ is transversal to each $\Delta(X_i)$ and every face $\sigma$ of $\Sigma$ intersects $\partial_{r-\dim\sigma\,}\Delta(X)$ inside $\partial^{\sm}_{r-\dim\sigma\,}\Delta(X)$.\footnote{We already know this intersection lies in one of the $\partial_{r-\dim\sigma\,}\Delta(X_i)$ so this last condition is that it should not lie in an intersection of two of them unless they coincide locally.}
	\end{definition}

    \begin{figure}[ht]
    \begin{minipage}{0.48\textwidth}\ 
      \includegraphics[height=0.66\linewidth]{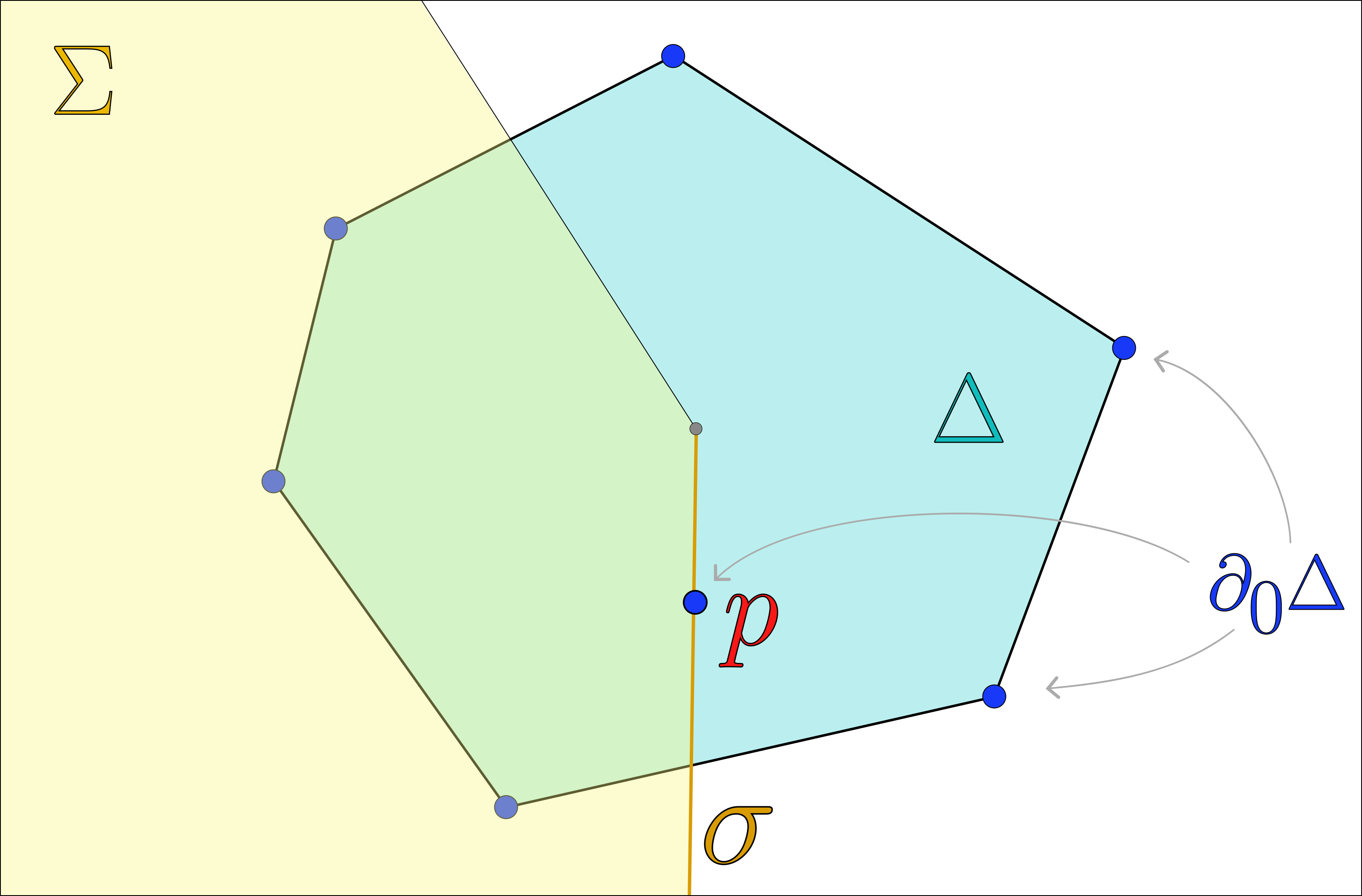}
    \end{minipage}
    \begin{minipage}{0.48\textwidth}\quad
      \includegraphics[height=0.66\linewidth]{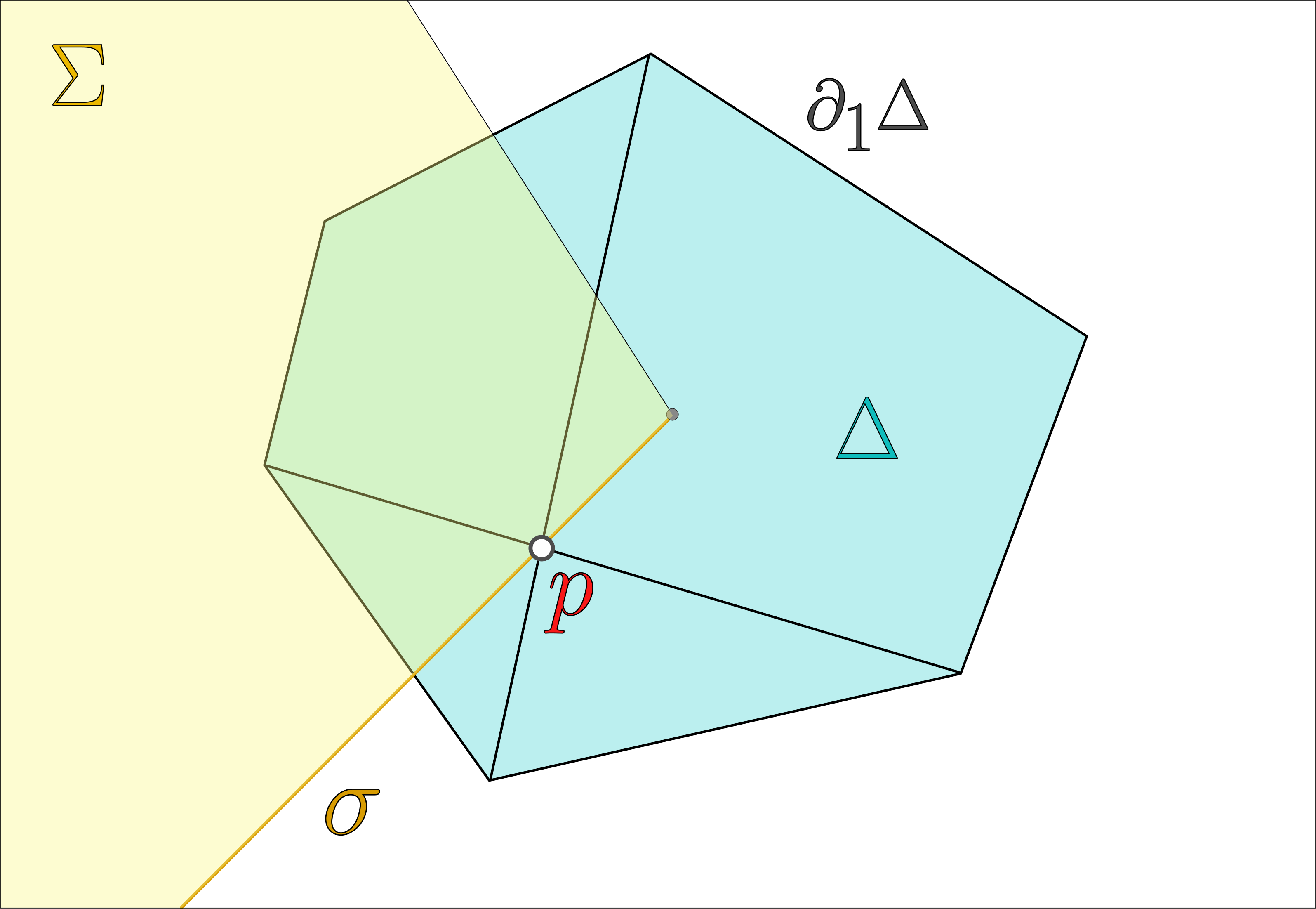}
    \end{minipage}
    \caption{Two examples of cones $\Sigma$ \emph{not} transversal to $\Delta$. On the left \eqref{cond:first} fails because $p\in\partial_0\Delta$; on the right \eqref{cond:second} fails because $p\in\partial_1\Delta\setminus\partial_1^{\sm}\Delta$.}
    \end{figure}
    \begin{figure}[h]
    \centering
    \includegraphics[width=.5\textwidth, keepaspectratio]{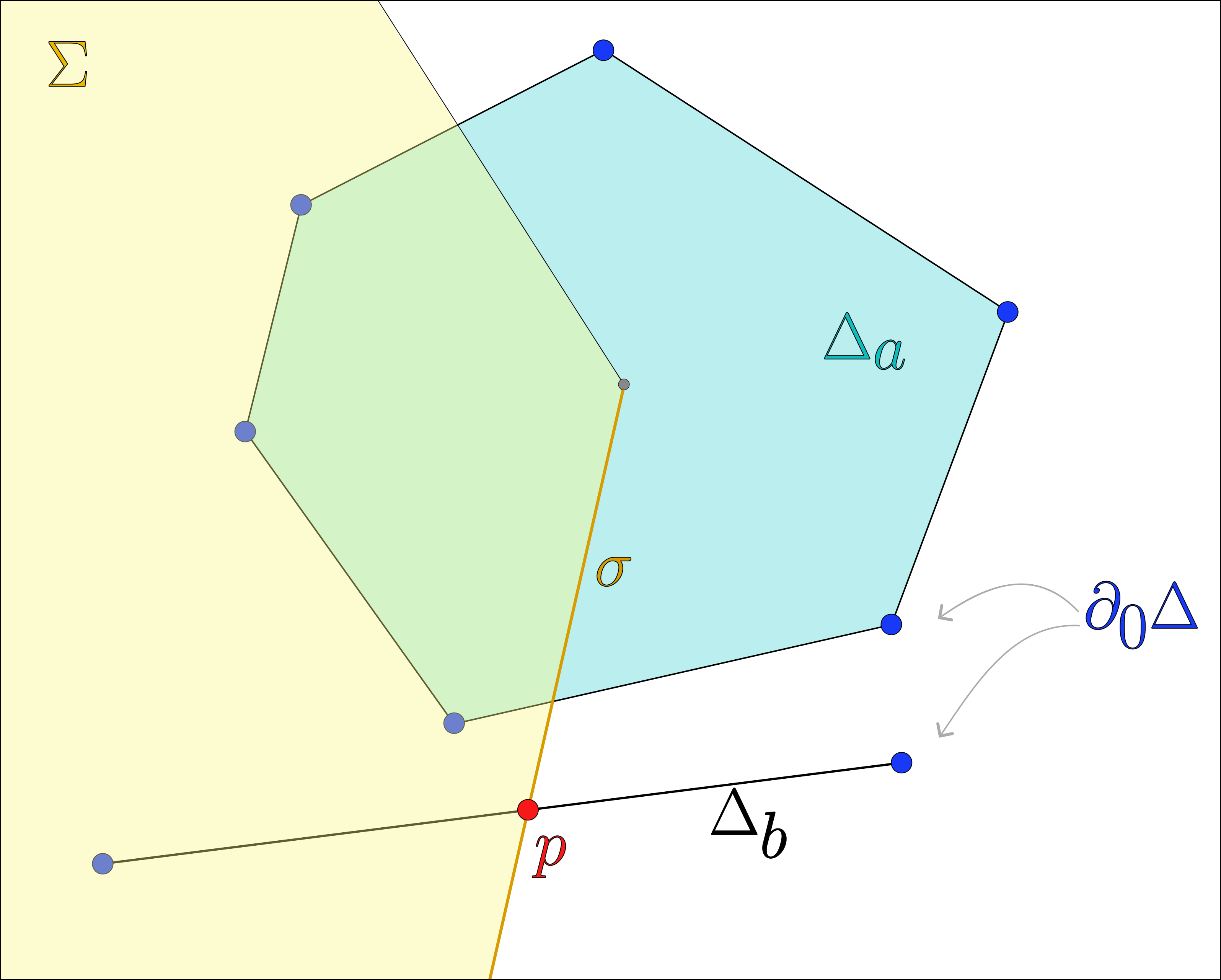}
    \caption{An example of $\Sigma$ transversal to $\Delta$ when $X = X_a \cup X_b$ is reducible. Although $p$ belongs to $\sigma \cap(\partial_2\Delta_b\setminus\partial_2^{\sm}\Delta_b)$, this does not break Definition \ref{def:transversal}\;(2) because that only applies to $\partial_d$ with  $d\le\dim\Delta_b = 1$.}
    \label{fig:transversal}
    \end{figure}   

    Note that we could have asked for (1) and (2) to hold for all $d\le\dim\Delta$ for any $X$. But by dealing with its irreducible components separately, thus letting $d$ vary, we allow for transversality in examples like the one illustrated in Figure \ref{fig:transversal} where some component of $X$ has a moment polyhedron of smaller dimension. This will be important in the applications to moduli spaces.
    
    \begin{lemma}\label{lem:semistable_finite_stabiliser}
If $\Sigma$ is transversal to $\Delta$ then $(\pi_X^\ast L)$-semistable points are $(\pi_X^\ast L)$-stable for the $\dT$ action on $X \times \Sigma_\C$. In particular $X_\Sigma$ is a Deligne--Mumford stack.
	\end{lemma}

	\begin{proof}
By working one irreducible component at a time we may assume that $X$ is irreducible. We will show that the stabiliser group $(\dT)_{(x,\Bz)}$ of every semistable point $(x,\Bz)\in(X\times\Sigma_\C)^{\ss}$ is finite. By \cite[Remark 8.3]{Kirwan} this implies that semistable points are stable.

By Lemma \ref{lem:stratification_stable_locus} $(x,\Bz) \in X^{\sigma-\ss} \times \sigma_\C^\circ$ for some face $\sigma\subseteq\Sigma$. Without loss of generality,
        \begin{enumerate}
            \item[(a)] We assume $\sigma$ is the smallest face such that $x \in X^{\sigma-\ss}$, because passing to a smaller face makes the possible stabiliser groups larger. Indeed, suppose $\lambda \subset \sigma$ is a smaller face with $x \in X^{\lambda-\ss}$ and $\mathbf{w} \in \lambda^\circ_\C$, then $(x, \mathbf{w})$ is semistable and
            $$
            (\dT)_{(x,\Bz)}\ \cong\ \sT_x \cap \wsigma\,\subset\ \sT_x \cap \widetilde\sT_{\lambda^\perp} \= (\dT)_{(x, \mathbf{w})}.
            $$
            \item[(b)] We similarly assume $\sT_x$ is maximal among the stabilisers groups of points in $X^{\sigma-\ss}$. If $y \in X^{\sigma-\ss}$ is another point with $\sT_x \subset \sT_y$ then $(y,\Bz)$ is semistable and
            $$
            (\dT)_{(x,\Bz)}\ \cong\ \sT_x \cap \wsigma\,\subset\ \sT_y \cap \wsigma \= (\dT)_{(y, \Bz)}.
            $$
        \end{enumerate}

        Let $Y\subseteq X^{\sT_x}$ be the irreducible component containing $x$. By Lemma \ref{lem:trivial_act_from_polyhedron}, $\mu(Y) \subset \partial_d\Delta$ is a polyhedron of dimension $d:=r-\dim\sT_x$ contained in a translate of $(\sT_x^0)^\perp$, where $\sT_x^0$ is the identity component of $\sT_x$ (and a torus). Since $x \in X^{\sigma-\ss} \cap Y$, we know that $\sigma$ intersects $\mu(Y)$.
Thus by (a) we know the relative interior of $\sigma$ intersects $\mu(Y)$, and by (b) combined with Lemma \ref{lem:interior_polytope} we know that this intersection is inside the relative interior of $\mu(Y)$. 
        
Now apply Definition \ref{def:transversal}\,(1) and (2) to the parts of $\mathring\sigma\cap\mathring\mu(Y)$ that lie in $\partial_d^{\sm}\Delta$ and $\partial_{d}\Delta\setminus\partial_{d}^{\sm}\Delta$ respectively. This shows that $\rspan(\sigma)$ is transversal to $(\sT_x^0)^\perp$\,---\,the tangent space to $\mu(Y)$. Equivalently the intersection $\sT_x^0\cap\wsigma$ of their annihilators is $\{0\}$, so $(\dT)_{(x,\Bz)} = \sT_x \cap \wsigma$ is finite.
	\end{proof}

\subsection{Properness}
We next show that $X_\Sigma$ is projective if $\Sigma \cap \Delta$ is bounded and $H^0(X,\cO_X)^\sT$ is finite dimensional. This latter condition is too strong (for instance if $X$ has an entirely unstable component $Y$ with infinite dimensional $H^0(Y,\cO_Y)^\sT$) so we weaken it as follows.

    \begin{lemma}
        Suppose the restriction map $H^0(X,\cO_{X})^\sT \rightarrow H^0(X^{\Sigma-\ss},\cO_{X^{\Sigma-\ss}})^\sT$ has finite dimensional image. If $\Sigma$ is transversal to $\Delta$ and $\Sigma \cap \Delta$ is bounded, then $X_\Sigma$ is projective.
    \end{lemma}
    \begin{proof} We work with the commutative diagram
\beq{diagT}
\begin{tikzcd}[row sep=18pt]
H^0(X,\cO_X)^\sT \ar[r, "r"]\ar[d, "\pi_X^*"']& H^0\(X^{\Sigma-\ss},\cO_{X^{\Sigma-\ss}}\)^\sT \ar[d, "\pi_X^*"] \\
H^0(X\times\Sigma_\C,\cO_{X\times\Sigma_\C})^\dT \ar[r, "R"]& H^0\((X\times\Sigma_\C)^{L-\ss},\cO_{(X\times\Sigma_\C)^{L-\ss}}\)^\dT.
\end{tikzcd}
\eeq
The horizontal arrows are the restriction maps. For the right hand vertical map note that $\pi_X:X\times\Sigma_\C\to X$ restricts to $(X\times\Sigma_\C)^{L-\ss}\to X^{\Sigma-\ss}$ by \eqref{eq:stratification_ss}.

By Proposition \ref{pro:proj_criterion} the GIT quotient $(X\times\Sigma_\C,\,\pi_X^*L)\div\;\dT$ is projective if and only if the restriction map $R$
has finite dimensional image. It is sufficient to prove this after dividing by nilpotents since a quasi-projective Deligne-Mumford stack $Y$ is projective if and only if $Y_{\red}$ is.
As in \eqref{eq:section_decomposition} we write the bottom left hand term as
$$
H^0(X\times\Sigma_\C,\cO_{X\times\Sigma_\C})^\dT\ \cong\ H^0(X,\cO_X)^\sT\,\oplus\ \bigoplus\nolimits_{\theta\in\Sigma\setminus\{0\}}H^0(X, \cO_X)\_\theta \otimes \Sym(\Sigma_\C^\vee)\_{-\theta}.
$$
The first summand is in the image of the top left of \eqref{diagT}, on which $r$ has finite rank by assumption. Therefore the restriction of $R$ to this summand also has finite rank.

We will show that the boundedness of  $\Sigma\cap\Delta$ implies that the generators $f \otimes \Bz^\theta$ of the other summands have nilpotent $R(f \otimes \Bz^\theta)$.

Suppose not; then $f(x) \neq 0$ for some closed point $x\in X^{\Sigma-\ss}$. By the definition of $X^{\Sigma-\ss}$ this means $x\in X^{L\otimes\Psi^{-1}-\ss}$ for some $\psi\in\Sigma\cap\Delta$\,---\,i.e. $\exists\,s\in H^0(X,L^k)_{k\psi}$ with $s(x)\ne0$.

Then $f^\ell  s \in H^0(X,L^k)_{k\psi+\ell\Bd}$ is nonzero, so $\psi+\frac{\ell \Bd}{k} \in \Sigma \cap \Delta$ for every $\ell>0$. Since $\theta\ne0$ this makes $\Sigma\cap\Delta$ unbounded.
    \end{proof}
    \subsection{The moment polyhedron}
    \begin{lemma}\label{lem:polyhedron_of_cut}
        The moment polyhedron for $\sT \curvearrowright (X_\Sigma,L_\Sigma)$ is $\Sigma \cap \Delta\subseteq\ft^\vee_\Q$.
    \end{lemma}
    \begin{proof}
A standard feature of GIT quotients, reviewed in Lemma \ref{lem:ss_on_ss_is_ss}, gives
$$
            H^0\((X_\Sigma)_{\red},\,L_\Sigma^k\)\ \cong\ H^0\(X_{\red}\times \Sigma_\C, \pi_X^\ast L^k\)^{\dT}\ \text{ for $k\gg0$ sufficiently divisible}.
$$
As in \eqref{eq:section_decomposition} the right hand side decomposes as
$\bigoplus_{\theta \in \Sigma} H^0(X_{\red},L^k)\_{\theta} \otimes \Sym(\Sigma_\C^\vee)\_{-\theta}$. This is acted on by $\sT$ (not $\dT$!) with the subspace of $\sT$-weight $\psi$ being 
\beq{Tspace}
H^0(X,L^k)_{\psi} \otimes\langle\Bz^{\psi}\rangle\ \text{ if }\psi\,\in\, \Sigma
\eeq
and zero otherwise. But $H^0(X,L^{k})_{\psi}\ne0$ if and only if $\psi/k\in\Delta$ by the definition of $\Delta$ (restricted to $k\gg0$ as in the last of the Remarks \ref{rem:polyhedron}), so \eqref{Tspace} is nonzero if and only if $\psi/k\in\Sigma\cap\Delta$. This is therefore $\Delta(X_\Sigma)$.
    \end{proof}
    \subsection{The stratification}
The stratification $(X \times \Sigma_\C)^{\ss}\ = \bigsqcup_{\sigma \in \Faces} X^{\sigma-\ss} \times \sigma_\C^\circ$ of Lemma \ref{lem:stratification_stable_locus} is $(\sT\times\dT)$-invariant. Quotienting by $\dT$ gives the $\sT$-invariant stratification of the algebraic cut illustrated in Figure \ref{fig:stratification},
$$
    X_\Sigma \= \bigsqcup\nolimits_{\sigma \in \Faces} X_\sigma^\circ\,.
$$
The closed strata $X_{\sigma}$ and locally closed strata $X_{\sigma}^\circ$ are
\begin{align*}
    X_\sigma\ :=\ \left( X \times \sigma\_\C \right)\div\,\dT \qquad \text{and} \qquad X_{\sigma}^\circ\ :=\ X_{\sigma} \setminus \Bigl( \bigcup\nolimits_{\sigma^\prime \subset \sigma} X_{\sigma^\prime} \Bigr).
\end{align*}
Notice the first of these is precisely the algebraic cut of $X$ by the cone $\sigma$, as the notation suggests, embedded in $X_\Sigma$ by
\begin{align*}
    X \times \sigma\_\C\ \subset\ X \times \Sigma_\C \qquad \Longrightarrow \qquad X_\sigma\ \subset\ X_\Sigma.
\end{align*}
We next prove the simple description of the strata $X_\sigma^\circ$ claimed in Theorem \ref{thm:structure_cut}\,(e).

\begin{lemma}\label{lem:polyhedron_strata}
    Given a face $\sigma \subseteq \Sigma$, the corresponding locally closed stratum is
$$
        X^\circ_{\sigma}\ \cong\ X^{\sigma-\ss}/\,\wsigma.
$$
    In particular $X_0 \cong X^\circ_0 \cong X\div\;\sT$ and $\mu(X_{\sigma}, L_\sigma) = \Delta \cap \sigma$ for each face $\sigma \subseteq \Sigma$.
\end{lemma}
\begin{proof}
We express the quotient of the stratum $X^{\sigma-\ss} \times \sigma_\C^\circ$ by $\dT$ in two steps\,---\,first by $(\widetilde\sT_{\sigma^\perp})_{\diag}$ and then by $(\sT/\;\widetilde\sT_{\sigma^\perp})_{\diag}$.
    \begin{align*}
        X^\circ_{\sigma}\ \cong\ \frac{X^{\sigma-\ss} \times \sigma_\C^\circ}{\dT}\ \cong\ \frac{(X^{\sigma-\ss}/\,\wsigma) \times \sigma_\C^\circ}{(\sT/\;\widetilde\sT_{\sigma^\perp})_{\diag}}\ \cong\ X^{\sigma-\ss}/\,\wsigma,
    \end{align*}
    where in the last step we have used that $(\sT/\;\widetilde\sT_{\sigma^\perp})_{\diag}\cong\sigma_\C^\circ$. Finally $\mu(X_{\sigma}, L_\sigma) = \Delta \cap \sigma$ is just Lemma \ref{lem:polyhedron_of_cut} applied to each cut $X_{\sigma}$.
\end{proof}

\subsection{The fixed loci on the cut}
We now describe the fixed locus $X_\Sigma^\sT$, assuming that $\Sigma$ is transversal to $\Delta$.
\begin{proof}[Proof of Proposition \ref{pro:fixed_locus}]
Since $\Sigma$ is transversal to $\Delta$, every face $\sigma \subseteq \Sigma$ intersects $\partial_{r-\dim(\sigma)}\Delta$ in only discrete points $p\in\partial^{\sm}_{r-\dim(\sigma)}\Delta$. Let $\tau_p$ be the affine space of Definition \ref{def:tau} approximating $\partial_{r-\dim(\sigma)}\Delta$ locally around $p$.

    Let $F\subseteq X^\sT_\Sigma$ be a connected component. Denote by $\sigma$ the unique face of $\Sigma$ such that the point $\mu(F, L_\Sigma)$ lies in $\mathring\sigma$. Thus $F$ is a closed substack of $X_\sigma^\circ$, so there is a decomposition
    $$
    X_\Sigma^\sT\ \cong\ \coprod\nolimits_{\sigma \in \Faces} (X_\sigma^\circ)^\sT.
    $$
Recall $X_\sigma^\circ\cong X^{\sigma-\ss}/\,\wsigma$ from Lemma \ref{lem:polyhedron_strata}.
We claim its $\sT$-fixed locus $(X_\sigma^\circ)^\sT$ is
\beq{cupH}
 \left(\frac{X^{\sigma-\ss}}\wsigma\right)^{\!\!\sT}\=\frac{\bigcup_{\sH}\bigl(X^\sH\bigr)^{\sigma-\ss}}\wsigma\ \text{ inside }\ \frac{X^{\sigma-\ss}}\wsigma\,,
\eeq
    where the union is over the subgroups $\sH \subset \sT$ complementary to $\wsigma$ (i.e. $\sH \times \sT_{\!\sigma^\perp} \to \sT$ is an isogeny). The inclusion $\supseteq$ is clear. To see $\subseteq$ we pick a point $[x]\in X^{\sigma-\ss}/\,\wsigma$. Then $\wsigma$ acts quasi-freely on the lift $x\in X^{\sigma-\ss}$ because $X_\Sigma$ is a Deligne-Mumford stack. Therefore the infinitesimal action $\ft\to T_x\(\wsigma\!\cdot x\)$ to the tangent space to the orbit is a surjection $\ft\onto\sigma^\perp$ whose kernel $\fh\subset\ft$ defines the subgroup $\sH\subset\sT$ such that $x\in X^\sH$.

Recall from Lemma \ref{lem:trivial_act_from_polyhedron} that the moment polyhedron of a component $Y$ of $X^\sH$ lies in a translate of $\fh^\perp$ which intersects $\sigma$ in at most a single point (since $\fh$ is a complement to $\sigma^\perp\subset\ft$). For those $Y$ which intersect $X^{\sigma-\ss}$ the intersection is a point $p$, so $\mu(Y)\subset\tau_p$ and $Y$ is a component of $\mu^{-1}(\tau_p)$ from Definition \ref{def:mu_tau}. Thus \eqref{cupH} becomes
\begin{equation*}
    (X_\sigma^\circ)^\sT\ \cong\ \bigcup\nolimits_{p \in \sigma\;\cap\;\partial_{r-\dim(\sigma)}\Delta}\ \frac{\mu^{-1}(\tau_p)^{\sigma-\ss}}\wsigma\,.
\end{equation*}
Finally we claim all the terms in the union are distinct and disjoint because they have distinct $L_\Sigma$-moment polytopes $\{p\}$. This is because
the $\sT$-weights of
$$
H^0\(\mu^{-1}(\tau_p)^{\sigma-\ss}/\,\wsigma,L_\Sigma^k\)\=H^0\(\mu^{-1}(\tau_p)^{\sigma-\ss},L^k\)^{\wsigma}\ \text{ for }\ k\gg0
$$
lie in both $\tau_p$ and the annihilator $\sigma\subset\ft^\vee_\Q$ of $\wsigma$. So the only weight is $\{p\}=\sigma\cap\tau_p\subset\ft^\vee_\Q$.
\end{proof}

\section{Behrend--Fantechi virtual classes on quotients}\label{BFsec}
In this Section we recall how to descend an equivariant perfect obstruction theory to a GIT quotient with finite stabilisers. Applied to the algebraic cut  
when $\Sigma$ is transversal to $\Delta$, this shows a $\sT$-equivariant perfect obstruction theory on $X$ induces a $\sT$-equivariant perfect obstruction theory on $X_\Sigma$. We also prove a certain compatibility between these two obstruction theories.

\subsection{Descending perfect obstruction theories}
Suppose $\sG$ acts on a quasi-projective polarised Deligne-Mumford stack $(X,L)$ such that all $L$-semistable points are $L$-stable, so that $X \div \sG$ is a Deligne-Mumford stack too. Let $(\EE,\phi)$ be a $\sG$-equivariant \emph{obstruction theory}: a morphism 
$$
\phi:\EE\To\LL_X
$$
in the equivariant derived category $D^-(\coh^\sG X)$ to the cotangent complex of $X$ such that
\begin{itemize}
\item $h^0(\phi)\colon h^0(\EE)\to\Omega_X$ is an isomorphism and
\item $h^{-1}(\phi)\colon h^{-1}(\EE)\to h^{-1}(\LL_X)$ is a surjection.
\end{itemize}
In this section we will assume that $(\EE,\phi)$ is a \emph{perfect} obstruction theory \cite{BehrendFantechi}, which means that in addition $\EE\in\Perf^{\;\sG}(X)$ is a perfect complex of amplitude $[-1,0]$. If $X$ satisfies the resolution property (for instance if it is a scheme; see \cite[Proposition 5.1]{Kr2} for other sufficient conditions) then $\EE$ is quasi-isomorphic to a two-term complex of $\sG$-equivariant vector bundles $E^{-1}\to E^0$. However a perfect obstruction theory defines a virtual cycle even without this condition \cite[Theorem 6.2.1]{Kr1}.

\begin{lemma}\label{lem:inducingPotOnQuotient}
$(\EE,\phi)$ induces a perfect obstruction theory on $X\div\sG$ by \eqref{morphism_of_triangles} below.
\end{lemma}

\begin{proof}
Differentiating the $\sG$ action gives a morphism
\beq{surj}
\LL_{X^{\ss}}\To\cO_{X^{\ss}}\otimes\fg_\C^\vee\ \text{ inducing a surjection }\,\Omega_{X^{\ss}}\Onto\cO_{X^{\ss}}\otimes\fg_\C^\vee
\eeq
on $h^0$ because semistable points are stable and so have finite stabiliser groups. Therefore taking cocones completes the right hand square below to a commutative diagram of exact triangles that defines $\(\EE_{X\div\sG},\phi_{X\div\sG}\)$:
		\begin{equation}\label{morphism_of_triangles}
			\begin{tikzcd}
				\EE_{X\div\sG} \arrow[r] \arrow[d, "\phi_{X\div\sG}"] & {\mathbb{E}|\_{X^{\ss}}} \arrow[d, "\phi|\_{X^{\ss}}"] \arrow[r]& {\mathcal{O}_{X^{\ss}} \otimes \fg_\C^\vee} \arrow[d, equal]\\
				\pi^*\mathbb{L}_{X\div\sG} \arrow[r] & {\mathbb{L}}_{X^{\ss}} \arrow[r] & {\mathcal{O}_{X^{\ss}} \otimes \fg_\C^\vee.\!}
			\end{tikzcd}
		\end{equation}
Here $\pi\colon X^{\ss}\to X\div\sG$ is the projection. Taking the long exact sequence of cohomologies and using \eqref{surj} then shows that $\mathbb{E}_{X\div\sG}\in\Perf^{\;\sG}(X^{\ss})$ has amplitude $[-1,0]$ and $\phi_{X\div\sG}$ is an isomorphism on $h^0$ and a surjection on $h^{-1}$. Therefore  $\(\mathbb{E}_{X\div\sG},\phi_{X\div\sG}\)$ descends under the equivalence $\pi^*:\Perf X\div\sG\cong\Perf^{\;\sG}X^{\ss}$ to a perfect obstruction theory on $X\div\sG$.
	\end{proof}
\begin{remark} Let $\sH \triangleleft \sG$ be a closed normal subgroup. (It is therefore also reductive by \cite[Definition 1.1.11 and below]{Conrad}.) Then replacing $\sG$ by $\sH$ in the above construction gives (the pullback to $X^{\sH-\ss}$ of) a perfect obstruction theory on $X\div\;\sH$ from \eqref{morphism_of_triangles}. But we can form \eqref{morphism_of_triangles} in $D^-\(\!\coh^\sG(X^{\sH-\ss})\)\cong D^-\(\!\coh^{\sG/\sH}(X\div\;\sH)\)$, thus ensuring the induced perfect obstruction theory on $X\div\;\sH$ is $\sG/\sH$-equivariant.
\end{remark}

Let $\sT$ be a reductive (not necessarily normal) subgroup of $\sG$ such as a maximal torus. Via its action on $(X,L)$ inherited from that of $\sG$ we have $X^{\sG-\ss}\subseteq X^{\sT-\ss}$ and so a diagram
\begin{equation}\label{eq:abelianisation_diagram}
\begin{tikzcd}[row sep=18pt]
X^{\sG-\ss}/\;\sT\ \ar[r, hook, "i"]\ar[d, "p", shorten <= -.8mm]& X\div\;\sT \\
X\div\sG,\!
\end{tikzcd}
\end{equation}
with $p$ smooth and $i$ an open embedding. Thus they define flat pullback maps on cycles.

	\begin{proposition}\label{pro:compatibility_of_obs}
	Let $\phi\colon\EE\to\LL_X$ be a $\sG$-equivariant perfect obstruction theory on $(X,L)$. Suppose $\sT$-semistable points are $\sT$-stable (so $\sG$-semistable points are also $\sG$-stable, as shown in (1) of Section \ref{section:NA_JK_loc}). Then the induced virtual fundamental classes \cite{BehrendFantechi} satisfy
		\begin{align*}
			p^\ast [X\div\sG]^{\vir} \= i^\ast [X\div\;\sT]^{\vir}\ \text{ in }\ A_\ast(X^{\sG-\ss}/\;\sT)\ \cong\ A^\sT_\ast(X^{\sG-\ss}).
		\end{align*}
	\end{proposition}

	\begin{proof}
Firstly \cite[Proposition 5.10]{BehrendFantechi} applied to the (smooth and therefore lci) open immersion $i$ proves $i^\ast[X\div\;\sT]^{\vir}=[X^{\sG-\ss}/\;\sT]^{\vir}$ by taking $u=v=i$ in \cite[Equation 11]{BehrendFantechi}.

Use $p^*$ to map \eqref{morphism_of_triangles} to its version with $\sG$ replaced by $\sT$, and take cones. The result is the two versions of the left hand vertical arrow sit in following the diagram of exact triangles on $X^{\sG-\ss}/\;\sT$,
$$
\begin{tikzcd}
p^*\EE_{X\div\sG} \ar[r]\ar[d,"p^*\phi_{X\div\sG}"] & \EE_{X^{\sG-\ss}/\;\sT} \ar[d,"\phi_{X^{\sG-\ss}/\;\sT}"] \ar[r]& \op{Ad}^\vee \ar[d,equals,shift right=.5ex]\\
p^*\LL_{X\div\sG} \ar[r] & \LL_{X^{\sG-\ss}/\;\sT} \ar[r] & \op{Ad}^\vee\!.\!
\end{tikzcd}
$$
Here $\op{Ad}^\vee\cong\Omega_p$ is the bundle $X^{\sG-\ss}\times\_\sT(\fg/\ft)^\vee$ on $X^{\sG-\ss}/\;\sT$, where $\sT$ acts with the coadjoint action on $(\fg/\ft)^\vee$.
This gives $p^* [X\div\sG]^{\vir}= [X^{\sG-\ss}/\;\sT]^{\vir}$ by taking both $u$ and $v$ to be the smooth morphism $p$ in \cite[Proposition 5.10]{BehrendFantechi}.
	\end{proof}

\subsection{The obstruction theory on the cut}
Consider the setting of Theorem \ref{thm:structure_cut}: a polarised projective-over-affine Deligne-Mumford stack $(X,L)$ acted on by a torus $\sT$, with a simplicial cone  $\Sigma := \cone(\psi_1, \dots, \psi_k)\subset\ft^\vee_\Q$ transverse to the moment polytope $\Delta(X)$. Thus $X\div\;\sT$ and the algebraic cut $X_\Sigma$ are Deligne-Mumford stacks.

\begin{corollary}\label{cor:obs_cut}
A $\sT$-equivariant perfect obstruction theory $(\EE,\phi)$ on $X$ induces the perfect obstruction theory $(\EE_{X_\Sigma},\phi_{X_\Sigma})$ of \eqref{eq:morph_triangles_cut} on $X_\Sigma$.
\end{corollary}
\begin{proof}
$(\EE,\phi)$ induces a $(\sT \times \dT)$-equivariant perfect obstruction theory
    \begin{align*}
        \phi \boxplus \id : \EE \boxplus \Omega_{\Sigma_\C} \To \LL_X \boxplus \Omega_{\Sigma_\C}
    \end{align*}
on $X \times \Sigma_\C$.    By Lemma \ref{lem:inducingPotOnQuotient} it induces a $\sT$-equivariant perfect obstruction theory on $X_\Sigma$
as (descent down $\pi:(X\times\Sigma_\C)^{\ss}\to X_\Sigma$ of) the $(\sT \times \dT)$-equivariant morphism
\begin{equation}\label{eq:morph_triangles_cut}
	\begin{tikzcd}[row sep=2pt]
        \mathbb E \boxplus\(\cO\otimes\Sigma_\C^\vee\) \arrow[r] \arrow[dd] & \cO \otimes \ft_\C^\vee \arrow[dd, equal] &  & \pi^*\mathbb{E}_{X_\Sigma} \arrow[dd, "\pi^*\phi_{X_\Sigma}"] \\ & &\cong & \\
    \mathbb{L}_X \boxplus\(\cO\otimes\Sigma_\C^\vee\) \arrow[r] & \cO \otimes \ft_\C^\vee &  & \pi^*\mathbb{L}_{X_\Sigma},\!
    \end{tikzcd}
    \end{equation}
    where $\LL_X \rightarrow \cO \otimes \ft_\C^\vee$ and $\Sigma_\C^\vee\rightarrow\ft_\C^\vee$ are induced from the $\sT$ actions on $X$ and $\Sigma_\C$.
\end{proof}
    
\noindent We now have two obstruction theories on the big open stratum $X_\Sigma^\circ \subset X_\Sigma$ of Theorem \ref{thm:structure_cut},
    \begin{enumerate}
        \item\label{item:first_obs} one induced from $(\EE_{X_\Sigma},\phi_{X_\Sigma})$ by restriction to $X_\Sigma^\circ \subset X_\Sigma$, and
        \item\label{item:second_obs} one induced on $X_\Sigma^\circ \cong X^{\Sigma-\ss}/\,\wSigma$ via Lemma \ref{lem:inducingPotOnQuotient} from the $\wSigma$-equivariant obstruction theory $(\EE,\phi)$ on $X$.
    \end{enumerate}

\begin{lemma}\label{lem:restriction_open_stratum}
    The obstruction theories \eqref{item:first_obs} and $\eqref{item:second_obs}$ agree on $X_\Sigma^\circ$.
\end{lemma}
\begin{proof}
 By construction the obstruction theory $\big[\mathbb{E}_{X_\Sigma} \to \mathbb{L}_{X_\Sigma}\big]\big|_{X_\Sigma^\circ}$ is given by restricting \eqref{eq:morph_triangles_cut} to $X^{\Sigma-\text{ss}} \times \Sigma_\C^\circ$ and then descending it to $X_\Sigma^\circ\cong(X^{\Sigma-\text{ss}} \times \Sigma_\C^\circ)/\;\dT\cong X^{\Sigma-\text{ss}}/\,\wSigma$.
    The $\sT$ action induces the isomorphism
    $$
    \Sigma^\circ_\C\ \cong\ \sT/\,\wSigma \qquad \text{and thus} \qquad \Sigma_\C^\vee\ \cong\ \ker(\ft_\C^\vee \to \fs_\C^\vee),
    $$
    where $\fs_\Z$ is the cocharacter lattice of $\wSigma$. Dividing both rows of \eqref{eq:morph_triangles_cut} by the subcomplex $\Sigma_\C^\vee \xrightarrow{\ \sim\ }\ker(\ft_\C^\vee \to \fs_\C^\vee)$ makes it quasi-isomorphic to (the restriction to $X^{\Sigma-\text{ss}}$ of)
$$    
\begin{tikzcd}[row sep=16pt]
        \mathbb E \arrow[r] \arrow[d] & \cO \otimes \fs_\C^\vee \arrow[d, equal] \\
    \mathbb{L}_X \arrow[r] & \cO \otimes \fs_\C^\vee.\!
    \end{tikzcd}
    $$
Comparing this to \eqref{morphism_of_triangles} we see it is (the pullback to  $X^{\Sigma-\text{ss}}$ of) the obstruction theory on $X^{\Sigma-\ss}/\,\wSigma$ induced by $(\EE,\phi)$.
\end{proof}

\section{Noncompactness}
While it would be easiest to assume $(X,L)$ is projective, we make a small digression to discuss the weakest conditions under which $(X,L)\div\,\sT$ is projective and our results hold.

Our conditions\,---\,which we call (weak) \emph{$\eta$-semiprojectivity} by analogy with \cite{HRV}\footnote{Closely related notions carry names like ``attracting 1-parameter subgroup" or ``BB-complete $\C^*$ action".}\,---\,ensure that our projective-over-affine Deligne-Mumford stack $X$ behaves, $\sT$-equivariantly, as if it were projective. We state a number of conditions and relations between them, then spend the rest of the Section proving them.
(In Section \ref{C*sec} we go even further and allow $(X,L)\div\sT$ to be noncompact in the presence of an extra $\C^*$ action.)

Picking $\eta\in \ft_\Q\setminus\{0\}$, we define the 1-parameter subgroup $\widetilde\eta:\C^*\into\sT$ to be the minimal integral cocharacter in $\ft_\Z\cap\(\Q_{>0}\cdot\eta\)$. Then the condition is that it \emph{contracts $X$ to a compact core $X^{\widetilde\eta}$} in the following sense.

\begin{definition}\label{etasp}
    Given $\eta\in \ft_\Q$, we say that $X$ is \textit{$\eta$-semiprojective} if
    \begin{enumerate}
        \item\label{cond:semipr_1} $\lim_{s\to 0}\widetilde{\eta}(s)\cdot x$ exists for every closed point $x \in X$, and
        \item\label{cond:semipr_2} the fixed locus $X^{\widetilde{\eta}}$ is projective.
    \end{enumerate}
\end{definition}
    
So $\eta$-semiprojectivity is a topological property of $\widetilde\eta:\C^*\acts X$.
We can rephrase \eqref{cond:semipr_1} in terms of the cone $\cC(X)\subseteq\ft_\Q^\vee$
of $\Q_{\ge0}$-linear combinations of $\sT$-weights in $H^0(\cO_X)_{\red}$ of \eqref{Cdef}, and \eqref{cond:semipr_2} in terms of $H^0(\cO_X)^{\widetilde{\eta}}$.

\begin{proposition}\label{pro:alg_char_eta}
$\eta$-semiprojectivity is open in $\eta$ and equivalent to
    \begin{itemize}
        \item[$(1')$] $\cC(X) \subset \lbrace \eta\le0\rbrace$, and\vspace{.5mm}
        \item[$(2')$] $H^0(\cO_X)^{\widetilde{\eta}}$ is finite dimensional.
    \end{itemize}
Moreover \eqref{cond:semipr_1} $=(1')$, while \eqref{cond:semipr_1},\,\eqref{cond:semipr_2} together imply $(1'')\ \cC(X) \setminus \lbrace 0 \rbrace \subset \lbrace \eta<0\rbrace$.
    \end{proposition}

    \begin{proposition}\label{pro:eta_semiproj_dominant}
        Given a $\sT$-equivariant (or $\widetilde\eta$-equivariant) morphism $f: X \to Y$ between projective-over-affine Deligne-Mumford stacks,
        \begin{itemize}
            \item $Y$ is $\eta$-semiprojective $\!\so\! X$ is $\eta$-semiprojective if $f$ is proper,
            \item $X$ is $\eta$-semiprojective $\!\so\! Y$ is $\eta$-semiprojective if $f$ is dominant.
        \end{itemize}
    \end{proposition}

The results of the paper rely on cutting $X$ in the $\eta$-positive direction, so they only depend on the closure of $X^{\lbrace \eta\geq 0\rbrace-\ss} :=\bigcup_{\psi \in \lbrace \eta\ge 0 \rbrace} X^{L\otimes\Psi^{-1}-\ss}$ in $X$. Throwing the rest away gives a weaker notion of $\eta$-semiprojectivity that will suffice for our purposes.

\begin{definition}\label{weaksemi}
    We say that $(X,L)$ is \textit{weakly $\eta$-semiprojective} if the closure $\overline{X^{\lbrace \eta\geq 0 \rbrace-\ss}}$ is $\eta$-semiprojective.
\end{definition}

\begin{proposition}\label{weak2} Weak $\eta$-semiprojectivity is open in $\eta$. It is equivalent to \eqref{cond:semipr_1} $=(1')$ of Definition \ref{etasp} plus any of the following three equivalent conditions,
    \begin{enumerate}
        \setcounter{enumi}{2}
        \item\label{cond:semipr_3} $H^0\Big(\cO_{\overline {X^{\lbrace \eta\geq 0 \rbrace-\ss}}}\Big){}^{\raisebox{0.6ex}{$\scriptstyle\!\;\widetilde\eta$}}$ is finite dimensional,
        \item\label{cond:semipr_4} $H^0(\cO_X)^{\widetilde\eta}\to H^0\(\cO_{X^{\lbrace \eta\geq 0 \rbrace-\ss}}\)^{\widetilde\eta}$ has finite rank,
        \item\label{cond:semipr_5} $\dim H^0\bigl(\cO_{Y}\bigr)^{\widetilde\eta} <\infty$ for every irreducible component $Y \subseteq X$ with $\mu(Y) \cap \lbrace \eta\geq 0 \rbrace \neq \emptyset$.
    \end{enumerate}
        \end{proposition}
        
It is sometimes simpler to ignore (2),\,$(2')$,\,(3),\,(4) or (5) and replace them by an assumption using $\sT$ in place of $\widetilde\eta$.

\begin{proposition}\label{Tversion}
Consider the following conditions,
    \begin{itemize}
        \item[$(2'')$] $H^0(\cO_X)^\sT$ is finite dimensional or $X^\sT$ is projective,
        \item[$(3'')$] $H^0\Big(\cO_{\overline{X^{\lbrace \eta\geq 0\rbrace-\ss}}}\Big){}^{\raisebox{0.6ex}{$\scriptstyle\!\;\sT$}}$ is finite dimensional.
    \end{itemize}
If \eqref{cond:semipr_1},\,$(2'')$ hold then $X$ is $\eta^\prime$-semiprojective for a general small perturbation $\eta^\prime$ of $\eta$.

\noindent If \eqref{cond:semipr_1},\,$(3'')$ hold then $(X,L)$ is weakly $\eta^\prime$-semiprojective for a general small perturbation $\eta^\prime$.
\end{proposition}

So the summary is that if $C(X)\subset\ft^\vee_\Q$ is contained in a half space and $X^\sT$ is compact then we can choose $\eta$ to ensure $X$ is $\eta$-semiprojective. And for weak $\eta$-semiprojectivity we replace the second condition by the compactness of $\(\overline{X^{\lbrace \eta\geq 0 \rbrace-\ss}}\)^\sT$. 

\begin{proposition}\label{pro:eta_proj_compact}
   Suppose $(X,L)$ is weakly $\eta$-semiprojective and $\Sigma \setminus \lbrace 0 \rbrace \subset \lbrace \eta>0\rbrace$ is a simplicial cone transverse to $\Delta(X)$. Then $X\div\,\sT$ and the algebraic cut $X_\Sigma$ are projective.
\end{proposition}

Of course we can combine this with Proposition \ref{Tversion}: if (1) and $(2'')$ or $(3'')$ hold then $(X,L)$ is weakly $\eta'$-semiprojective so applying Proposition \ref{pro:eta_proj_compact} to $\eta'$ (using that the condition $\Sigma \setminus \lbrace 0 \rbrace \subset \lbrace \eta>0\rbrace$ is open in $\eta$) shows that $X\div\sT,\,X_\Sigma$ are projective. The rest of this Section is devoted to proving the above results.

\begin{proof}[Proof of Proposition \ref{pro:alg_char_eta}]
We begin with \eqref{cond:semipr_1} $\iff(1')$. Pick nonconstant equivariant algebra generators $f_1, \dots, f_n\in H^0(\cO_X)$ to define an equivariant closed embedding $X_\mathrm{aff} \into \AA^n$.
  
Since $X \to X_\mathrm{aff}$ is projective the existence of $s\to 0$ limits on $X$ is equivalent to the existence of such limits on $X_\mathrm{aff}$. And since $X_\mathrm{aff} \into \AA^n$ is closed this is equivalent to the existence of limits on  $\AA^n$, i.e. to the representation $\widetilde\eta \acts \AA^n$ having nonnegative weights.\medskip
  
Notice $(2')\!\so\!$ \eqref{cond:semipr_2} is given by Lemma \ref{trayn}.\medskip

Next we show \eqref{cond:semipr_1} and \eqref{cond:semipr_2} together imply $(2')$.
By \eqref{cond:semipr_1} the closure of every orbit $\widetilde\eta\cdot x$ on $X$ contains a point $\lim_{s\to0}\widetilde\eta(s)\cdot x$ in $X^{\widetilde\eta}$, so $f(X) = f(X^{\widetilde\eta})$ for every
$f \in H^0(\cO_X)^{\widetilde\eta}$. And $X^{\widetilde\eta}$ is projective by \eqref{cond:semipr_2}, so $f(X)$ is a finite number of points and $f$ is locally constant.
Thus the Noetherian ring $H^0(\cO_X)^{\widetilde\eta}$ is Artinian, giving $(2')$.\medskip

We claim that $(2')$ implies that $\cC(X) \setminus \lbrace 0 \rbrace$ does not intersect $\{\eta=0\}$. Otherwise we get $f\in H^0(\cO_X)_{\red}$ with nonzero weight but $\widetilde\eta$-weight zero. Its powers are nonzero (because the ring is reduced) and linearly independent (because their weights are distinct) so they make $\dim H^0(\cO_X)^{\widetilde\eta}=\infty$, contradicting $(2')$.

Therefore $\eta$-semiprojectivity implies $(1'')\ \cC(X) \setminus \lbrace 0 \rbrace\subset\{\eta<0\}$\,---\,an open condition in $\eta$ which implies \eqref{cond:semipr_1}. Since small perturbations $\eta'$ of $\eta$ only make the fixed locus smaller $X^{\widetilde\eta^\prime} \subseteq X^{\widetilde\eta}$ (as can be seen by embedding in $\PP^m \times \AA^n$), condition \eqref{cond:semipr_2} is also open. Thus $\eta$-semiprojectivity is open.
\end{proof}

\begin{proof}[Proof of Proposition \ref{pro:eta_semiproj_dominant}]
If $f$ is proper the closure of an $\wt\eta(\C^*)$ orbit upstairs is proper over the closure of the orbit downstairs, so if $\lim_{s\to0}$ exists downstairs then it also exists upstairs. Thus the $\eta$-semiprojectivity of $X$ implies the the $\eta$-semiprojectivity of $Y$.

By Proposition \ref{pro:alg_char_eta} $\eta$-semiprojectivity is a property of $\wt\eta(\C^*)\acts H^0(\cO_X)$ for projective-over-affine Deligne-Mumford stacks. If
$f$ is dominant then $\wt\eta(\C^*)\acts H^0(\cO_Y)\subseteq H^0(\cO_X)$ enjoys the same properties.
    \end{proof}

\begin{proof}[Proof of Proposition \ref{weak2}]
Since $X^{\lbrace \eta\geq 0 \rbrace-\ss}$ is open its complement is projective-over-affine, so \eqref{cond:semipr_1} and $(1')$ are equivalent for it. But by construction $\Delta\(X\setminus X^{\lbrace \eta\geq 0 \rbrace-\ss}\)\subseteq\{\eta\le0\}$ which implies that $\cC\(X\setminus X^{\lbrace \eta\geq 0 \rbrace-\ss}\)\subseteq\{\eta\le0\}$ by Lemma \ref{lem:polyhedron_structure}.
Thus all $s\to0$ limits exist on $X\setminus X^{\lbrace \eta\geq 0 \rbrace-\ss}$ and $(1')=$ \eqref{cond:semipr_1} holds automatically on it.

So checking \eqref{cond:semipr_1} for $X$ is equivalent to checking it for $\overline{X^{\lbrace \eta\geq 0 \rbrace-\ss}}$. Therefore by Definition \ref{weaksemi} weak $\eta$-semiprojectivity is equivalent to \eqref{cond:semipr_1} and \eqref{cond:semipr_3}.\medskip

For \eqref{cond:semipr_3} $\iff$ \eqref{cond:semipr_4} note the map \eqref{cond:semipr_4} factors through \eqref{cond:semipr_3} as
    \beq{neweq}
H^0(\cO_X)^{\widetilde\eta}\ \rt r\ H^0\Big(\cO_{\overline {X^{\lbrace \eta\geq 0 \rbrace-\ss}}}\Big){}^{\raisebox{0.6ex}{$\scriptstyle\!\;\widetilde\eta$}}
\ \Into\ H^0\(\cO_{X^{\lbrace \eta\geq 0 \rbrace-\ss}}\)^{\widetilde\eta}.
    \eeq
Therefore it has finite rank if the central term is finite dimensional. Thus \eqref{cond:semipr_3} $\so$ \eqref{cond:semipr_4}.

Conversely, \eqref{cond:semipr_4} implies that $\im r$ in \eqref{neweq} is finite dimensional. Because the composition $\overline {X^{\lbrace \eta\geq 0 \rbrace-\ss}} \into X \to X_{\mathrm{aff}}$ is projective, $H^0\(\cO_{\overline {X^{\lbrace \eta\geq 0 \rbrace-\ss}}}\)$ is a finitely generated module over $H^0(\cO_X)$. Taking $\widetilde\eta$-invariants we find the central term of \eqref{neweq} is a finitely generated module over the finite dimensional $\im r$. Thus it is finite dimensional, giving \eqref{cond:semipr_3}.\medskip

For \eqref{cond:semipr_3} $\iff$ \eqref{cond:semipr_5} we use
\beq{bigcup}
\overline{X^{\lbrace \eta\geq 0 \rbrace-\ss}}\=\bigcup\nolimits_{Y\subseteq X\,:\,\mu(Y)\;\cap\;\{\eta\ge0\}\,\neq\,\emptyset}\ Y,
\eeq
where the $Y\subseteq X$ are the irreducible components of $X$ with their maximalist scheme structure of Notation\;\ref{irredcpts}. This gives
$$
H^0\(\,\overline{X^{\lbrace \eta\geq 0 \rbrace-\ss}}\,\)^{(\widetilde\eta)}\ \Into\ \bigcup\nolimits_Y H^0(\cO_Y)^{(\widetilde\eta)},
$$
first without the $(\widetilde\eta)$s and then with by taking $\widetilde\eta$-invariants. Thus the left hand side is finite dimensional if and only the right hand side is.\medskip

The openness of weak $\eta$-semiprojectivity will also follow from \eqref{bigcup}. Since each $\mu(Y)$ is a union of polyhedra, perturbing $\eta$ can only make the set of $Y$s with $\mu(Y)\cap\{\eta\ge0\}\neq\emptyset$ smaller, so \eqref{bigcup} gives\footnote{We pass to reduced structures because $X^{\lbrace \eta'\geq 0 \rbrace-\ss}$ could have an embedded locus not in $X^{\lbrace \eta\geq 0 \rbrace-\ss}$.}
\beq{closedinc}
\(\,\overline{X^{\lbrace \eta'\geq 0 \rbrace-\ss}}\,\)_{\red}\ \subseteq\ \(\,\overline{X^{\lbrace \eta\geq 0 \rbrace-\ss}}\,\)_{\red}.
\eeq
Notice Definition \ref{etasp} depends only on the reduced structure, and if it holds for the right hand side of \eqref{closedinc} then it holds for the \emph{closed} subset on the left. That is, if $(X,L)$ is weakly $\eta$-semiprojective, then it is $\eta'$-semiprojective.
\end{proof}

\begin{proof}[Proof of Proposition \ref{Tversion}]
Suppose \eqref{cond:semipr_1} and $(2'')$ hold. Since $X^{\widetilde\eta'} = X^\sT$ for a general $\eta'\in \ft_\Q$ (as can be seen by embedding into $\PP^m \times \AA^n$) we see $X^{\widetilde\eta'}$ is projective by Lemma \ref{trayn}. This is \eqref{cond:semipr_2} for $\eta'$. And if it is a sufficiently small perturbation of $\eta$ it too satisfies the open condition $(1'')$ which implies \eqref{cond:semipr_1}. Hence $X$ is $\eta'$-semiprojective.\medskip

Finally suppose \eqref{cond:semipr_1} and $(3'')$ hold for $(X,L)$. It follows that \eqref{cond:semipr_1} holds for the closed subset $\(\overline{X^{\lbrace \eta\geq 0 \rbrace-\ss}}\)_{\red}\subseteq X$. Applying Proposition \ref{pro:alg_char_eta} to this subset shows it is $\eta^\prime$-semiprojective for a small perturbation $\eta'$ of $\eta$. By the closed inclusion \eqref{closedinc} this means \eqref{cond:semipr_1} and \eqref{cond:semipr_2} also hold for $\(\overline{X^{\lbrace \eta'\geq 0 \rbrace-\ss}}\)_{\red}$. Thus $(X,L)$ is weakly $\eta'$-semiprojective.
\end{proof}

\begin{proof}[Proof of Proposition \ref{pro:eta_proj_compact}]
 Since the composition
 $$
        H^0(\cO_X)^\sT\To H^0\(\cO_{X^{\lbrace \eta \geq 0 \rbrace-\ss}}\)^\sT \To H^0\(\cO_{X^{\Sigma-\ss}}\)^\sT
$$
    factors through (4) it has finite rank.  By Proposition \ref{pro:proj_criterion}  this shows $X\div\,\sT$ is projective.
  
    By Proposition \ref{pro:alg_char_eta}\;($1''$) and Lemma \ref{lem:polyhedron_structure} we find that $\lbrace \eta > 0 \rbrace \cap \Delta(X)$ is bounded, so $\Sigma\cap \Delta(X)$ is too. By Theorem \ref{thm:structure_cut}\;(c) this implies that $X_\Sigma$ is projective.
\end{proof}

\section{Localisation on the algebraic cut}\label{section:loc_cut}

\begin{definition} \label{data} Our \emph{standard data} $\sT\acts(X,L,\EE,\phi),\,\alpha\in H^*_\sT(X)$ and assumptions are
\begin{itemize}
\item an algebraic torus $\sT\acts(X,L)$ acting on a polarised projective-over-affine Deligne-Mumford stack such that semistable points are stable,
\item a $\sT$-equivariant perfect obstruction theory $(\EE,\phi)$ on $X$,
\item an equivariant cohomology class $\alpha \in H_\sT^\ast(X)$.
\end{itemize}
To do the algebraic cut we will use the following \emph{auxiliary data} and assumptions,
\begin{itemize}
\item a simplicial cone $\Sigma := \cone(\psi_1, \dots, \psi_r)\subseteq\ft^\vee_\Q$ of full dimension $r=\rk\sT$, transverse to the moment polytope $\Delta=\Delta^\sT(X,L)$, such that
\item $\Sigma\cap\Delta$ is bounded and $\;\im\!\big[H^0(X, \cO_X)^\sT\to H^0(X^{\Sigma-\ss}, \cO_{X^{\Sigma-\ss}})^\sT\big]$ is finite dimensional, so $X\div\;\sT$ and $X_\Sigma$ are projective Deligne-Mumford stacks by Theorem \ref{thm:structure_cut}.
\end{itemize}
\end{definition}

By Corollary \ref{cor:obs_cut} $(\EE,\phi)$ induces a $\sT$-equivariant perfect obstruction theory on $X_\Sigma$, while $\alpha$ induces a class $\alpha\_\Sigma\in H^*_\sT(X_\Sigma)$ as follows. Let $\sT \times \dT$ act on $X$ through $\sT$ via the homomorphism $(t_1,t_2)\mapsto t_1t_2$. This makes $\pi_X:X\times\Sigma_\C\to X$ into a $(\sT \times \dT)$-equivariant map and defines
\beq{classes_on_cut}
\ H^\ast_\sT(X)\To H^\ast_{\sT \times \dT\!}(X)\xrightarrow{\,\pi_X^\ast}H^\ast_{\sT \times \dT\!}(X\times \Sigma_\C)\To H^*_{\sT\times\dT\!}\((X\times\Sigma_\C)^{\dT-\ss}\)\,\cong\,H^*_\sT(X_\Sigma),\ 
\eeq
which we denote by $\alpha\mapsto\alpha\_\Sigma$. Note its restriction to $X\div\;\sT\subset X_\Sigma$ is just
\beq{induced_class}
\alpha\ \in\ H^\ast_\sT(X)\rt{r}H^\ast_\sT(X^{\sT-\ss})\ \cong\ H^*(X\div\;\sT)\ \ni\ \alpha\_0,
\eeq
where $r$ denotes restriction. More generally, recall from Proposition \ref{pro:fixed_locus} the components $F_p$ of the fixed locus $(X_\Sigma)^\sT$, where $p\in\mu(X_\Sigma^\sT) =\partial_{0}(\Sigma \cap \Delta)$, and their description $F_p=\mu^{-1}(\tau_p) \div\,\wsigma$. We see that $\alpha\_\Sigma|\_{F_p}$ is
\beq{alphap}
\alpha\ \in\ H^*_\sT(X)\rt{r}H^*_\sT\Big(\mu^{-1}(\tau_p)^{\wsigma\!-\ss}\Big)\,\cong\,H^*_{\sT/\,\wsigma\!}(F_p)\To H^*_{\sT}(F_p)\ \ni\ \alpha_p,
\eeq
where again $r$ is restriction and the final arrow is pullback along $\sT\to\sT/\,\wsigma$.

The virtual localisation formula \cite{GraberPandharipande} for $\sT \curvearrowright(X_\Sigma,\EE_{X_\Sigma},\phi_{X_\Sigma})$ now reads
\beq{eq:general_loc_cut}
    \int_{[X_\Sigma]^{\vir}} \alpha\_\Sigma \= \sum_{p \in \partial_{0}(\Sigma \cap \Delta)} \int_{[F_p]^{\vir}} \frac{\alpha_p}{e^\sT\(N^{\vir}_{F_p/X_\Sigma}\)}\,.
\eeq
We describe $N^{\vir}_{F_p/X_\Sigma}$ in Lemma \ref{lem:normal_bundles} below. As in Remark \ref{rem:classification_fix_loci}, and as illustrated in Figure \ref{fig:classification_fix}, we consider separately the 3 types of fixed loci:
\begin{enumerate}
\item \textcolor{darkgreen}{the special fixed locus} $F_0=X\div\;\sT$,
\item \textcolor{blue}{the old fixed loci} where $p\in\mathring\Sigma\cap\partial_0\Delta$, and
\item \textcolor{red}{the new fixed loci} where $p\in\mathring\sigma\cap\partial_{r-\dim\sigma}\Delta$ for $0<\dim\sigma<r$.
\end{enumerate}	

\begin{proposition}\label{pro:special_old_contributions}
The contribution of (1) \textcolor{darkgreen}{the special fixed locus} to \eqref{eq:general_loc_cut} is
\beq{eq:special_contribution}
         \int_{[F_0]^{\vir}}\frac{\alpha\_\Sigma|\_{F_0}}{e^\sT\(N^{\vir}_{F_0/X_\Sigma}\)} \= \int_{[X\div\;\sT]^{\vir}}\frac{\alpha\_0}{\prod_{j=1}^r \bigl(c_1(N_j)-\psi_j\bigr)}\,,
\eeq
    where $N_j := X^{\ss} \times\_\sT \Psi_j\in \Pic(X\div\;\sT)$. The contribution of (2) \textcolor{blue}{the old fixed locus} to \eqref{eq:general_loc_cut} is
\beq{old}
      \sum_{p \in \mathring\Sigma\cap\partial_0\Delta} \int_{[F_p]^{\vir}} \frac{\alpha\_\Sigma|\_{F_p}}{e^\sT\(N^{\vir}_{F_p/X_\Sigma}\)} \= \frac{1}{\vert \wSigma\vert}\sum_{\substack{F\subseteq X^\sT\\ \mu(F) \in \Sigma}}
        \int_{[F]^{\vir}} \frac{\alpha \vert_{F} }{e^\sT(N^{\vir}_{F/X})}\,,
\eeq
where on the right hand side the virtual class on $F$ is induced from $F\subseteq X^\sT\into X$.
\end{proposition}
In order to eventually remove the (residues of the) contributions from (3) \textcolor{red}{the new fixed loci} to \eqref{eq:general_loc_cut} we will use the following result about all fixed loci $F_p$ for $p \in \mu(X^\sT_\Sigma)$.

\begin{proposition}\label{pro:weights_normal_bundle}
Write $F_p \cong\mu^{-1}(\tau_p)\div\,\wsigma$ as in Proposition \ref{pro:fixed_locus}\;\eqref{cond:3}. The rational function
\beq{ratint}
        \int_{[F_p]^{\vir}} \frac{\alpha\_\Sigma|\_{F_p}}{e^\sT(N^{\vir}_{F_p/X_\Sigma})}\,:\,\ft_\Q
        \xymatrix{\ar@{-->}[r]& \ \Q}
\eeq
    has poles only on hyperplanes of the form $\lbrace\gamma = 0\rbrace$, where $\gamma \in \ft_\Q^\vee$ is either
    \begin{enumerate}
        \item[(i)] in $\rspan(\sigma)$, or
        \item[(ii)] on a 1-dimensional edge of the simplicial cone $\Sigma \cap \tau_p$.     \end{enumerate} 
\end{proposition}

\begin{figure}[h]
    \centering
    \includegraphics[width=0.9\textwidth, keepaspectratio]{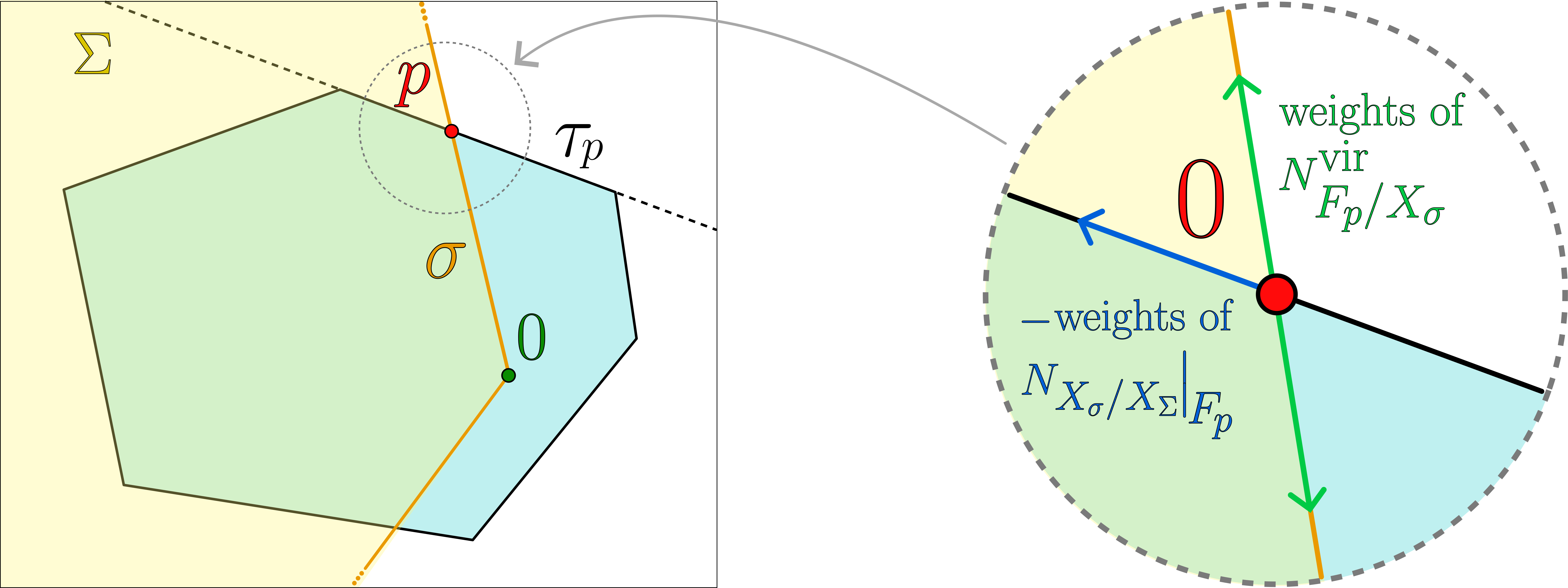}
    \caption{Minus the weights of $N^{\vir}_{F_p/X_\Sigma}=N^{\vir}_{F_p/X_\sigma}\oplus N_{X_\sigma/X_\Sigma}$ described in terms of $\Sigma$ and $\tau_p$. On the right we picture the tangent space to $\ft^\vee_\Q$ at $p$. The $\sT$-weights of $N^{\vir}_{F_p/X_\sigma}$ (in green) belong to $\rspan(\sigma)$, while minus the weights of $N_{X_\sigma/X_\Sigma}\vert_{F_p}$ (in blue) are the generators of the simplicial cone $\Sigma \cap \tau_p$.}
    \label{fig:weights}
\end{figure}

\begin{remark}
    If $X$ is irreducible, minus the weights of the $\sT$ action on the classical normal bundle $N_{F_p/X_\Sigma}$ lie along each direction of the edges of $\Delta \cap \Sigma$ emanating from $p$ \cite[Theorem 2]{Brion}. In particular, in Figure \ref{fig:weights}, there would be no green arrow pointing upwards, out of $\Delta$ (and those pointing downwards would all be along 1-dimensional rays of $\sigma\cap\Delta$ emanating from $p$). In our \emph{virtual} setting Proposition \ref{pro:weights_normal_bundle} is the best we can do. The rest of this Section is devoted to proving Propositions \ref{pro:special_old_contributions} and \ref{pro:weights_normal_bundle}.

\end{remark}
\subsection{Virtual normal bundles}
Let $p \in \mu(X_\Sigma^\sT)$ and consider the corresponding fixed locus $F_p \cong \mu^{-1}(\tau_p)\div\,\wsigma \subseteq X_\sigma\subseteq X_\Sigma$ of Proposition \ref{pro:fixed_locus}\;\eqref{cond:3}. We describe its virtual normal bundle in terms of some special line bundles. Consider the action
$$
\sT \times \dsigma\,\acts\,X^{\sigma-\ss}, \qquad (t,s) \cdot x\ :=\ (ts) \cdot x.
$$
For $\psi \in \ft^\vee_\Z$ let $\psi^{\diag}$ denote the character of $\sT \times \dsigma$ arising from the second factor only,
\begin{align*}
    \psi^{\diag}(t,s)\ :=\ \psi(s),
\end{align*}
with corresponding $(\sT \times \dsigma)$-representation $\Psi^\diag$.
Then the generators $\psi_j,\ 1\le j\le k$ of the cone $\Sigma = \cone(\psi_1, \dots, \psi_k)$ induce $\sT$-equivariant line bundles
\begin{equation}\label{eq:Nj_bundle}
    N_j\ :=\ X^{\sigma-\ss} \times\_{\dsigma}\Psi^\diag_j\ \in\  \Pic^\sT\!\!\(X^{\sigma-\ss}/\,\dsigma\)\ \stackrel{\eqref{eq:algcutstrata}}\cong\ \Pic^{\sT\!}(X^\circ_\sigma).
\end{equation}

\begin{lemma}\label{lem:normal_bundles}
Give $X_\sigma$ its perfect obstruction theory from Corollary \ref{cor:obs_cut}. Then 
\begin{eqnarray}
\EE_{X_\Sigma}|_{F_p} &\cong& \EE_{X\div\,\wsigma}\!\big|_{F_p}\ \oplus\ \bigoplus\nolimits_{j \,:\, \psi_j \notin \sigma} N_j^\vee\big|_{F_p}, \label{Es} \\
N^{\vir}_{F_p/X_\Sigma} &\cong& N^{\vir}_{F_p/X_\sigma}\ \oplus\ \bigoplus\nolimits_{j \,:\, \psi_j \notin \sigma} N_j\big|_{F_p}. \label{Ns}
\end{eqnarray}
The $\sT$-weights of the first summand lie in $\rspan(\sigma)\subset\ft^\vee_\Q$. In the second summand, the $\sT$-weight of $N_j \vert_{F_p}$ is the (possibly rational\footnote{Since the $N_j$ are orbi-line bundles that do not necessarily descend to the coarse moduli scheme, we get rational weights in general.}) projection of $-\psi_j$ to $\tau_p$ along $\rspan(\sigma)$, i.e. minus a rational generator of a 1-dimensional edge of $\Sigma\cap \tau_p$ emanating from $p$.
\end{lemma}

%

\begin{proof}
    Apply the diagram \eqref{eq:morph_triangles_cut} to both $\Sigma$ and $\sigma$ to describe and compare the obstruction theories on $X_\Sigma$ and $X_\sigma$. Writing $\Sigma_\C=\sigma\_\C\oplus\bigoplus_{j:\psi_j \notin \sigma}\Psi^{\;\diag}_j$ by Notation\;\ref{sigma_rep}, this gives
\beq{EEE}
        \EE_{X_\Sigma} \big|_{X^\circ_\sigma}\ \cong\ \EE_{X_\sigma}\big|_{X^\circ_\sigma}\ \oplus\ \bigoplus\nolimits_{j \,:\, \psi_j \notin \sigma}\,N^\vee_j.
\eeq 
Substituting in the isomorphism from Lemma \ref{lem:restriction_open_stratum},
$$
\EE_{X_\sigma}\big|_{X^\circ_\sigma}\ \cong\ \EE_{X\div\,\wsigma}\!\big|_{X^\circ_\sigma}
$$
then restricting to $F_p$ gives \eqref{Es}. Restricting \eqref{EEE} to $F_p$, dualising and taking $\sT$-moving parts gives \eqref{Ns}.

The weights of $N^{\vir}_{F_p/X_\sigma}$ lie in $\rspan\sigma$ because $X_\sigma$ is entirely fixed by $\wSigma$ by Remark \ref{rem:action_factorisation}.

Finally, given $x \in \mu^{-1}(\tau_p)$ over $[x]\in\mu^{-1}(\tau_p)\div\,\wsigma\cong F_p$, we want to compute the weight in $\ft^\vee_\Q$ of the $\sT$ action on $N_j \vert_{[x]}$. Since the isogeny $\sT_{\!\tau_p^\perp}\times \wsigma \rightarrow \sT$  induces an isomorphism
\beq{eq:isogeny_iso}
\ft^\vee_\Q\,\rt{\,\sim\,}\,\(\tau_p^\perp\)^\vee\,\oplus\,\(\sigma^\perp\)^\vee
\eeq
on rational characters, it suffices to compute the weight of $\sT_{\!\tau_p^\perp}\times \wsigma$ acting on $N_j \vert_{[x]}$. Given $(a,b)\in\sT_{\!\tau_p^\perp}\times \wsigma$ acting on $[x,z]\in\eqref{eq:Nj_bundle}$,
$$
(a,b)\cdot[x,z]\=[ab\cdot x,z]\=[a\cdot x,b^{-1}\cdot z]\=[x,\psi_j(b^{-1})z]\=-\psi_j(b)[x,z],
$$
so the weight is $0\oplus-\psi_j \vert_{\wsigma}\!$. Inside $\ft^\vee_\Q$ the inverse of \eqref{eq:isogeny_iso} identifies
    $$
    \(\tau_p^\perp\)^\vee\ \cong\ \rspan(\sigma) \qquad \text{and} \qquad \(\sigma^\perp\)^\vee\ \cong\ \tau_p-p,
    $$
so the weight is the projection of $-\psi_j$ along $\rspan(\sigma)$ to the tangent space to $\tau_p$. And by Proposition \ref{pro:fixed_locus}\eqref{cond:fix} this is minus the direction of a 1-dimensional edge of the cone $\Sigma \cap \tau_p$.
\end{proof}

This result, applied to $p=0,\,\sigma=\{0\}$, gives the first formula \eqref{eq:special_contribution} of Proposition \ref{pro:special_old_contributions}: the fixed part of \eqref{Es} determines the virtual cycle on the fixed locus $X\div\,\sT$ as being the one from the made from descent down the quotient as in Lemma \ref{lem:inducingPotOnQuotient}, while \eqref{Ns} identifies the virtual normal bundle.

Similarly, to get the second formula \eqref{old} of Proposition \ref{pro:special_old_contributions}, we apply Lemma \ref{lem:normal_bundles} to $p\in\mu(X^\sT),\,\sigma=\Sigma$ to see the obstruction theory and virtual normal bundles in $X_\Sigma$ are the same as those in $X$ modulo the action of the finite group $\wSigma$. We then combine this with the isomorphism $F_p\cong F/\,\wSigma$ inducing
$$
A_*(F)\,\cong\,A_*(F/\,\wSigma)\,=\,A_*(F_p), \qquad [F]^{\vir}\,\Mapsto\,\big|\wSigma\!\big|\cdot[F_p]^{\vir}
$$
by Lemma \ref{lem:restriction_open_stratum}, and, by \eqref{alphap},
$$
H^*(F)\,\cong\,H^*(F/\,\wSigma)\,=\,H^*(F_p), \qquad \alpha\vert_F\,\Mapsto\,\alpha\_\Sigma|\_{F_p}\,.
$$
Finally applying Lemma \ref{lem:normal_bundles} to any $p\in\mu(X_\Sigma^\sT)$ gives Proposition \ref{pro:weights_normal_bundle} since the denominator in the integrand \eqref{ratint} is a product of weights of the virtual normal bundle.

\section{Jeffrey--Kirwan residues}\label{section:JK_res}

If we write the contribution \eqref{eq:special_contribution} of the special fixed locus $X \div\;\sT$ to the localisation formula \eqref{eq:general_loc_cut} as \begin{align}\label{eq:origin_contribution}
    \int_{[X\div\;\sT]^{\vir}}\frac{\alpha\_0}{\prod_{j=1}^r\(c_1(N_j)-\psi_j\)} \= \sum_{\ell_1, \dots, \ell_r = 0}^{\infty} \frac{(-1)^r}{\psi_1^{\ell_1+1}\cdots \psi_r^{\ell_r+1}}\int_{[X\div\;\sT]^{\vir}} \alpha\_0 \prod_{j=1}^r c_1(N_j)^{\ell_j},
\end{align}
we note that
\begin{align*}
    \int_{[X\div\;\sT]^{\vir}} \alpha\_0 \,\text{ is the coefficient of }\, (-1)^r\prod_{j=1}^r \frac{1}{\psi_j} \,\text{ in \eqref{eq:origin_contribution}}.
\end{align*}
So, naively speaking, we would like to do something like ``\emph{take the coefficient of $\prod_{j=1}^r\frac1{\psi_j}$ in the localisation formula \eqref{eq:general_loc_cut}}'' to extract this coefficient and deduce a formula for it. While we cannot quite do this because the contributions of the other fixed loci in \eqref{eq:general_loc_cut} have more complicated denominators, in this Section we review residue operators (depending on a choice of cone $\Sigma\subset\ft^\vee_\Q$) that \emph{do} extract this coefficient from \eqref{eq:general_loc_cut}. This residue kills both the left hand side of \eqref{eq:general_loc_cut} and\,---\,for sufficiently wide $\Sigma$\,---\,the contribution of the new fixed loci to the right hand side, as we will show in Section \ref{section:JK_loc}. This will give the virtual Jeffrey--Kirwan localisation formula \eqref{vJK1}. We mainly follow \cite{BrionVergne} but there exist different approaches; see \cite[Equation 29]{SV} for instance. \medskip

    Let $\ft_\Z$ be a free abelian group of rank $r$ and $\ft_\Q=\ft_\Z\otimes\Q$  the associated $r$-dimensional $\Q$-vector space. A finite spanning subset $\Gamma \subset \ft_\Q^\vee$  induces two hyperplane arrangements
    \begin{equation}\label{eq:hyperplanes}
        \cH\,:=\,\bigcup\nolimits_{\gamma \in \Gamma}\{\gamma=0\}\,\subset\,\ft_\Q, \hspace{1cm} \cH^\vee\,:=\,\bigcup\nolimits_{\{\gamma_1, \dots, \gamma_{r-1}\}\subset\Gamma\,} \rspan(\gamma_1, \dots, \gamma_{r-1})\,\subset\,\ft^\vee_\Q.
    \end{equation}
Similarly we have $\cH_\C\subset\ft_\C$. On its complement $\ft_\C\setminus\cH_\C$ the rings of regular functions and Laurent series  are respectively 
    \begin{equation}\label{eq:R_gamma}
        R_\Gamma\,:=\,\Sym(\ft_\C^\vee)\bigl[\Gamma^{-1}\bigr] \qquad \text{and} \qquad \widehat{R}_\Gamma\,:=\,\Big(\prod\nolimits_{d\geq 0} \Sym^d\ft_\C^\vee\Big) \bigl[\Gamma^{-1}\bigr].
    \end{equation}

    \begin{remark}
        In our application, $\Gamma\subset\ft^\vee_\Q$ will be the set of all weights of $N^{\vir}_{X^\sT\;/\;X}$  or  $N^{\vir}_{X_\Sigma^\sT\;/\;X_\Sigma}$, so $\cH$ will be the locus of poles of $1/e^\sT(N^{\vir}):\ft_\Q\dashrightarrow H^*(X^\sT)$ or $H^*\(X_\Sigma^\sT\)$. Thus $1/e^\sT(N^{\vir})$ will be in $R_\Gamma\otimes H^*(X^\sT)$ or $R_\Gamma\otimes H^*(X_\Sigma^\sT)$.
    \end{remark}

    In the one dimensional case $\ft_\Q \cong \Q$ the functions  \eqref{eq:R_gamma} are $R_\Gamma=\C[z^{\pm1}],\,\widehat R_\Gamma=\C(\!(z)\!)$ and there is a unique residue operation up to sign\footnote{The two residues correspond to the two possible orientations on the loop around $0 \in \C$ along which we integrate $\frac1{2\pi i}f(z)dz$ to take the residue of $f$. Flipping orientations corresponds to taking the residue at $\infty \in \PP^1$ instead of $0 \in \PP^1$. The choice is equivalent to a choice of one of the two chambers of $\ft_\Q \setminus \cH = \Q\setminus \lbrace 0 \rbrace$, which in turn determines a unique chamber of $\ft_\Q^\vee \setminus \cH^\vee$ pairing positively with it.} --- namely the one singling out the coefficient of $\pm z^{-1}$.
    In higher dimensions it turns out there are, up to taking linear combinations, finitely many residues on $R_\Gamma,\,\widehat{R}_\Gamma$ determined by a choice of a pair of chambers in $\ft_\Q \setminus \cH$ and $\ft_\Q^\vee \setminus \cH^\vee$. We now describe these residues, following \cite{BrionVergne}.
    
    \subsection{Definition of the JK residue}
The \textit{Jeffrey--Kirwan (JK) residue} associated to a choice of two elements $\eta \in \ft_\Q \setminus \cH$ and $\xi \in \ft_\Q^\vee \setminus \cH^\vee$ is a linear functional
$$
        \JK^\eta_\xi : \widehat{R}_\Gamma \To \C,
$$
    defined as follows. Firstly, we project $\widehat{R}_\Gamma$ to the degree $r$ functions $(R_\Gamma)_r$ with respect to the weight 1 scaling $\C^*$ action on $\ft_\C$ (so $\Sym^d(\ft^\vee_\C)$ has degree $-d$). Secondly, since $\Gamma$ spans $\ft^\vee_\Q$, every $f \in(R_\Gamma)_r$ can be written as a linear combination of functions of the form
    \begin{align}\label{eq:generating_fraction}
        \frac{1}{\prod_{i=1}^r \gamma_i} \quad \text{ where } \rspan(\gamma_1, \dots, \gamma_r)\,=\,\ft_\Q^\vee,\text{ and}
    \end{align}  
    \begin{align}\label{eq:nongenerating_fraction}
        \frac{q}{\prod_{i=1}^k \gamma_i} \quad \text{where }  \rspan(\gamma_1, \dots, \gamma_k)\,\neq\,\ft_\Q^\vee\,\text{ and }q \in \Sym^{k-r}(\ft_\C^\vee).
    \end{align}
In both \eqref{eq:generating_fraction} and \eqref{eq:nongenerating_fraction}   each $\gamma_i$ is in $\Gamma \cup -\Gamma$ and, without loss of generality (but very importantly), we choose it to satisfy $\langle \gamma_i, \eta\rangle >0$. Wedging together any basis for $\ft^\vee_\Z$ defines the same element $\pm\;o\in\Wedge^r\ft_\Q^\vee$ up to sign. Then the JK residue of \eqref{eq:generating_fraction} is\footnote{Really it is defined as a certain inverse Laplace transform \cite[Section 5]{BrionVergne}, but for simplicity we take the result \eqref{eq:def_JK_r} of this transform as the definition.}
\beq{eq:def_JK_r}
        \JK^\eta_\xi \left( \frac{1}{\prod_{i=1}^r \gamma_i} \right) \= 
\left\{\!\!\begin{array}{cl}
\left|\frac o{\gamma_1 \wedge \dots \wedge \gamma_r}\right| & \xi\,\in\, \mathrm{Cone}(\gamma_1, \dots, \gamma_r), \\
0 & \text{otherwise,}
\end{array}\right.       
\eeq
    while the JK residue of \eqref{eq:nongenerating_fraction} is zero. As a definition this is overdetermined, as there are many linear relations among the functions \eqref{eq:generating_fraction}. But Brion and Vergne prove it is consistent:

 \begin{proposition}\label{JKpm}
        The expression \eqref{eq:def_JK_r} determines a well defined linear functional
$$
        \JK^\eta_\xi: \widehat R_\Gamma\To(R_\Gamma)_r\To\C.
$$
It depends on $\eta$ and $\xi$ only through the chambers of $\ft_\Q \setminus \cH$ and $\ft_\Q^\vee \setminus \cH^\vee$ they lie in. It does not depend on $\Gamma$, in the sense that for $f \in \widehat{R}_{\Gamma_1} \cap \widehat{R}_{\Gamma_2}$ and $\eta \notin \cH_1 \cup \cH_2$, $\xi \notin \cH_1^\vee \cup \cH_2^\vee$, the residue $\JK^\eta_\xi(f)$ is the same when computed with respect to $\Gamma_1$ or $\Gamma_2$. Finally, $\JK^{-\eta}_{-\xi}=(-1)^r\JK^\eta_\xi$.
  \end{proposition}
        
    \begin{proof}
        We just state the dictionary between our notation and that of \cite{BrionVergne}. Their $\Delta\subset V$ is our $\Gamma\subset\ft_\Q^\vee$. They denote by $\delta\subset V^\vee\setminus\cH$ and $\gamma\subset V\setminus\cH^\vee$ the chambers which contain $\eta$ and $\xi$. Then $\JK^\eta_\xi$ is well defined because it can be written in terms of their operator $\mathrm{Res}_{\gamma, \delta}$ (defined by a certain inverse Laplace transform, and depending on a choice of orientation on $\ft_\Q$) as the composition
        $$
     \widehat R_\Gamma\To(R_\Gamma)_r\ \cong\ S_\Delta\oplus NG_\Delta[-r]\To S_\Delta\rt{\,\mathrm{Res}_{\gamma, \delta}\,} \wedge^r V^\vee\rt{\,\pm\;o\,}\C.
        $$
Here $S_\Delta\subset(R_\Gamma)_r$ is the span of the functions \eqref{eq:generating_fraction} (which they characterise beautifully as those which tend to zero along every straight line in $\ft_\C\setminus \cH$) while $NG_\Delta[-r]$ is the span of the functions \eqref{eq:nongenerating_fraction} (which they characterise as those killed by some constant coefficient differential operator). The direct sum decomposition is \cite[Theorem 1]{BrionVergne}, which in turn defines the projection to $S_\Delta$.

We choose the sign to make $\pm \;o$ positively oriented with respect to their chosen orientation on $\ft_\Q$. (Flipping the orientation changes the sign of both of the last two arrows, leaving $\JK^\eta_\xi$ unchanged.) Finally, $\JK^{-\eta}_{-\xi}=(-1)^r\JK^\eta_\xi$ follows from \eqref{eq:def_JK_r} on replacing all $\gamma_i$ by $-\gamma_i$.
    \end{proof}
    
    For later use we record the following immediate consequence of \eqref{eq:def_JK_r}.
    
    \begin{proposition}\label{pro:vanishing_property}
        Let $f \in \widehat{R}_\Gamma$ be regular away from $\bigcup_{i=1}^\ell \lbrace \gamma_i =0 \rbrace \subset \ft_\C$, and choose the sign of each $\gamma_i \in \pm \Gamma$ so that $\langle \gamma_i, \eta \rangle >0$. If $\xi\not\in\cone(\gamma_1, \dots, \gamma_\ell)$, then $\JK^\eta_\xi(f) = 0$. In particular, if $\langle \xi, \eta \rangle <0$ we have $\JK^\eta_\xi=0$.
    \end{proposition}

The classical residue vanishes on functions of one variable which have no poles. There are analogous results for JK residues; we give three very similar variants.

    \begin{proposition}\label{pro:equality_of_residues}
    Given finitely many rational functions $f_i\in R_\Gamma$ and $\xi_i\in\ft^\vee_\Q\setminus\cH^\vee$, suppose the sum $f=\sum_if_ie^{\xi_i}\in\widehat R_\Gamma$ has no poles, i.e. it lies in $\prod_{d\ge0}\Sym^d\ft^\vee_\C$. Then
\beq{1}
    \sum\nolimits_i\JK^\eta_{\xi_i}\!\(f_ie^{\xi_i}\)\=\sum\nolimits_i\JK^{-\eta}_{\xi_i}\!\(f_ie^{\xi_i}\).
\eeq
Similarly if $f_s=\sum_if_ie^{s\xi_i}\in\widehat R_\Gamma$ has no poles for any $s\in\Q$, then
\beq{2}
    \sum\nolimits_i\JK^\eta_{\xi_i}(f_i)\=\sum\nolimits_i\JK^{-\eta}_{\xi_i}(f_i).
\eeq
If $f=\sum_if_ie^{\xi_i}\in\widehat R_\Gamma$ has no poles but we drop the $\xi_i\not\in\cH^\vee$ condition, choose $\eps\in\ft^\vee_\Q\setminus\cH^\vee$ small enough that the segments $\{\xi_i+s\eps\}_{s\in(0,1]}\subset\ft^\vee_\Q\setminus\cH^\vee$ instead. Then \eqref{1} is replaced by
\beq{3}
   \sum\nolimits_i\JK^\eta_{\xi_i+\eps}\!\(f_ie^{\xi_i}\)\=\sum\nolimits_i\JK^{-\eta}_{\xi_i+\eps}\!\(f_ie^{\xi_i}\).
\eeq
         \end{proposition}

\begin{proof}
Consider the linear subspace $E_\Gamma \subset \widehat{R}_\Gamma$ spanned by functions of the form $f e^{\xi}$ for $f \in R_\Gamma$ and $\xi \in \ft^\vee_\Q \setminus \cH^\vee$. On $E_\Gamma$ we package together the residues $\JK^\eta_\gamma$ for varying $\xi$ and fixed $\eta$, defining 
        \begin{equation}\label{eq:global_residue}
        \res^\eta \,:\, E_\Gamma \To \C \quad\text{by}\quad  \res^\eta(fe^\xi)\ :=\ \JK^\eta_\xi(f e^\xi)
    \end{equation}
on generators and extending by linearity. 

       Let $\Lambda$ denote the chamber of $\ft_\Q\setminus\cH$ to which $\eta$ belongs. By comparing the definition \eqref{eq:def_JK_r} of $\JK^\eta_\xi$ to Jeffrey-Kirwan's characterisation \cite[Proposition 3.2]{JK_quantisation} of their operator $\res^\Lambda$ we find it is essentially the same as $\res^\eta$,
$$
    \res^\eta(f) \= i^r \res^\Lambda\!\(\tilde{f}\;\),\ \text{ where }\ \tilde{f}(z)\,:= \,f(iz).
$$       
When $\tilde f$ has no poles, $\res^\Lambda(\tilde f)=\res^{-\Lambda}(\tilde f)$ by \cite[Lemma 3.3]{JeffreyKogan}, so $\res^\eta(f)=\res^{-\eta}(f)$. This proves \eqref{1}. To deduce \eqref{2} we apply \eqref{1} to $f_s$ to give
$$
    \sum\nolimits_i\JK^\eta_{s\xi_i}\!\(f_ie^{s\xi_i}\)\=\sum\nolimits_i\JK^{-\eta}_{s\xi_i}\!\(f_ie^{s\xi_i}\).
$$
Since $\JK^\eta_\xi$ depends only on the chamber of $\ft_\Q^\vee \setminus \cH^\vee$ that $\xi$ lies in we may replace each $\JK^{\pm\eta}_{s\xi_i}$ by $\JK^{\pm\eta}_{\xi_i}$. Then taking $\lim_{s\to0^+}$ gives \eqref{2}.

Finally in case \eqref{3} we apply \eqref{1} to $fe^{s\eps}$ to give
$$
    \sum\nolimits_i\JK^\eta_{\xi_i+s\eps}\!\(f_ie^{\xi_i+s\eps}\)\=\sum\nolimits_i\JK^{-\eta}_{\xi_i+s\eps}\!\(f_ie^{\xi_i+s\eps}\).
$$
Replacing each $\JK^{\pm\eta}_{\xi_i+s\eps}$ by $\JK^{\pm\eta}_{\xi_i+\eps}$ then taking $\lim_{s\to0^+}$ gives \eqref{3}.
    \end{proof}
    

    \section{Jeffrey--Kirwan localisation}\label{section:JK_loc}
 We are finally ready to prove the abelian virtual JK localisation formula in the perfect obstruction theory setting. We follow the philosophy of \cite{JeffreyKogan} with some simplifications (from Lemma \ref{lem:key_property_oLambda}) and many complications. So we fix the standard data of Definition \ref{data},
$$
\sT\,\acts\,\(X,L,\EE,\phi\),\ \,\alpha\,\in\,H^*_\sT(X).
$$
In place of the auxiliary data $\Sigma\subset\ft^\vee_\Q$ of Definition \ref{data}\footnote{Later the $\Sigma$ of Definition \ref{data} will be a cone approximating the half space $\{\eta\ge0\}$.} we choose $\eta\in\ft_\Q\setminus\cH$, where $\cH \subset \ft_\Q$ is the hyperplane arrangement \eqref{eq:hyperplanes} defined by the set $\Gamma \subset \ft_\Q^\vee$ of $\sT$-weights of $N^{\vir}_{X^\sT/X}$. For simplicity we make the following assumptions (which will be weakened below),
\begin{itemize}
\item[(a)] $\langle\mu(F), \eta\rangle\ne0$ for all components $F\subseteq X^\sT$ and $\mu(F)\in\ft^\vee_\Q\setminus\cH^\vee$ if $\langle\mu(F), \eta\rangle>0$,
\item[(b)] $X$ is $\eta$-semiprojective (Definition \ref{etasp}).
\end{itemize}

    \begin{figure}[h]
    \centering
    \includegraphics[width=.4\textwidth, keepaspectratio]{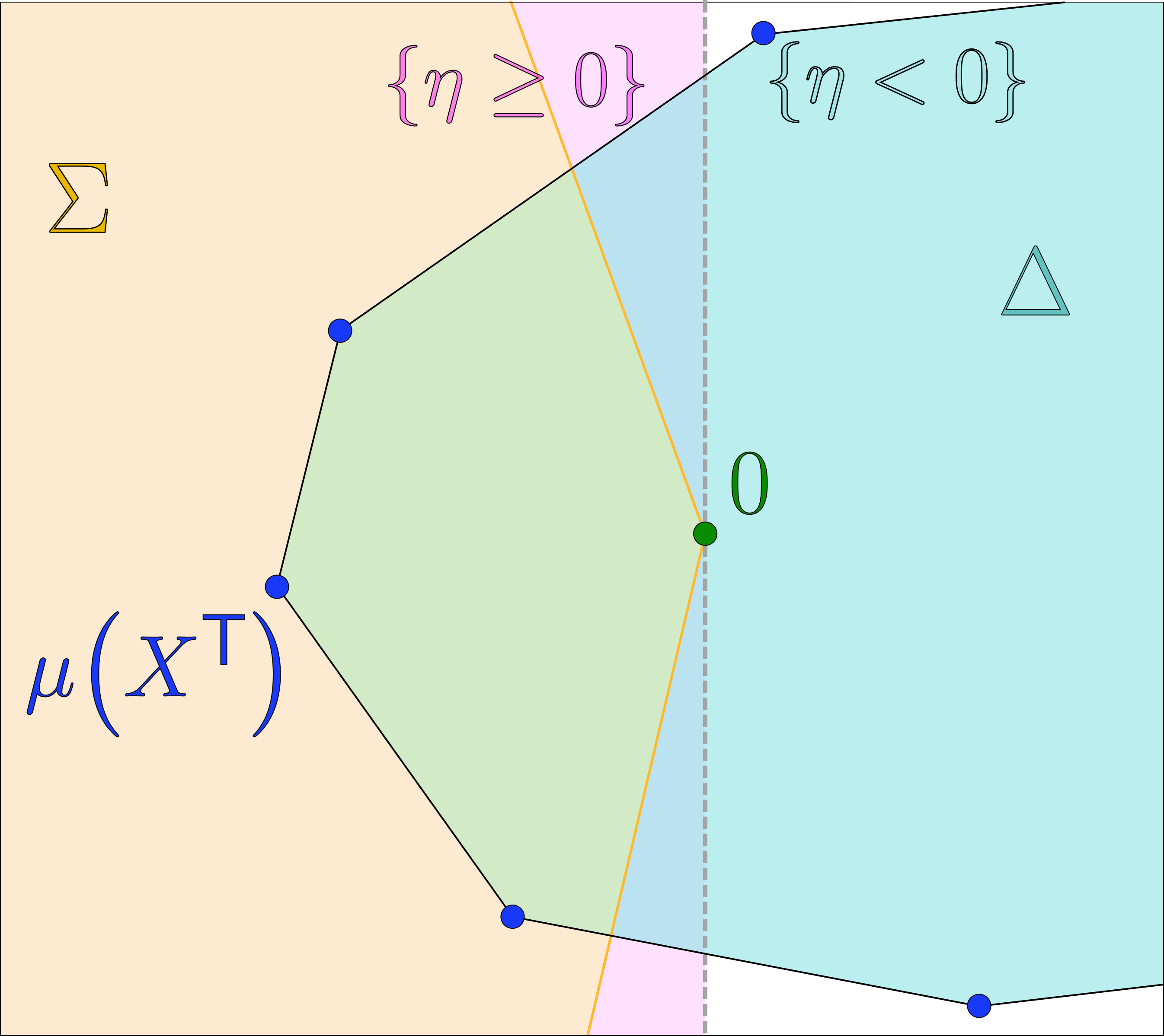}
    \caption{An $\eta \in \ft_\Q$ such that $\lbrace \eta\geq 0 \rbrace \cap \Delta$ is bounded and $0 \notin \eta(\mu(X^\sT))$. The $\eta$-wide cone $\Sigma$ approximates the half-space $\lbrace\eta>0\rbrace$ as far as $\mu(X^\sT)$ is concerned.}
    \label{fig:eta}
    \end{figure}  
    
    \begin{theorem}[Abelian virtual JK localisation, first version]\label{th:abelian_JK}
For $\eta$ satisfying {\rm (a),\,(b),}
\beq{simples}
            \int_{[X\div\;\sT]^{\vir}} \alpha\_0e^{c_1(L_0)}\ = \sum_{\substack{F\subseteq X^\sT\\\
            \langle \mu(F), \eta \rangle >0}} \JK^{\eta}_{\mu(F)}\int_{[F]^{\vir}} \frac{\alpha e^{c_1^\sT(L)}\vert_F}{e^\sT(N^{\vir}_{F/X})}\,.
\eeq
    \end{theorem}
    We can relax (b) to $(X,L)$ being weakly $\eta$-semiprojective (Definition \ref{weaksemi}) to make it look more like the assumption of Definition \ref{data}. And we can drop (a) completely if we perturb by a small character $\eps \in \ft^\vee_\Q$ to ensure $\mu(F) \notin \cH^\vee$. The upshot is a more general version of the abelian JK localisation theorem.

    \begin{theorem}[Abelian virtual JK localisation]\label{th:abelian_JK_general}
Assume $(X,L)$ is weakly $\eta$-semiprojective and let $\epsilon \in\ft^\vee_\Q$ be sufficiently small that the segments $\{\mu(F) + s\epsilon\}_{s \in(0,1],\,F\subseteq X^\sT}$ are disjoint from $\cH^\vee$. Then
\beq{JKab}
            \int_{[X\div\;\sT]^{\vir}} \alpha\_0\ = \sum_{\substack{F\subseteq X^\sT\\\
            \langle \mu(F), \eta \rangle >0}} \JK^{\eta}_{\mu(F)+\epsilon}\int_{[F]^{\vir}} \frac{\alpha\vert_{F}}{e^\sT(N^{\vir}_{F/X})}\,.
\eeq
    \end{theorem}
    
We can recover Theorem \ref{th:abelian_JK} from this by replacing $\alpha$ by $\alpha e^{c_1^\sT(L)}$.
The rest of this Section will be devoted to proving Theorem \ref{th:abelian_JK}. Then Theorem \ref{th:abelian_JK_general} will follow by some modifications and perturbations.

\subsection{The $\eta$-wide cone}\label{wide}
We begin by forming an algebraic cut $(X_\Sigma,L_\Sigma)$. Choose a full-dimensional simplicial cone $\Sigma\subset\{\eta\ge0\}\subset\ft^\vee_\Q$ satisfying the following conditions (illustrated in Figure \ref{fig:eta}). Note that since we are working over $\Q$, the generic cone satisfying the first two (open) conditions will satisfy the third.
\begin{itemize}
\item $\eta\big|_{\Sigma \setminus \{0\}}\,>\,0$,
\item $\Sigma$ is \emph{$\eta$-wide}: it satisfies $\langle \mu(F), \eta \rangle>0\iff\mu(F) \in\Sigma$ for every component $F\subseteq X^\sT\!$,
\item $\Sigma$ is transverse to $\Delta$ as in Definition \ref{def:transversal}.
\end{itemize}\smallskip

Let $p$ be a point of $\mu(X_\Sigma^\sT) \setminus \mu(X^\sT)$\,---\,so one of the red or green points in Figure \ref{fig:classification_fix}. Recall from Proposition \ref{pro:fixed_locus}
        that $p=\sigma\cap\partial^{\sm}_{r-\dim\sigma\;}\Delta=\sigma\cap\tau_p$ for some face $\sigma\subset\Sigma$ and the affine space $\tau_p$ of Definition \ref{def:tau}. Locally, about $p$, the intersection $\Sigma \cap \tau_p$ is a simplicial cone in $\tau_p$. The following crucial property of $\eta$-wide cones\,---\,illustrated in Figure \ref{fig:maximum}\,---\,is the reason we use them.

    \begin{lemma}\label{lem:key_property_oLambda}
At least one of the generators $\nu$ of $\Sigma \cap \tau_p$ emanating from $p$ satisfies $\langle\nu, \eta\rangle>0$. 
    \end{lemma}

    \begin{proof}
Note $\tau_p$ is positive dimensional because $p\in\mu(X_\Sigma^\sT)$ is \emph{not} in $\mu(X^\sT)$. There is an irreducible component $Y \subseteq \mu^{-1}(\tau_p)_{\red}$ such that $p \in \mu(Y)$. By transversality $p$ is a point of $\partial^{\sm}_{\dim\tau_p\,}\mu(Y)$, so by 
Lemma \ref{lem:interior_polytope} it is in the interior of 
\begin{figure}[h]
        \centering
        \includegraphics[width=.6\textwidth, keepaspectratio]{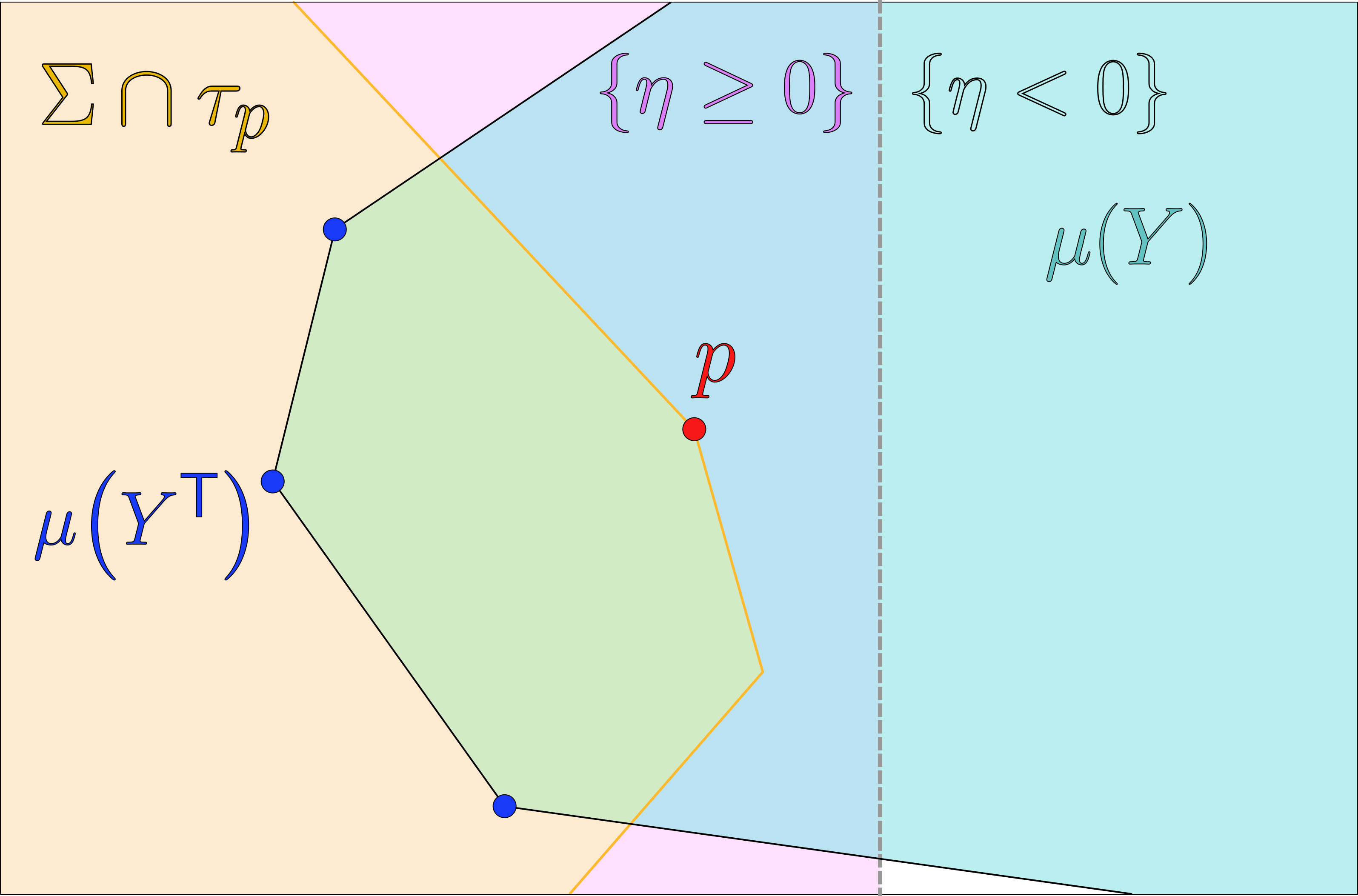}
        \caption{
        This picture is in the affine subspace $\tau_p \subseteq \ft^\vee_\Q$.
        The intersection $\Sigma \cap \mu(Y)\cap\tau_p$ lies in $\lbrace \eta \geq 0\rbrace$, as does at least one of its generating directions at $p$  (but not, in this example, the other).}
        \label{fig:maximum}
        \end{figure}    
$\mu(Y)$. Now $\mu(Y)$ is a full dimensional polyhedron in $\tau_p$ whose vertices lie in $\mu(Y^\sT)$ by Lemma \ref{lem:polyhedron_structure} (which applies because we claim $H^0(\cO_Y)^\sT=\C$: it is contained in $\(H^0(\cO_Y)\otimes\Sym\Sigma^\vee_\C\)^\dT$ and this is $\C$ by Proposition \ref{pro:proj_criterion}\;(6) because $\emptyset\ne Y_\Sigma\subset X_\Sigma$ is projective).
    
     In particular $\mu(Y)\cap\{\eta\ge0\}$ is a (compact!) polytope, on which (the restriction of)  $\eta$ attains its maximum at a vertex $v \in \mu(Y^\sT)$ (it is one of the blue points in Figure \ref{fig:maximum}). And $v \in \mathring\Sigma$ because $\Sigma$ is $\eta$-wide.
     
     But $p$ (the red point in Figure \ref{fig:maximum}) is a vertex of the different polytope $\Sigma \cap \tau_p$. So $p$ and $v$ are distinct vertices of $\Sigma \cap \mu(Y)$ and $\eta$ increases along the interval from $p$ to $v$. Since the direction of this interval is in the convex hull of the directions of the rays of $\Sigma \cap \mu(Y)$ emanating from $p$, this shows $\eta$ must increase along at least one of those rays.
    \end{proof}

So by Theorem \ref{thm:structure_cut}\;(a) the algebraic cut $(X_\Sigma,L_\Sigma)$ is a Deligne-Mumford stack. By our assumption of $\eta$-semiprojectivity and Proposition \ref{pro:eta_proj_compact} it is projective. By Corollary \ref{cor:obs_cut} it has a $\sT$-equivariant perfect obstruction theory $(\EE_\Sigma,\phi_\Sigma)$.

So we may apply the localisation formula \eqref{eq:general_loc_cut} to $\alpha\_\Sigma e^{c_1^\sT(L_\Sigma)}$ on $X_\Sigma$.
Using $c_1^\sT(L_\Sigma)|_{F_p}=p+c_1(L_\Sigma)|_{F_p}$ we obtain
\beq{4}
\int_{[X_\Sigma]^{\vir}} \alpha\_\Sigma e^{c_1^\sT(L_\Sigma)}\=\sum_{p \in \mu(X_\Sigma^\sT)=\partial_{0}(\Sigma \cap \Delta)} e^p I_p\;,
\eeq
where the $I_p$ are the rational functions described in Propositions \ref{pro:special_old_contributions} and \ref{pro:weights_normal_bundle},
$$
        I_{p}\= \int_{[F_p]^{\vir}} \frac{\alpha\_\Sigma \,e^{c_1(L_\Sigma)}\vert_{F_p}}{e^\sT\(N^{\vir}_{F_p/X_\Sigma}\)} \,:\, \ft_\Q \xymatrix{\ar@{-->}[r]& \ \Q.}
$$

\subsection{Residues}
We will take JK residues of \eqref{4} to prove the JK formula \eqref{simples}. Set $\Gamma_\Sigma\subset\ft^\vee_\Q$ to be the set of weights of $\sT\acts N^{\vir}_{X_\Sigma^\sT/X_\Sigma}$, defining new hyperplane arrangements $\cH_\Sigma\subset\ft_\Q$ and $\cH^\vee_\Sigma\subset\ft^\vee_\Q$ as in \eqref{eq:hyperplanes}. Since $\mu(X^\sT_\Sigma)$ may now intersect $\cH^\vee_\Sigma$ we choose an auxiliary $\eps\in-\mathring\Sigma\setminus\cH_\Sigma^\vee$ in minus the interior of $\Sigma$ to perturb by.  By taking it sufficiently small and also picking a small perturbation of $\eta$ if necessary\,---\,while keeping $\Sigma$ fixed\,---\,we may assume
\begin{enumerate}
\item[(i)] $\eta\not\in\cH_\Sigma$ and $\eps\in-\mathring\Sigma\setminus\cH_\Sigma^\vee$,
\item[(ii)] $\{\mu(F)+s\eps\}_{s\in(0,1]}$ lies in $\{\eta>0\}\setminus\cH_\Sigma^\vee$ for all $F\subseteq X_\Sigma^\sT$ except $F_0$,
\item[(iii)] $\Sigma$ remains $\eta$-wide,
\item[(iv)] the open condition (a) continues to hold, and
\item[(v)] Lemma \ref{lem:key_property_oLambda} continues to hold.
\end{enumerate}
Since \eqref{4} has no poles, equation \eqref{3} of Proposition \ref{pro:equality_of_residues} then gives,
  \beq{finally}
    \sum\nolimits_{p\in\mu(X_\Sigma^\sT)}\JK^\eta_{p+\eps}(e^pI_p)\=\sum\nolimits_{p\in\mu(X_\Sigma^\sT)}\JK^{-\eta}_{p+\eps}(e^pI_p).
   \eeq

    In the next two Propositions, we find that the two sides of \eqref{finally} compute the two sides of the JK formula \eqref{simples} up to a factor $\big|\wSigma\!\big|^{-1}$.
    
    \begin{proposition}
The left hand side of \eqref{finally} is
        \begin{align*}
\frac{1}{\vert \wSigma \vert}\sum_{\substack{F\subseteq X^\sT\\\
            \mu(F) \in \Sigma}} \JK^{\eta}_{\mu(F)}\left(e^{\mu(F)}\int_{[F]^{\vir}} \frac{\alpha e^{c_1(L)}\vert_{F}}{e^\sT(N^{\vir}_{F/X})}\right).
        \end{align*}
    \end{proposition}
    \begin{proof}
For $p\in\Sigma\cap\partial_0\Delta$\,---\,an old (blue) fixed locus in Figure \ref{fig:classification_fix}\,---\,the integral $I_p$ is evaluated in \eqref{old}. Its poles lie in $\cH\not\ni p$ by (a). So, by (ii), $\JK^\eta_{p+\eps}=\JK^\eta_p$ 
when applied to $e^pI_p$, giving precisely the formula above.

So we are left with 
showing the vanishing of the contribution of any point $p\in \mu(X_\Sigma^\sT) \setminus \mu(X^\sT)$\,---\,i.e. any red point $p$ or the green point $p=0$.

    Write the fixed locus corresponding to $p$ as the cut $F_p \cong (\mu^{-1}(\tau_p))_\sigma$ as in Proposition \ref{pro:fixed_locus}\;\eqref{cond:3}, and note that $\sigma \neq \Sigma$ and $\rspan(\sigma)$ is complementary to the tangent space of $\tau_p$ at $p$.
        
By Proposition \ref{pro:weights_normal_bundle} the weights of $N^{\vir}_{F_p/X_\Sigma}$ are either
        \begin{enumerate}
        \item $-\nu_1, \dots, -\nu_{r-d}$, where $\nu_1, \dots, \nu_{r-d}$ are generators of the simplicial cone which, locally around $p$, is $\Sigma \cap \tau_p$, or
        \item in $\mathrm{span}(\sigma)$.
    \end{enumerate}
Since $p\in\sigma\subset\mathrm{span}(\sigma)$ and $\eps\in-\mathring\Sigma$,
$$
p+\eps\ \in\ \mathrm{span}(\sigma)\ -\ \sum\nolimits_{j=1}^{r-d}a_j\nu_j\ \text{ for some }\,a_j\,\in\,\Q_{>0}.
$$
    Define $\overline{\nu}_j := \pm \nu_j$, with the sign chosen so that $\langle\overline{\nu}_j, \eta \rangle >0$. That sign is positive for at least one $j$ by Lemma \ref{lem:key_property_oLambda}, so
$$
p+\eps\ \in\ \mathrm{span}(\sigma)\ +\ \sum\nolimits_{j=1}^{r-d}b_j\overline\nu_j\ \text{ with at least one }\,b_j\,<\,0.
$$
The $b_j$ are unique since generators of $\mathrm{span}(\sigma)$ plus the generators $\nu_i$ together form a $\Q$-basis of $\ft^\vee_\Q$ by Proposition \ref{pro:fixed_locus}\;\eqref{cond:fix}. So if we denote by $\gamma_i$ the weights of $N^{\vir}_{F_p/X_\Sigma}$ and by $\overline\gamma_i=\pm\gamma_i$ the normalised weights chosen so that $\langle\overline\gamma_i,\eta\rangle>0$, then we have shown that $p+\eps$ cannot be written as a positive linear combination of the $\overline\gamma_i$. Thus by Proposition \ref{pro:vanishing_property},
\begin{equation*}
\JK^\eta_{p+\eps}\!\left(e^p\int_{[F_p]^{\vir}} \frac{\alpha\_\Sigma \,e^{c_1(L_\Sigma)}\vert_{F_p}}{e^\sT\(N^{\vir}_{F_p/X_\Sigma}\)}\right)\,=\ 0.\qedhere
\end{equation*}
    \end{proof}

    \begin{proposition}
The right hand side of \eqref{finally} is
$$
\frac1{\vert \wSigma \vert} \int_{[X\div\;\sT]^{\vir}} \alpha\_0 e^{c_1(L_0)}.
$$
\end{proposition}
    \begin{proof}
        For  $p\in\mu(X_\Sigma^\sT)\setminus\{0\}\subset\Sigma \setminus \lbrace 0 \rbrace$ we have $\langle \eta,p+\eps \rangle>0$ by (ii). Therefore  $\JK^{-\eta}_{p+\eps}(e^pI_p)=0$ by Proposition \ref{pro:vanishing_property}.
        
        For $p=0$ its contribution $I_0$ to the localisation formula \eqref{4} was computed in \eqref{eq:special_contribution},\,\eqref{eq:origin_contribution},
               \begin{align*}
            I_0 \= \sum_{\ell_1, \dots, \ell_r = 0}^{\infty} \frac{(-1)^r}{\psi_1^{\ell_1+1}\cdots \psi_r^{\ell_r+1}}\int_{[X\div\;\sT]^{\vir}} \alpha\_0 e^{c_1(L_0)}\prod_{j=1}^r c_1(N_j)^{\ell_j}.
        \end{align*}
The identity $\JK^{-\eta}_{\eps}=(-1)^r\JK^\eta_{-\eps}$ from Proposition \ref{JKpm} now gives
\beqa
            \JK_{\eps}^{-\eta}(I_0) &=& (-1)^r\JK_{-\eps}^{\eta}\left(\frac{(-1)^r}{\prod_{j=1}^r \psi_j}\int_{[X\div\;\sT]^{\vir}} \alpha\_0 e^{c_1(L_0)}\right)\\ &=&  \frac{1}{\big|\wSigma\big|}\int_{[X\div\;\sT]^{\vir}} \alpha\_0 e^{c_1(L_0)}\,
\eeqa
by the definition \eqref{eq:def_JK_r} of $\JK^\eta_{-\eps}$, using that the generators $\psi_i$ of the cone $\Sigma$ are integral with $\psi_1\wedge\dots\wedge\psi_r=\pm\big| \wSigma\!\big|\;o$, and $-\eps\in\Sigma$.
    \end{proof}

So we have proved Theorem \ref{th:abelian_JK}, but with $\JK^\eta_{\mu(F)}$ replaced by $\JK^\eta_{\mu(F)+\eps}$ in the formula \eqref{simples}. However, by (a) and (ii), $\mu(F)$ and $\mu(F)+\eps$ are in the same chamber with respect to the (dual) hyperplane arrangement defined by the weights of $N^{\vir}_{F/X}$ for those components $F\subseteq X^T$ with $\langle\eta,\mu(F)\rangle>0$. So Proposition \ref{pro:vanishing_property} finishes the proof. \medskip

Then replacing $e^{c_1^\sT(L)}$ by $e^{sc_1^\sT(L)}$ for $s\ge0$ and taking $\lim_{s\to0}$\,---\,or replacing our use of \eqref{3} in the above proof by the simpler \eqref{2}\,---\,we obtain
\beq{semisimple}
            \int_{[X\div\;\sT]^{\vir}} \alpha\_0\ = \sum_{\substack{F\subseteq X^\sT\\\
            \langle \mu(F), \eta \rangle >0}} \JK^{\eta}_{\mu(F)}\int_{[F]^{\vir}} \frac{\alpha\vert_F}{e^\sT(N^{\vir}_{F/X})}\,.
\eeq

\subsection*{Weakening $\eta$-semiprojectivity}
Notice we only used $\eta$-semiprojectivity to deduce the projectivity of $X_\Sigma$ from Theorem \ref{thm:structure_cut}\;(c). If instead we assume only that $(X,L)$ is weakly $\eta$-semiprojective then 
$\Sigma \cap \Delta \subset \lbrace \eta\geq 0 \rbrace \cap \Delta$ is bounded and the composition
$$
    H^0\bigl(X,\cO_X\bigr)^\sT \To H^0\bigl(X^{\lbrace \eta\ge0 \rbrace -\ss},\,\cO_{X^{\lbrace \eta\ge0 \rbrace -\ss}}\bigr)^\sT \To H^0\bigl(X^{\Sigma -\ss},\,\cO_{X^{\Sigma -\ss}}\bigr)^\sT
$$
has finite rank. Thus Theorem \ref{thm:structure_cut}\;(c) still applies. 

\subsection*{Weakening (a)}
So we have proved \eqref{semisimple} assuming only weak $\eta$-semiprojectivity and (a). Next we assume weak $\eta$-semiprojectivity and the following weakening of (a),
\begin{itemize}
\item[(a$'$)] $\langle\mu(F), \eta\rangle\ne0$ for all components $F\subseteq X^\sT$.
\end{itemize}
That is, we drop the assumption that $\mu(F)\not\in\cH^\vee$ if $\langle\mu(F), \eta\rangle>0$ from (a). We then restore it by replacing $(X,L)$ by $(X,L\otimes\epsilon)$ for some character\footnote{We should really take an integral power $(L\otimes\eps)^N$; since $\JK^\eta_{N\xi}=\JK^\eta_{\xi}$ this changes nothing in \eqref{compl}.} $\epsilon\in\ft^\vee_\Q\setminus\cH^\vee$, thus replacing each $\mu(F)$ by $\mu(F)+\eps$. For $\eps$ sufficiently small this preserves
\begin{itemize}
\item the open condition (a$'$)\,---\,thereby promoting it to (a),
\item the indexing set $\{F\subseteq X^\sT:\langle\mu(F),\eta\rangle>0\}$ of the sum in \eqref{semisimple},
\item the open $\sT$-GIT chamber containing $L$, so the GIT quotients $X\div\,\sT$ formed using $L$ or $L\otimes\epsilon$ are the same,
\item weak $\eta$-semiprojectivity by the openness proved in Proposition \ref{weak2}.
\end{itemize}
Thus (a) and weak $\eta$-semiprojectivity hold for $(X,L\otimes\eps)$ and \eqref{semisimple} gives
\beq{compl}
            \int_{[X\div\;\sT]^{\vir}} \alpha\_0\ = \sum_{\substack{F\subseteq X^\sT\\\
            \langle \mu(F), \eta \rangle >0}} \JK^{\eta}_{\mu(F)+\eps}\int_{[F]^{\vir}} \frac{\alpha\vert_F}{e^\sT(N^{\vir}_{F/X})}\,.
\eeq

\subsection*{Dropping (a$'$)}
Finally we drop the assumption (a$'$) to prove the same formula (and so Theorem \ref{th:abelian_JK_general}) assuming only weak $\eta$-semiprojectivity.
Any components $F\subseteq X^\sT$ which violate (a$'$)\,---\,i.e. those with $\langle\mu(F),\eta\rangle=0$\,---\,we call ``bad". We perturb $\eta$ to $\eta'$ in the same chamber of $\ft^\vee_\Q\setminus\cH$ so that $\langle\mu(F),\eta'\rangle\not=0$ for all bad $F$
while preserving weak $\eta$-semiprojectivity by the openness of Proposition \ref{weak2}. Thus (a$'$) now holds, giving a perturbed formula \eqref{compl} with $\eta$ replaced by $\eta'$.

Since $\JK^\eta_\xi$ is locally constant in $\eta$ the only change to the original formula \eqref{compl} is the appearance of new terms on the right hand side from bad $F$ for which $\langle\mu(F),\eta'\rangle>0$ after perturbation; we claim these contribute zero.

Indeed let $\gamma_i\in\ft^\vee_\Q$ be all the weights of $N^{\vir}_{X^\sT/X}$ defining the hyperplane arrangement $\cH$, and let $\overline\gamma_i=\pm\gamma_i$ be the $\eta$-normalised weights with sign fixed so that $\langle\eta,\overline\gamma_i\rangle>0$. Then since $\eta,\,\eta'$ lie in the same chamber of $\ft^\vee_\Q\setminus\cH$, the $\overline\gamma_i$ are also $\eta'$-normalised. Since bad $F$ have $\langle\mu(F),\eta\rangle=0$ their $\mu(F)$ are disjoint from the cone generated by the $\overline\gamma_i$, so the same is true of $\mu(F)+\eps$ for sufficiently small $\eps$. Therefore
$$
\JK^{\eta'}_{\mu(F)+\eps}\int_{[F]^{\vir}} \frac{\alpha\vert_F}{e^\sT(N^{\vir}_{F/X})}\=0\,\text{ for bad $F$ and }\,|\eps|\,\ll\,1
$$
by Proposition \ref{pro:vanishing_property}. This proves \eqref{compl} with $\JK^\eta$ replaced by $\JK^{\eta'}$ and $|\eps|\ll1$. Deforming back to $\JK^\eta$ and the original $\eps$ without changing chambers proves Theorem \ref{th:abelian_JK_general}.
     
\section{Nonabelian Jeffrey--Kirwan localisation}\label{section:NA_JK_loc}
In this Section we fix an action $\sG\acts(X,L,\EE,\phi)$ on a polarised projective-over-affine Deligne-Mumford stack with perfect obstruction theory. Let $\sT\subset \sG$ denote a maximal torus, $W$ be the Weyl group and let $\cH\subset\ft_\Q$ and $\cH^\vee\subset\ft_\Q^\vee$ be the hyperplane arrangements \eqref{eq:hyperplanes} defined by the weights of $\sT \curvearrowright N^{\vir}_{X^\sT/X}$. As in Theorem
\ref{th:abelian_JK_general} choose $\eta \in \ft_\Q \setminus \cH$ such that
        \begin{itemize}
        \item $(X,L)$ is weakly $\eta$-semiprojective (Definition \ref{weaksemi}) and
                \item $\sT$-semistable points of $X$ are $\sT$-stable.
    \end{itemize}

\noindent We will show these conditions imply that $\sG$-semistable points are also $\sG$-stable and

\begin{lemma}\label{check}
$X\div\;\sT$ and $X\div\;\sG$ are projective Deligne-Mumford stacks.
\end{lemma}

    
\noindent
Fix $\alpha \in H^\ast_\sG(X)$. Via $H^\ast_\sG(X)\to H^\ast_\sT(X)$ it gives $\alpha^\sT\in H^\ast_\sT(X)$ and, by \eqref{induced_class}, $\alpha\_0\in H^\ast(X\div\sG)$.

    \begin{theorem}[Virtual JK localisation]\label{th:JK}
Choose $\epsilon \in \ft^\vee_\Q$ sufficiently small that the segments $\{\mu(F) + s\epsilon\}_{s \in (0,1],\,F\subseteq X^\sT}$ are disjoint from $\cH^\vee$. Then
$$
            \int_{[X\div \sG]^{\vir}} \alpha\_0 \= \frac{1}{\vert W \vert}\sum_{\substack{F\subseteq X^\sT\\\
            \langle \mu(F), \eta \rangle >0}} \JK^{\eta}_{\mu(F)+\eps}\left(e^\sT(\fg_\C/\ft_\C) \cdot\int_{[F]^{\vir}} \frac{\alpha^\sT\vert_{F}}{e^\sT(N^{\vir}_{F/X})}\right).
$$
        Here $\sT\acts\fg_\C/\ft_\C$ is the adjoint action so $e^\sT(\fg_\C/\ft_\C)=\prod\epsilon_i$ where $\epsilon_i \in \ft_\Z^\vee$ are the roots of $\sG$.
\end{theorem}

As in the abelian case we can replace $\alpha$ by $\alpha e^{c_1^\sG(L)}$ throughout. We will deduce Theorem \ref{th:JK} from Theorem \ref{th:abelian_JK_general} using the diagram \eqref{eq:abelianisation_diagram},
\beq{diagram}
\begin{tikzcd}[row sep=18pt]
X^{\sG-\ss}/\;\sT\ \ar[r, hook, "i"]\ar[d, "p"]& X\div\;\sT \\
X\div\sG.\!
\end{tikzcd}
\eeq
Note $p$ is smooth (but not proper!) and $i$ an open embedding. Since both are flat we can  pull back Chow homology classes as well as cohomology classes. 
The following result of Martin \cite{Martin} and Maddock \cite{Maddock}, extended from projective varieties $X$ to projective-over-affine Deligne-Mumford stacks $X$, expresses integrals over $X\div\sG$ in terms of integrals over $X\div\;\sT$.

\begin{theorem}\label{martin}
If $\;\widetilde{\gamma} \in A_\ast(X\div\;\sT)$ lifts $\gamma \in A_\ast(X\div\sG)$ in the sense that $i^*\widetilde\gamma=p^*\gamma$ then 
\beq{eq:martin}
        \deg(\gamma) \= \frac{1}{\vert W \vert} \deg\bigl( e(\op{Ad}) \cap \widetilde{\gamma}\bigr),\ \text{ where }\ \op{Ad}\,:=\,X^{\sT-\ss} \times\_\sT (\fg_\C/\ft_\C).
\eeq
Here $\sT\acts(\fg_\C/\ft_\C)$ via the adjoint action and $\deg$ denotes the pushforward to a point.
\end{theorem}

We combine this with Theorem \ref{th:abelian_JK_general} to prove Theorem \ref{th:JK} as follows. By Proposition \ref{pro:compatibility_of_obs},
		\begin{align*}
			p^\ast [X\div\sG]^{\vir} \= i^\ast [X\div\;\sT]^{\vir}\ \text{ in }\ A_\ast(X^{\sG-\ss}/\;\sT).
		\end{align*}
		 Choose $\alpha \in H^\ast_\sG(X)$ inducing $\alpha^\sT\in H^\ast_\sT(X)$ and, by \eqref{induced_class}, $\alpha\_0\in H^\ast(X\div\sG)$ and $\alpha_0^\sT \in H^\ast(X\div\;\sT)$. Since these satisfy $p^*\alpha\_0= i^*\alpha_0^\sT$,
$$
i^*\(\alpha_0^\sT)\cap[X\div\;\sT]^{\vir}\)\=p^*\(\alpha_0\cap[X\div\sG]^{\vir}\),
$$
so we can apply Theorem \ref{martin} to deduce the required formula,
$$
\int_{[X\div \sG]^{\vir}}\!\alpha\_0\ \stackrel{\eqref{eq:martin}}=\ \frac1{|W|}\!\int_{[X\div \sT]^{\vir}}\!\! e(\op{Ad})\cup\alpha^\sT_0\ \stackrel{\eqref{JKab}}=\ \frac1{|W|}\hspace{-3mm}\sum_{\substack{F\subseteq X^\sT\\\
            \langle \mu(F), \eta \rangle >0}}\hspace{-3mm} \JK^{\eta}_{\mu(F)+\eps}\int_{[F]^{\vir}}\!\! \frac{e^\sT(\fg_\C/\ft_\C)\alpha^\sT\vert_{F}}{e^\sT(N^{\vir}_{F/X})}\,.
$$
So after proving Lemma \ref{check} the rest of this Section will be devoted to showing we can deduce Theorem \ref{martin} from Maddock's proof of it for projective varieties.

\subsection*{Proof of Lemma \ref{check}}
To show $\sG$-semistable points are stable we will show the stabiliser group $\sG_x\subseteq\sG$ of any $\sG$-polystable point $x$ is finite. By Matsushima $\sG_x$ is reductive, so it is sufficient to show its maximal torus $\sT(\sG_x)$ is trivial. We may assume $\sT(\sG_x)$ lies in the maximal torus $\sT\subset\sG$ on replacing $x$ by some $g\cdot x$, because this conjugates $\sG_x$ to $\sG_{gx}=g\sG_xg^{-1}$. So $\sT(\sG_x)\subseteq\sT_x$, which is finite. Thus both $X\div\,\sT$ and $X\div\sG$ are Deligne-Mumford stacks.

By Proposition \ref{pro:proj_criterion} weak $\eta$-semiprojectivity implies that $X\div\;\sT$ is projective and the restriction map $r$ in the diagram
$$
\begin{tikzcd}[column sep=20pt]
    H^0(X,\cO_X)^\sG \arrow[rr, "R"]\arrow[d, hook'] && H^0(X^{L,\sG-\ss},\cO_{X^{L,\sG-\ss}})^\sG\arrow[d, hook', shorten <= -.6mm]\\
    H^0(X,\cO_X)^\sT \arrow[r, "r"] & H^0(X^{L,\sT-\ss},\cO_{X^{L,\sT-\ss}})^\sT \arrow[r] & H^0(X^{L,\sG-\ss},\cO_{X^{L,\sG-\ss}})^\sT
\end{tikzcd}
$$
has finite rank. Thus $R$ does too, so $X\div\;\sG$ is also projective by Proposition \ref{pro:proj_criterion}.

\subsection*{From stacks to schemes}
Since we always work with rational coefficients, pullback along the projection  $\pi:X\to|X|$ from $X$ to its coarse moduli space induces isomorphisms   
$$
\pi^*\,:\,H^*(|X|)\rt\sim H^*(X)\ \text{ and }\ \pi_*\,:\,A_*(X)\rt\sim A_*(|X|)
$$
such that $\langle \pi^*a,\gamma\rangle=\langle a,\pi_*\gamma\rangle$. Moreover the diagram \eqref{diagram} maps commutatively to the corresponding diagram of coarse moduli spaces,
$$
\begin{tikzcd}[row sep=18pt]
{|X^{\sG-\ss}/\;\sT|\ }\ar[r, hook, "i"]\ar[d, "p", shorten <= -1mm]& {|X\div\;\sT|} \\
{|X\div\sG|.\!}
\end{tikzcd}
$$
Thus if $\widetilde\gamma$ lifts $\gamma$ in the sense that $p^*\gamma=i^*\widetilde\gamma$ then the same is true of the classes $\pi_*\widetilde\gamma,\,\pi_*\gamma$ on the coarse moduli spaces. So to prove Theorem \ref{martin} for projective-over-affine Deligne-Mumford stacks it is sufficient to prove it for projective-over-affine schemes. 

\subsection*{From schemes to varieties} Chow groups are insensitive to nonreduced structure, so it is sufficient to prove Theorem \ref{martin} for \emph{reduced} projective-over-affine schemes.

\subsection*{Irreducibility} We next show that to prove Theorem \ref{martin}  it is is sufficient to assume $X$ is irreducible. So write $X=\bigcup_jX_j$ as a finite union of irreducible projective-over-affine varieties, and assume Theorem \ref{martin} holds for each of them.

Write (non-canonically) $\gamma = \sum_{j=1}^n \gamma_j$ for $\gamma_j\in A_\ast(X_j\div \sG)$. To begin with choose the lifts $\overline{p^*\gamma_j}\in A_*(X\div\;\sT)$. By \eqref{eq:martin} for each of these,
$$
        \deg(\gamma) \= \sum_{j=1}^n \deg(\gamma_j) \= \frac{1}{\vert W\vert}\sum_{j=1}^n \deg\bigl(e(\op{Ad}) \cap \overline{p^*\gamma_j}\,\bigr) \= \frac{1}{\vert W\vert}\deg\Bigl(e(\op{Ad}) \cap \sum_{j=1}^n \overline{p^*\gamma_j} \Bigr).
$$
This is what we wanted to prove if the given lift $\widetilde\gamma$ is $\sum_j \overline{p^*\gamma_j}$. But even if not, their difference has $i^*=0$, so can be taken to lie in the complement of $X^{\sG-\ss}/\;\sT\subset X\div\;\sT$.
Writing it (non-canonically) as $\sum_j\beta_j$ with $i^*\beta_j=0$, another application of Theorem \ref{martin} to $X_j$ gives
\begin{equation*}
\deg\(e(\op{Ad})\cap\widetilde\gamma\)-\deg\left(e(\op{Ad})\cap\sum\nolimits_j \overline{p^*\gamma_j}\right)\=\sum\nolimits_j \deg\(e(\op{Ad})\cap\beta_j\)=0.
\end{equation*}

\subsection*{From projective-over-affine to projective}
So now consider an irreducible polarised projective-over-affine $\sG$-variety $(X,L)$ such that $X\div\;\sT$ is projective. By Proposition \ref{pro:proj_criterion}\;(6) this means  $H^0(X,\cO_X)^\sT \cong \C$. Since we also assume $(X,L)$ is weakly $\eta$-semiprojective we have a 1-parameter subgroup $\wt\eta:\C^*\into\sT$ all of whose nonzero weights on $H^0(\cO_X)$ are strictly negative. We will show how to use this to equivariantly complete $(X,L)$ to a projective variety $\(\;\overline X,\overline L\;\)$ with the same (semi)stable points and GIT quotients to which we can apply Maddock's proof of Theorem \ref{martin} to prove it for $(X,L)$.

Since $H^0(\cO_X)$ is finitely generated it contains a finite dimensional subspace $W^\vee$ which generates. We can take it to be preserved by $\sG$ by \cite[Proposition 4.6]{mukai2003}. We can also assume it does not contain the constants $\C$ by splitting them off if necessary. Therefore $(W^\vee)^\sT\subseteq H^0(X,\cO_X)^\sT \cong \C$ is zero. Since its $\wt\eta$-weights are all strictly negative we see that more is true: $(\Sym^{>0}W^\vee)^\sT=0$.

So now replacing $L$ by a power if necessary, there exists an equivariant embedding $(X,L) \hookrightarrow (\PP^m \times W, \cO(1))$, where $\sG$ acts linearly on both factors and $(\Sym W^\vee)^\sT=(\Sym^0W^\vee)^\sT=\C$ . Letting $\overline W:=\PP(W\oplus\C)$ be its projective completion we get, for any $k>0$, a $\sG$-equivariant embedding
$$
\(\;\overline X,\overline L\;\)\ \Into\ \(\PP^m \times \overline W, \cO(1,k)\).
$$
For $k\gg0$, by Lemma \ref{lem:equivariant_compactification} below,
$$
\overline X^{\;\overline L,\sT-\ss}\=\overline X\cap(\PP^m \times \PP^n)^{\cO(1,k),\sT-\ss}\=\overline X\cap(\PP^m \times \AA^n)^{\cO(1),\sT-\ss}\=X^{L,\sT-\ss}
$$
for $k\gg0$, where the first and last equalities hold by Remark \ref{rem:ss_on_closed}. Since the same holds for the $\sG$-semistable locus, Theorem \ref{martin} for $(X,L)$ follows from Theorem \ref{martin} for $\(\overline X,\overline L\)$.

\begin{lemma}\label{lem:equivariant_compactification}
    Let $\sG$ be a reductive group with representations $V,\,W$ such that 
$(\Sym W^\vee)^\sG \cong \Sym^0 W^\vee\cong\C$. Then for $k\gg0$ the projective completion $W\subset\overline W:=\PP(W\oplus\C)$ induces
$$
\(\PP(V)\times W\)^{\cO(1)-\ss}\ \cong\ 
\(\PP(V)\times\overline W\;\)^{\cO(1,k)-\ss}.
$$
In particular, $\(\PP(V)\times W\)\div\sG\cong\(\PP(V)\times\overline W\;\)\div\sG$.
\end{lemma}

\begin{proof}
    Consider the $\sG$-invariant coordinate ring
    \begin{align*}
        R\ :=\ \bigoplus\nolimits_{d=0}^\infty H^0\bigl(\cO_{\PP(V)\times W}(d)\bigr)^\sG\ \cong\ (\Sym V^\vee \otimes \Sym W^\vee)^\sG,
    \end{align*}
    bigraded by the scalar $\C^\ast$ actions on $V$ and $W$. It is finitely generated since $\sG$ is reductive so we may choose homogeneous generators $f_1, \dots, f_\ell \in R$. Since $(\Sym W^\vee)^\sG \cong \C$ their bidegrees $(d_i,e_i)$ have $d_i>0$ so we may fix $k\gg0$ with $kd_i>e_i$ for every $i$. Set
    \begin{align*}
        \overline R\ :=\ \bigoplus\nolimits_{d=0}^\infty H^0\bigl( \cO_{\PP(V)\times\overline W}(d,kd)) \bigr)^\sG\ \cong\ \bigoplus\nolimits_{d=0}^\infty \Bigl(\Sym^d V^\vee \otimes \bigoplus\nolimits_{e=0}^{kd} z^{kd-e}\!\cdot\Sym^e W^\vee \Bigr)^\sG,
    \end{align*}
where $z$ is the homogeneous coordinate dual to $\C\subset W\oplus\C$. The restriction map\,---\,the obvious injection $\overline R\into R$\,---\,maps
\beq{gens}
z^{kd_i-e_i} f_i\Mapsto f_i
\eeq
onto the generators $f_1, \dots, f_\ell$. Hence it is an isomorphism and the two GIT quotients are isomorphic. Moreover it follows that the $z^{kd_i-e_i} f_i$ of \eqref{gens} generate $\overline R$\;, but they all vanish on $\lbrace z=0\rbrace = \PP(V)\times\(\;\overline W\setminus W\)$ which is therefore in the strictly unstable locus.
\end{proof}

\section{\texorpdfstring{$(-2)$}{(-2)}-shifted symplectic structures on quotients}
\label{sec:derived_symp_red}

A slight strengthening of the existence of a perfect obstruction theory on a Deligne-Mumford stack $X$ is the existence of a quasi-smooth derived structure $\sX$ which induces it. That is, $X=t_0(\sX)$ should be the truncation of $\sX$ and the perfect obstruction theory should be the map $\EE:=\LL_{\sX}|_X\to\LL_X$ induced by the inclusion $X\subset\sX$. We did not need such a derived enhancement in the last Section because Behrend-Fantechi virtual cycles depend only on the perfect obstruction theory.

The construction of the virtual cycles of \cite{BorisovJoyce,OhThomas}, however, does require a (slight) derived enhancement of a symmetric 3-term obstruction theory\,---\,a \emph{$(-2)$-shifted symplectic structure} $\sX$ on $X=t_0(\sX)$ whose restriction to $X$ induces the symmetric obstruction theory. We briefly survey what we need of this theory.

\begin{remark} In this Section and the next we assume all our Deligne-Mumford stacks are quasi-projective \emph{in the strong sense of} \cite{Kr2} so the resolution property holds. This is the generality in which Kiem and Park \cite{KP} show the virtual cycle \cite{OhThomas} can be defined. Conveniently, if $(X,L)$ is quasi-projective in this strong sense and $\sG$ acts with all semistable points being stable then $X\div\sG$ is also quasi-projective by \cite[Proposition 5.1]{Kr2}. The applications we have in mind are all $(-2)$-shifted symplectic derived \emph{schemes} anyway; the prototypical example being the derived moduli scheme $\sX$ of stable sheaves on a Calabi-Yau fourfold \cite{ParkYou}.
\end{remark}

    \subsection{Shifted symplectic structures}
For a derived Artin stack $\sX$ with cotangent complex $\LL_\sX$ and dual $\TT_\sX$, let $\widehat{\Wedge\udot}\;\LL_\sX$ denote its (Hodge completed) de Rham complex
$$
\Big(\prod\nolimits_{i\ge0}\Wedge^i\;\LL_\sX[-i],\,d_{\mathrm{dR}}\Big)\=\
\lim_{\stackrel\longleftarrow n}\Big(\cO_\sX \xrightarrow{d_{\mathrm{dR}}} 
\LL_\sX \xrightarrow{d_{\mathrm{dR}}} \dots \xrightarrow{d_{\mathrm{dR}}} \Wedge^{n} \;\LL_\sX \Big),
$$
defined using the direct product (not direct sum). Similarly we denote its truncations by
$$
\widehat{\Wedge^{\ge p}}\;\LL_\sX\,=\,\Big(\prod\nolimits_{i\ge p}\Wedge^i\;\LL_\sX[-i],\,d_{\mathrm{dR}}\Big) \quad\text{ and }\quad \Wedge^{<p}\;\LL_\sX\,=\,\Big(\bigoplus\nolimits_{i<p}\Wedge^i\;\LL_\sX[-i],\,d_{\mathrm{dR}}\Big).\vspace{-2mm}
$$
These fit in the obvious exact triangle $\widehat{\Wedge^{\ge p}}\;\LL_\sX\to\widehat{\Wedge\udot}\;\LL_\sX\to\Wedge^{<p}\;\LL_\sX$, inducing
\beq{dRtri}
\dots\to\HH^k\Big(\widehat{\Wedge^{\ge p}}\;\LL_\sX\Big)\To H^k_{\mathrm{dR}}(\sX)\To\HH^k\Big(\Wedge^{<p}\;\LL_\sX\Big)\To\HH^{k+1}\Big(\widehat{\Wedge^{\ge p}}\;\LL_\sX\Big)\to\dots.
\eeq
A \emph{$(-n)$-shifted closed $p$-form} is defined to be an element 
$$
\omega\ \in\ \HH^{p-n}\Big(\widehat{\Wedge^{\geq p}}\;\LL_\sX\Big).
$$
It is called \emph{exact} if its de Rham class $[\omega]\in H^{p-n}_{\mathrm{dR}}(\sX)$ vanishes, i.e. $\omega$ lifts to $\HH^{p-n-1}\(\Wedge^{<p}\;\LL_\sX\)$ by \eqref{dRtri} (so there is a primitive $\phi\in\HH^{p-n-1}\(\Wedge^{<p}\;\LL_\sX\)$ with $d_{\mathrm{dR}}\phi=\omega$ in $\widehat{\Wedge\udot}\;\LL_\sX$).
Via $\widehat{\Wedge^{\geq p}}\;\LL_\sX \to \Wedge^p \LL_\sX[-p]$ the closed form $\omega$ induces a \emph{shifted $p$-form} $\omega_0 \in\HH^{-n}(\Wedge^{p}\;\LL_\sX)$.
A $(-n)$-shifted closed $2$-form $\omega$ is called \emph{symplectic} if
\beq{nss}
\omega_0 : \TT_\sX[n] \longrightarrow \LL_\sX\ \text{ is an isomorphism.}
\eeq
\subsection{\texorpdfstring{$(-2)$}{-2}-shifted symplectic structures and symmetric obstruction theories}
Suppose $(\sX,\omega)$ is a $(-n)$-shifted symplectic derived Deligne-Mumford stack with the resolution property. Thus $\LL_\sX\vert_X$ admits a  locally free resolution supported in degrees $(-\infty,0]$. Then dualising, shifting by $[n]$ and using the isomorphism \eqref{nss} gives a quasi-isomorphic complex of locally frees supported in $[-n,\infty)$. Therefore $n\ge0$ and the complex may be truncated to a quasi-isomorphic complex $\EE\cong\LL_\sX\vert_X$ of locally free sheaves supported in degrees $[-n,0]$. When $n=2$ we may furthermore assume, by \cite[Proposition 4.1]{OhThomas} or \cite[Proposition 7.3]{KP}, that $\EE$ is self-dual, with the isomorphism \eqref{nss} represented by a genuine map of complexes
\begin{equation}\label{diag:dueldot}
    \begin{tikzcd}[row sep=tiny]
        \EE^\vee[2] \arrow[dd, "\omega_0"']& & (T^\vee)^\vee \arrow[r, "(a^\vee)^\vee"] \arrow[dd, equal] & E^\vee \arrow[dd, "{\rotatebox{90}{$\sim$}}"', "q"] \arrow[r, "a^\vee"] & T^\vee \arrow[dd, equal, xshift=-3pt]\\  &= & & & \\
        \EE & & T \arrow[r, "a"]& E \arrow[r, "a^\vee"]& T^\vee.\!\!
    \end{tikzcd}
\end{equation}
Here $q\colon E^\vee \xrightarrow{\sim} E$ is a nondegenerate quadratic form on $E$ and $a^\vee$ is the adjoint of the map $a$ with respect to this isomorphism. This self-dual complex 
\begin{equation}\label{eq:3term:pot}
    \EE\=\LL_\sX|_X\To\LL_X
\end{equation}
provides the symmetric 3-term obstruction theory used to produce a virtual cycle when $X$ is a scheme \cite{OhThomas} or a Deligne-Mumford stack \cite{KP}.

\subsection{Shifted Lagrangians}
Suppose $(\sX,\omega)$ is $(-n)$-shifted symplectic. An isotropic structure on a morphism $f \colon \sL \to \sX$ is a class
\beq{isotrop}
\eta\ \in\ \HH^{2-n}\Big(\mathrm{Cocone}\Big(\widehat{\Wedge^{\geq 2}}\;\LL_\sX \to f_*\,\widehat{\Wedge^{\geq 2}}\;\LL_\sL\Big) \Big)\ \text{ lifting }\ \omega \in \HH^{2-n}\Big(\widehat{\Wedge^{\geq 2}}\;\LL_\sX\Big).
\eeq
By adjunction and the morphism of exact triangles
$$
\xymatrix@R=14pt{
\text{Cocone} \ar[r] \ar[d] &  f^*\widehat{\Lambda^{\geq 2}}\;\mathbb{L}_\sX \ar[r] \ar[d] & \widehat{\Lambda^{\geq 2}}\;\mathbb{L}_\sL \ar[d] \\
\mathbb{L}_f \otimes \mathbb{L}_\sL[-3] \ar[r]& f^* \mathbb{L}_\sX \otimes \mathbb{L}_\sL[-2] \ar[r] & \mathbb{L}_\sL \otimes \mathbb{L}_\sL[-2]
}
$$
we obtain an element $\overline\eta\in\HH^{-n-1}\(\LL_f\otimes\LL_\sL\)$.
We say $(f,\eta)$ is \emph{shifted Lagrangian} if 
$$
\overline\eta\,\colon\TT_\sL \To\LL_f[-n-1]\ \text{ is an isomorphism.}
$$

\subsection{Shifted symplectic reduction}\label{subsection:sympl_reduction}
Suppose a reductive group $\sG$ acts on a $(-n)$-shifted symplectic derived stack $(\sX,\omega)$, preserving $\omega$ in the strong sense\footnote{This is a shifted analogue of an exact or hamiltonian action, which Park calls 0-locked.} of \cite[Section 3.3]{Park}. That is, the induced morphism $\sX/\sG \to B\sG$
carries a relative $(-n)$-shifted symplectic structure $\omega_{\mathrm{rel}}$ such that $\sigma^*(\sX/\sG,\omega_{\mathrm{rel}}) = (\sX,\omega)$
via the Cartesian diagram
\beq{equivariant}
\begin{tikzcd}
\sX \ar[r]\ar[d]\arrow[dr,phantom,"\square"] & \mathrm{pt} \ar[d,"\sigma"] \\
\sX/\sG \ar[r] & B\sG.\!\!
\end{tikzcd}
\eeq
From now on assume $n>0$. This implies the de Rham cohomology class of the $(-n)$-shifted 2-form $\omega$ vanishes\,---\,this goes back to \cite[Corollary 5.3]{Toen} and a careful proof is given in \cite[Proposition 3.2]{KPS}.  So $\omega$ is exact; in fact \cite[Proposition 3.2]{KPS} shows it is in the image of a unique primitive in $\HH^{1-n}(\Wedge^{<2}\;\LL_\sX)$ if $n\ge2$ and a canonical primitive if $n=1$. Similarly $\omega_{\mathrm{rel}}$ is also exact because the reductivity of $\sG$ gives
$$
H^{2-n}_{\mathrm{dR}}\bigl((\sX/\sG)/B\sG\bigr)\ \cong\ H^{2-n}_{\mathrm{dR}}(\sX)^{\sG}
$$
so the vanishing of the de Rham class of $\omega$ gives the vanishing of the relative de Rham class of $\omega_{\mathrm{rel}}$.
The primitives define a moment map for the action of $\sG$  \cite[Section 3.3]{Park}: a $\sG$-equivariant morphism
\begin{eqnarray}
\mu &\!\!:\!& \sX \To \fg^\vee[-n] \label{mu} \\
\text{inducing}\quad \mu/\sG &\!\!:\!& \sX/\sG \To \fg^\vee[-n]/\sG \= T^*[1-n]B\sG \label{muG}
\end{eqnarray}
endowed with a Lagrangian structure $\eta$ \eqref{isotrop} together with an equivalence of $(-n)$-shifted symplectic derived stacks
$$
\sX\ \cong\ \sX/\sG \times\_{\fg^\vee[-n]/\sG} \fg^\vee[-n].
$$
The right hand side uses the $(1-n)$-shifted symplectic structure on $T^*[1-n]B\sG=\fg^\vee[-n]/\sG$ to induce\,---\,by \cite[Theorem 2.9]{PTVV}\,---\,a $(-n)$-shifted symplectic structure on the Lagrangian intersection (fibre product) of \eqref{muG} and the Lagrangian morphism $\fg^\vee[-n] \to \fg^\vee[-n]/\sG$.
By \cite[Definition 3.3.1]{Park}, the shifted symplectic reduction $(\sX/\sG)_{\sr}$ of \cite{calaque2015lagrangian, safronov2016quasi,Park}\footnote{We use the suffix sr for (shifted) symplectic reduction, reserving the suffix red for the underlying reduced structure on a scheme or Deligne-Mumford stack.}  is then the Lagrangian intersection of \eqref{muG} and $0:B\sG\to\fg^\vee[-n]/\sG$,
\beq{Xred}
(\sX/\sG)_{\sr}\ :=\ \sX/\sG \times\_{\fg^\vee[-n]/\sG}B\sG.
\eeq
Again this is naturally $(-n)$-shifted symplectic by \cite[Theorem 2.9]{PTVV}.

    \begin{lemma}\label{reduction_and_truncation}
            The classical truncation of the shifted symplectic reduction $(\sX/\sG)_{\sr}$ is the stack quotient of the truncation $X:=t_0(\sX)$ of $\sX$:
            \begin{align*}
                t_0((\sX/\sG)_{\sr})\ \cong\ X/\sG.
            \end{align*}
        \end{lemma}
        \begin{proof}
            Since the truncation functor commutes with homotopy limits and colimits by \cite[Definition 2.2.4.3 and below]{Toen_homotopical}, we have 
            \begin{align*}
                t_0((\sX/\sG)_{\sr}) &\ \cong\ t_0(\sX/\sG) \times_{t_0(\fg^\vee[-n]/\sG)} t_0(B\sG)\\
               &\ \cong\ t_0(\sX/\sG) \times_{B\sG}B\sG\\
               &\ \cong\ t_0(\sX)/\sG\\
               &\ \cong\ X/\sG. \qedhere
            \end{align*}
        \end{proof}

The case of interest to us is $n=2$. Then the above gives a $(-2)$-shifted symplectic structure on $(\sX/\sG)_{\sr}$ and so,  by \eqref{diag:dueldot}, a 3-term symmetric obstruction theory $\LL_{(\sX/\sG)_{\sr}}\big|_{X/\sG}\to\LL_{X/\sG}$ on $X/\sG=t_0((\sX/\sG)_{\sr})$ whenever this is a scheme or Deligne-Mumford stack. 

        Finally we describe the restriction of the cotangent complex of the derived symplectic reduction to its truncation $t_0((\sX/\sG)_{\sr})=X/\sG$.
        
        \begin{lemma}\label{obstruction_theory_quotient_cy4}
            In $D(X)^\sG$ we have the exact triangle
$$
             \pi^*\(\LL_{(\sX/\sG)_{\sr}}|_{X/\sG}\) \To \cone\Big(\fg\otimes\cO_X[n] \to \LL_{\sX}|_X \Big) \To \fg^\vee\otimes\cO_X.
$$
        \end{lemma}
        \begin{proof}
Consider the derived zero locus of the moment map $\mu$ \eqref{mu},
            \begin{align*}
                \sZ\ :=\ \sX \times\_{\fg^\vee[-n]}0,
            \end{align*}
so that $(\sX/\sG)_{\sr}=\sZ/\sG$. These descriptions give the exact triangles of cotangent complexes
$$
\fg[n]\otimes\cO_{\sZ} \rt{(D\mu)^*|_{\sZ}} \LL_{\sX}|_{\sZ} \To \LL_\sZ \quad\text{ and }\quad\pi^*\LL_{(\sX/\sG)_{\sr}} \To \LL_{\sZ} \To \fg^\vee\otimes\cO_{\sZ},
$$
where we use $\pi$ for both the projection $\sZ\to(\sX/\sG)_{\sr}$ and its truncation $X\to X/\sG$. Consider the first as an expression for $\LL_\sZ$; substituting it into the second and
pulling back to $X=t_0(\sZ)$ gives the claimed exact triangle.
        \end{proof}

Though we do not need it, it is true \cite{Pa2} that $(D\mu)^*$ is dual to the derivative of the $\sG$ action $\LL_{\sX}\to\fg^\vee\otimes\cO_{\sX}$, under the shifted pairing on $\LL_{\sX}$ given by the shifted $2$-form $\omega_0$. (In the classical unshifted setting of $n=0$ this is the defining property of the moment map.)

So Lemma \ref{obstruction_theory_quotient_cy4} shows clearly why we take the shifted symplectic reduction. Taking only the quotient by $\sG$ replaces $\LL_\sX$ by the cocone of $\LL_\sX\to\fg^\vee\otimes\cO$, destroying its shifted symmetry. Taking the zeros of the shifted moment map does not alter the classical geometry but it changes the derived structure, removing the dual $(D\mu)^*\colon\fg\otimes\cO_X[n]\to\LL_\sX$ term and restoring the symmetry to give something $(-n)$-shifted symplectic.

\subsection{Descending virtual classes to quotients}\label{subsection:descending_cy4}

We now discuss the application of this derived-geometric machinery to intersection theory.

Assume an action $\sG\acts(X,L)$ on a polarised projective-over-affine scheme\,---\,such that $L$-semistable points are $L$-stable\,---\,lifts to a $(-2)$-shifted symplectic derived enhancement
$$
\sG\,\acts\,(\sX, \omega) \ \text{ with }\ t_0(\sX)\ \cong\ X,
$$
preserving $\omega$ in the strong sense of \eqref{equivariant}.
Suppose we also have an \textit{orientation} \cite[Equation 59]{OhThomas} for the induced symmetric obstruction theory $\EE:=\LL_{\sX}|_X\to\LL_X$ of \eqref{eq:3term:pot}, namely an isomorphism $o : \cO_X \rightarrow \det(\EE)$ such that the composition
$$
\cO_X \xrightarrow{\ o \otimes o\ } \det(\EE)^{\otimes 2} \xrightarrow{\ \id \otimes \det\omega_0^{\vee}\ } \det(\EE)\otimes \det(\EE)^\vee\ \cong\ \cO_X \ \text{ equals } \ (-1)^{\binom{\vd}{2}}\id_{\cO_X},
$$
where $\vd X:= \rank \EE$ is the virtual dimension of $X$. Since this is a $\Z/2\Z$ choice at each point of $X$ it is automatically preserved by the action of a connected group $\sG$; if $\sG$ is not connected we must add the assumption that it preserves $o$. Then these data induce
\begin{itemize}
\item a $(-2)$-shifted symplectic reduction $(\sX^{\ss}/\sG)_{\sr}$ by \eqref{Xred}, \item with truncation $X^{\ss}/\sG=X\div\sG$ by Lemma \ref{reduction_and_truncation},
\item a $3$-term symmetric obstruction theory $\EE_{X\div \sG}=\LL_{(\sX^{\ss}/\sG)_{\sr}}\big|_{X\div\sG}\to\LL_{X\div\sG}$ sitting in the exact triangle of Lemma \ref{obstruction_theory_quotient_cy4},
\beq{pullback_obs_reduction}
    \pi^\ast \EE_{X \div\sG} \To \cone\Big(\fg_\C \otimes\cO[2] \to \EE|_{X^{\ss}} \Big) \To \fg_\C^\vee\otimes\cO_{X^{\ss}},
\eeq
\item and an orientation $o\_{X\div \sG}$ on $\EE_{X \div \sG}$ induced by $o$ via\footnote{We are using the convention of \cite[Equation 63]{OhThomas}, with $o\_{V^\vee}\in\det(V^\vee\oplus V[2])$ the orientation which makes the first summand $V^\vee$ a \emph{positive} maximal isotropic. Under the trivialisation $\det(V^\vee\oplus V[2])\cong\C$ described in \cite[Equation 8]{OhThomas}, $o\_{V^\vee}$ corresponds to $(-i)^{\dim V}\in\C$.\label{ofoot}}
\vspace{-1mm}
\begin{multline}\label{or-sr}
\qquad o\ \in\ \det\EE\,\stackrel{\eqref{pullback_obs_reduction}}\cong\,\pi^*\det(\EE_{X\div\sG}) \otimes\det(\fg_\C^\vee\oplus\fg\_\C[2]) \\
\cong\ \pi^*\det(\EE_{X\div\sG})\otimes\cO_{X^{\ss}}\ \ni\ \pi^*\!o\_{X\div \sG}\otimes o\_{\fg^\vee_\C}.\qquad
\end{multline}

\end{itemize}
By \cite[Theorem 4.6]{OhThomas}, \cite[Definition 7.6]{KP} we therefore get virtual fundamental classes
$$
[X]^{\vir}\,\in\,A_{\frac12\!\vd X}^\sG(X) \quad\text{ and }\quad
[X\div\sG]^{\vir}\,\in\,A_{\frac12\!\vd X-\dim\sG}(X\div\sG).
$$
If $\sG=\sG_1\times\sG_2$ then the $(-2)$-shifted symplectic structure, symmetric obstruction theory, orientation and virtual cycle on $\sX\div\sG_1$ are all $\sG_2$-invariant; in particular $[X\div\sG_1]^{\vir}\in A^{\sG_2}_{\frac12\!\vd X-\dim\sG_1}(X\div\sG_1)$. And if $\sT$ is a reductive subgroup of $\sG$, such as its maximal torus, then $\sT$-semistable points are $\sT$-stable and $X\div\;\sT$ is a Deligne-Mumford stack, so we can compare $[X\div\sG]^{\vir}$ and $[X\div\sT]^{\vir}$ just as in Proposition \ref{pro:compatibility_of_obs}.

\begin{proposition}\label{pro:compatibility_cy4}
    Let $X$ be a derived Deligne-Mumford stack with $\sG$-equivariant  $(-2)$-shifted symplectic structure such that $\sT$-semistable points are $\sT$-stable. Then via the diagram $X\div\sG\xleftarrow{\ p\ }X^{\sG-\ss}/\,\sT\stackrel i{\Into}X\div\,\sT$ of \eqref{eq:abelianisation_diagram},
    $$
			p^\ast [X\div\sG]^{\vir} \= i^\ast [X\div\,\sT]^{\vir}\ \text{ in }\ A_{\frac12\!\vd X-\dim\sT}(X^{\sG-\ss}/\,\sT).
$$
	\end{proposition}

    \begin{proof}
This proved by equating both classes with $[X^{\sG-\ss}/\,\sT]^{\vir}$ just as in Proposition \ref{pro:compatibility_of_obs}.

To do this we briefly review the construction of the virtual cycle \cite[Section 4.2]{OhThomas}. We write the 3-term symmetric obstruction theory $\EE_{X\div\,\sT}:=\LL_{(\sX^{\ss}/\,\sT)_{\mathrm{sr}}}|_{X\div\,\sT}$ in the standard form $T\to E\to T^\vee$ of \eqref{diag:dueldot}, where $E\cong E^\vee$ carries a nondegenerate quadratic form and orientation induced by $o\_{X\div\,\sT}$. The stupid truncation $E\to T^\vee$ maps to $\LL_{X\div\,\sT}$ giving a (stupid) perfect obstruction theory. Therefore \cite[Section 4]{BehrendFantechi} defines a cone $C_{X\div\,\sT}\subset E^\vee\cong E$ by taking the fibre product of the intrinsic normal cone in $E^\vee/T$ with $E^\vee$. By \cite[Proposition 4.3]{OhThomas} $C_{X\div\,\sT}\subset E$ is \emph{isotropic} so we may apply the square-root Gysin class $\surd0_E^{\;!}$ of \cite[Definition 3.3]{OhThomas} to it, giving
$$
[X\div\,\sT]^{\vir}\=\sqrt0_E^{\,!}[C_{X\div\,\sT}]\ \in\ A_{\frac12\vd X-\dim\sT}\(X\div\,\sT\).
$$
Pulling back $\EE_{X\div\,\sT}$ and $o\_{X\div\,\sT}$ by the flat map $i:X^{\sG-\ss}/\,\sT\into X\div\,\sT$, the intrinsic normal cone also pulls back by \cite[Proposition 3.14]{BehrendFantechi}, so $C_{X^{\sG-\ss}/\,\sT}$ is $i^*C_{X\div\,\sT}\subset i^*E$. Since $\surd0_E^{\;!}$ commutes with flat pullback (indeed it defines a bivariant class \cite[Theorem 4.6]{KP}),
\beq{p=i}
\big[X^{\sG-\ss}/\,\sT\big]^{\vir}\=\sqrt0_{\;i^*\!E}^{\,!}\big[i^*C_{X\div\,\sT}\big]\=
i^*\sqrt0_E^{\,!}[C_{X\div\,\sT}]\=i^*[X\div\,\sT]^{\vir}.
\eeq

To prove a similar result for $p:X^{\sG-\ss}/\,\sT\to X\div\sG$ we recycle the notation $T\to E\to T^\vee$ to now denote a standard form for the 3-term symmetric obstruction theory $\EE_{X^{\sG-\ss}}\to\LL_{X^{\sG-\ss}}$, where $T$ and $E$ are $\sG$-equivariant bundles on $X^{\sG-\ss}$. The composition
$$
T^\vee\To\EE_{X^{\sG-\ss}}\To\LL_{X^{\sG-\ss}}\To\Omega_{X^{\sG-\ss}}\To\fg^\vee\otimes\cO_{X^{\sG-\ss}}
$$
is onto; denote its kernel by $(T/\fg)^\vee$. Then $T/\fg\to E\to(T/\fg)^\vee$ is a $\sG$-equivariant standard form for $\pi^*\EE_{X^{\sG-\ss}/\sG}$. It therefore descends to $X^{\sG-\ss}/\sG$; let $\underline E$ denote the descent of $E\cong\pi^*\underline E$. Truncating and taking cones as above, on both $X^{\sG-\ss}$ and $X^{\sG-\ss}/\sG$, gives isotropic cones 
$$
C_{X^{\sG-\ss}}\ \subseteq\ E\ \text{ and }\ C_{X^{\sG-\ss}/\sG}\ \subseteq\ \underline E
$$
such that $\pi^*$ of the latter is the former, again by \cite[Proposition 3.14]{BehrendFantechi}. Replacing $\sG$ by $\sT$ and descending down $\rho:X^{\sG-\ss}\to X^{\sG-\ss}/\,\sT$ similarly gives a cone $C_{X^{\sG-\ss}/\,\sT}$ with pullback $C_{X^{\sG-\ss}}$. In sum we get $\sG$-equivariant isomorphisms
$$
\rho^*C_{X^{\sG-\ss}/\sT}\ \cong\ C_{X^{\sG-\ss}}\ \cong\ \pi^*C_{X^{\sG-\ss}/\sG}\ \subseteq\ E.
$$
By the compatibility \eqref{or-sr} of the orientations we may again use that $\surd0_E^{\;!}$ commutes with flat pullback to deduce isomorphisms of $\sG$-equivariant cycle classes
$$
\rho^*\big[C_{X^{\sG-\ss}/\sT}\big]^{\vir}\ \cong\ \pi^*\big[C_{X^{\sG-\ss}/\sG}\big]^{\vir}\ \in A_{\frac12\!\vd X}(X^{\sG-\ss}).
$$
Quotienting by $\sT$ gives
$$
\big[C_{X^{\sG-\ss}/\sT}\big]^{\vir}\ \cong\ p^*\big[C_{X^{\sG-\ss}/\sG}\big]^{\vir}\ \in A_{\frac12\!\vd X-\dim\sT}(X^{\sG-\ss}/\,\sT).
$$
Combined with \eqref{p=i} this gives the required result.
    \end{proof}
    
    \subsection{Oriented $(-2)$-shifted symplectic structure on the cut}
   Suppose we have
   \begin{itemize}
   \item a rank $r$ torus $\sT$ acting on
   \item a polarised projective-over-affine Deligne-Mumford stack $(X,L)$ such that
   \item $L$-semistable points are $L$-stable and
   \item a $(-2)$-shifted symplectic enhancement $(\sX,\omega)$ of $X=t_0(\sX)$, preserved by $\sT$ in the strong sense of \eqref{equivariant},
   \item an orientation $o\in\Gamma(\det\EE)$ on the symmetric obstruction theory $\EE:=\LL_{\sX}|_X\to\LL_X$.
   \end{itemize}
   
Choose a simplicial cone $\Sigma = \cone(\psi_1, \dots, \psi_k) \subset \ft^\vee_\Q$\,---\,not necessarily of full dimension $r$\,---\,transverse to $\Delta=\Delta^\sT(X,L)$. It has its own $(-2)$-shifted symplectic enhancement
$$
\sSigma\ :=\ T^*[-2]\Sigma_\C\=\Sigma_\C\times\Sigma_\C^\vee[-2]\=\Sigma_\C\times\Spec\Sym\!\(\Sigma_\C[2]\)
$$
with $(-2)$-shifted symplectic structure $\omega\_\sSigma=d_{\mathrm{dR}}\xi$, where $\xi$ is the tautological shifted 1-form on the shifted cotangent bundle. The $\sT$ action on $\Sigma_\C$ induces one on $\sSigma$ preserving $\omega\_{\sSigma}$. Thus
$$
\sT \times \dT\ \acts\ \(\sX \times \sSigma,\, \omega + \omega\_{\sSigma}\)
$$
and we can take the shifted symplectic reduction \eqref{Xred} by $\dT$ to get a new $(-2)$-shifted symplectic Deligne-Mumford stack
$$
\sT\ \acts\ (\sX_\Sigma,\omega\_\Sigma).
$$
By Lemma \ref{obstruction_theory_quotient_cy4} the resulting symmetric obstruction theory $\EE_{X_\Sigma}=\LL_{\sX_{\sSigma}}\big|_{X_\Sigma}\to\LL_{X_\Sigma}$ is the (descent to $X_\Sigma$ of the) following morphism of complexes on $(X\times \Sigma_\C)^{\ss}$:
\begin{equation}\label{eq:cy4_triangles_cut}
	\begin{tikzcd}[row sep=0em, column sep=1.6em]
        \cone\Bigl(&[-2.5em]\cO \otimes \ft_\C[2] \arrow[r] \arrow[dd, equal] & \mathbb E\,\boxplus\,\cO \otimes\!\bigl(\Sigma_\C^\vee \oplus \Sigma_\C[2]\bigr) \Bigr) \arrow[r] \arrow[dd, shorten >=-4pt, shorten <=-4pt] & \cO \otimes \ft_\C^\vee \arrow[dd, equal] &  & \pi^*\mathbb{E}_{X_\Sigma} \arrow[dd] \\ & & & &\simeq & \\
    \cone\Bigl(& \cO \otimes \ft_\C[2] \arrow[r] & \mathbb{L}_X\,\boxplus\,\cO \otimes\!\bigl(\Sigma_\C^\vee \oplus \Sigma_\C[2]\bigr)\Bigr) \arrow[r] & \cO \otimes \ft_\C^\vee &  & \pi^*\mathbb{L}_{X_\Sigma},
    \end{tikzcd}
    \end{equation}
    where $\LL_X \rightarrow \cO \otimes \ft_\C^\vee$ and $\cO \otimes \Sigma^\vee_\C \rightarrow \cO \otimes \ft_\C^\vee$ are induced from the $\sT$ actions on $X$ and $\Sigma_\C$ respectively. Taking determinants in \eqref{eq:cy4_triangles_cut} and using the convention of Footnote \ref{ofoot} defines an orientation $o'$ on $\pi^*\EE_{X_\Sigma}$ such that
\begin{multline} \label{o4}
o\otimes o\_{\Sigma^\vee_\C}\ \in\ \det\EE\otimes\det(\Sigma_\C^\vee \oplus \Sigma_\C[2])\ \cong\ \det\EE\otimes\cO_{(X\times \Sigma_\C)^{\ss}}\stackrel{\eqref{eq:cy4_triangles_cut}}\cong \\
\pi^*\det\EE_{X_\Sigma}\otimes\cO_{(X\times \Sigma_\C)^{\ss}}\ \cong\ \pi^*\det\EE_{X_\Sigma}\otimes\det(\ft_\C^\vee \oplus \ft_\C[2])\ \ni\ o'\otimes o\_{\ft^\vee_\C}.
\end{multline}
Since $\sT\times\dT$ is connected it preserves $o'$, which therefore descends to define a $\sT$-equivariant orientation $o\_\Sigma$ on $\EE_{X_\Sigma}$ with $\pi^*o\_\Sigma=o'$. All told then, we get an oriented $(-2)$-shifted symplectic enhancement
\beq{Xsig}
\sT\,\acts\,(\sX_\Sigma,\,\omega\_\Sigma,\,o\_\Sigma)\ \text{ with $\sT$-equivariant virtual cycle }\ [X_\Sigma]^{\vir}\,\in\,A^\sT_{\frac{1}{2}\vd(X)-c}(X_\Sigma)
\eeq
by \cite[Theorem 4.6]{OhThomas}, where $c:=\codim\!\(\Sigma\subset\ft^\vee_\Q\)$.

\begin{proposition}\label{restriction_open_stratum_cy4}
The following two oriented $(-2)$-shifted symplectic enhancements of the big open stratum $X_\Sigma^\circ \cong X^{\Sigma-\ss}/\,\wSigma\subset X_\Sigma$ are the same,
    \begin{enumerate}
        \item the restriction of the $(-2)$-shifted symplectic structure $(\sX_\Sigma,\omega\_\Sigma,o\_\Sigma)$ \eqref{Xsig} obtained as the $(-2)$-shifted symplectic reduction $\((\sX\times\sSigma)^{\ss}/\,\dT\)_{\sr\,}$,
        \item the $(-2)$-shifted symplectic reduction $\(\sX^{\Sigma-\ss}/\,\wSigma\)_{\sr\,}$ oriented by the quotient orientation $o\_{X^{\Sigma-\ss}\div\,\wSigma}$ of \eqref{or-sr}.
    \end{enumerate}
\end{proposition}

\begin{proof}
For clarity we begin by handling the easiest case: where $k=r$ and the generators $(\psi_i)_{i=1}^r$ of $\Sigma = \cone(\psi_1, \dots, \psi_r)$ form a $\Z$-basis for $\ft^\vee_\Z$, so that $\wSigma=\{1\}$. Picking generators $1\in\Psi_i$ in each of the summands of $\Sigma_\C=\bigoplus_{i=1}^r\Psi_i$ and acting by $\sT$ gives isomorphisms
\beq{TSig}
\sT\ \cong\ \Sigma_\C^0 \quad\Longrightarrow\quad T^*[-2]\sT\ \cong\ \mathsf{\Sigma}^\circ_\C.
\eeq
(Explicitly\,---\,over the point $t\in\sT$ and $\(\psi_1(t)1,\dots,\psi_r(t)1\)\in\Sigma^\circ_\C$ and identifying $T^*[-2]\sT\cong\sT\times\ft^\vee_\C[-2]$ and $\mathsf{\Sigma}^\circ_\C\cong\Sigma^\circ_\C\times\Sigma^\vee_\C[-2]$ in the usual way\,---\,this isomorphism identifies the dual basis $z_i\in\Sigma^\vee_\C$ of \ref{sigma_rep} with $\psi_i(t)\cdot\psi_i\in\ft^\vee_\C$.)

The moment map for the action of $\sT$ on $T^*[-2]\sT\cong\sT\times\ft^\vee_\C[-2]$ is just projection to the second factor, so the moment map for
$$
\sT\,\acts\,\sX^{\Sigma-\ss}\times\mathsf{\Sigma}^\circ_\C\ \cong\ \sX^{\Sigma-\ss}\times\sT\times\ft^\vee_\C[-2]
$$
is $\mu=\pi_X^*\mu_{\sX}+\pi\_{\ft^\vee_\C[-2]}$. Its zero locus is therefore a copy of $\sX^{\Sigma-\ss}\times\sT$: the section of $\sX^{\Sigma-\ss}\times\sT\times\ft^\vee_\C[-2]\to\sX^{\Sigma-\ss}\times\sT$ given by the graph of $-\pi_X^*\mu\_X$,
\beq{graph}
\begin{tikzcd}[column sep=7em]
\sZ(\mu)\=\sZ\Big(\pi_X^*\mu\_X+\pi\_{\ft^\vee_\C[-2]}\Big)\ \cong\ \sX^{\Sigma-\ss}\times\sT\, \ar[r,hook,"{(\id_{\sX\times\sT},\,-\pi_X^*\mu\_X)}"]& \,\sX^{\Sigma-\ss}\times\sT\times\ft^\vee_\C[-2].
\end{tikzcd}
\eeq
We precompose with the equivariant isomorphism
\beq{TdT}
\sT\,\acts\(X^{\Sigma-\ss}\times\sT\)\,\rt{\ \sim\ }\,\(X^{\Sigma-\ss}\times\sT\)\,\curvearrowleft\,\dT, \qquad (x,t)\,\Mapsto\,(t\cdot x,\,t),
\eeq
where, on the left, $\sT$ acts trivially on $\sX^{\Sigma-\ss}$ and by multiplication on $\sT$. Thus we get an isomorphism from $\sT\acts\(X^{\Sigma-\ss}\times\sT\)$
to $\dT\acts\sZ(\mu)$ and
\beq{lastt}
\((\sX^{\Sigma-\ss}\times\mathsf{\Sigma}^\circ_\C)/\,\sT\)_{\sr}\ \cong\ \sZ(\mu)/\,\dT\ \cong\ (\sX^{\Sigma-\ss}\times\sT)/\,\sT\ \cong\ \sX^{\Sigma-\ss}
\eeq
preserving the $(-2)$-shifted symplectic structure. \medskip

In the general case we modify the reasoning as follows. Let $\sT':=\sT/\,\wSigma$ be the quotient which acts effectively on $\Sigma_\C$. Thus \eqref{TSig} becomes
$$
\sT'\ \cong\ \Sigma_\C^0 \quad\Longrightarrow\quad T^*[-2]\sT'\ \cong\ \mathsf{\Sigma}^\circ_\C.
$$
The moment map for the $\sT'$ action on $T^*[-2]\sT'\cong\sT'\times(\ft'_\C)^\vee[-2]$ is projection to the second factor, so to get the moment map for the $\sT$ action we compose this with the inclusion $\iota$,
\beq{lie}
0\To(\ft'_\C)^\vee\rt\iota\ft^\vee_\C\rt p(\Sigma^\perp)^\vee\To0,
\eeq
where $\Sigma^\perp\subset\ft_\C$, the annihilator of $\Sigma_\Q\subset\ft^\vee_\C$, is the Lie algebra of $\wSigma$.
Therefore the moment map for $\sX^{\Sigma-\ss}\times\mathsf{\Sigma}^\circ_\C\,\cong\,\sX^{\Sigma-\ss}\times\sT'\times(\ft'_\C)^\vee[-2]$ is
$$
\mu\=\pi_X^*\mu\_X+\iota\circ\pi\_{(\ft'_\C)^\vee[-2]}.
$$
We pass to its derived zero locus $\sZ(\mu)$ in two steps. We first take the zeros of the \emph{projection} $p(\mu)=\pi_X^*\mu^\perp_X$ \eqref{lie}\,---\,where $\mu_X^\perp:X\to(\Sigma^\perp)^\vee$ is the $\wSigma$-moment map for $\sX$\,---\,then on that we lift $\mu|\_{\sZ(p(\mu))}=\iota\circ\mu'$ via \eqref{lie} and pass to $\sZ(\mu')\cong\sZ(\mu)$. Similarly we quotient by $\sT$ in two steps: first by $\wSigma$ then by $\sT'$. The first step gives
$$
\(\sZ(\mu_X^\perp)/\,\wSigma\)\times\sT'\times(\ft'_\C)^\vee[-2]\=(\sX^{\Sigma-\ss}/\,\wSigma)\_{\sr}\times\sT'\times(\ft'_\C)^\vee[-2].
$$
On this we then repeat the steps \eqref{graph}, \eqref{TdT}, \eqref{lastt} with $\sT'\,,\mu'$ replacing $\sT,\,\mu$. Just as in \eqref{lastt} this removes the second and third factors, leaving us with $(\sX^{\Sigma-\ss}/\,\wSigma)\_{\sr}$.
\medskip

The orientations match up because of the compatibility of the convention of Footnote \ref{ofoot} with taking products, giving $o\_{\ft^\vee_\C}=o\_{\Sigma_\C^\vee}\otimes o\_{(\Sigma^\perp)^\vee}$. Therefore, suppressing some pullback maps, $o\_\Sigma=o\otimes o\_{\Sigma_\C^\vee}\otimes o_{\ft^\vee_\C}^{-1}$ is the same as $o\_{X^{\Sigma-\ss}\div\,\wSigma}=
o\otimes o_{(\Sigma^\perp)^\vee}^{-1}$.
\end{proof}    

    \section{Jeffrey--Kirwan localisation for CY4 virtual cycles}    \label{sec:JK_cy4}

We are finally ready to state and prove the $(-2)$-shifted symplectic version of JK localisation. The set up is
\begin{itemize}
\item a reductive group $\sG\acts(X,L)$ acting on a polarised strongly projective-over-affine Deligne-Mumford stack,
\item with a lift $\sG\acts(\sX,\omega,o)$ to an oriented $(-2)$-shifted symplectic enhancement in the strong sense of \eqref{equivariant}, preserving the orientation $o$ (this is automatic if $\sG$ is connected), 
    \item such that $\sT$-semistable points are $\sT$-stable.
\end{itemize}	
We choose an $\eta \in \ft_\Q$ such that $(X,L)$ is weakly $\eta$-semiprojective (Definition \ref{weaksemi}).
By the openness of Proposition \ref{weak2} we may perturb it to avoid the hyperplane arrangement $\cH\subset\ft_\Q$ defined by the weights of $\sT\acts N^{\vir}_{X^\sT/X}$ (these are described in Proposition \ref{cy4_normal_bundles} below). So $\eta\in\ft_\Q\setminus\cH$ and we let $\cH^\vee\subset\ft_\Q^\vee$ be the dual hyperplane arrangement \eqref{eq:hyperplanes}.

By Lemma \ref{check} these conditions imply that $\sG$-semistable points are $\sG$-stable and $X\div\,\sT$, $X\div\sG$ are (strongly) projective Deligne-Mumford stacks.

Fix $\alpha \in H^\ast_\sG(X)$ inducing $\alpha^\sT\in H^\ast_\sT(X)$ via $H^\ast_\sG(X)\to H^\ast_\sT(X)$ and $\alpha\_0\in H^\ast(X\div\sG)$ by \eqref{induced_class}.

    \begin{theorem}[CY4 virtual JK localisation]\label{th:JK_cy4}
Assuming the above choose $\epsilon \in \ft^\vee_\Q$ sufficiently small that the segments $\{\mu(F) + s\epsilon\}_{s \in (0,1],\,F\subseteq X^\sT}$ are disjoint from $\cH^\vee$. Then
$$
            \int_{[X\div \sG]^{\vir}} \alpha\_0 \= \frac{1}{\vert W \vert}\sum_{\substack{F \subseteq X^\sT\\\
            \langle \mu(F), \eta \rangle >0}} \JK^{\eta}_{\mu(F)+\eps}\left(e^\sT(\fg_\C/\ft_\C) \cdot\int_{[F]^{\vir}} \frac{\alpha^\sT\vert_{F}}{\sqrt e^\sT\!\(N^{\vir}_{F/X}\)}\right).
$$
        Here $\sT\acts\fg_\C/\ft_\C$ is the adjoint action so $e^\sT(\fg_\C/\ft_\C)=\prod\epsilon_i$ where $\epsilon_i \in \ft_\Z^\vee$ are the roots of $\sG$. The orientations on $F\subseteq X^\sT$ and $N^{\vir}_{F/X}$ are discussed below.
\end{theorem}

As before we can replace $\alpha$ by $\alpha e^{c_1^\sG(L)}$ to get a more familiar formula.

\begin{remark} \textbf{Orientations and signs.} \label{orsi}
It is shown in \cite[Section 7]{OhThomas} that $N^{\vir}_{F/X_\Sigma}$ is orientable. Our convention here is to pick any choice of orientation $o\_{N^{\vir}_{F/X_\Sigma}}$,\vspace{-1mm} then to fix the orientation on $\EE_{F}$ such that $o\_{\EE_X}\big|_{F}=o\_{N^{\vir}_{F/X_\Sigma}}\otimes o\_{\EE_{F}}$. \vspace{-1mm} This then fixes the signs of both $\sqrt e^\sT$ and $[F]^{\vir}$ in the localisation formula. Changing the orientation on $N^{\vir}_{F/X_\Sigma}$ changes both $\sqrt e^\sT$ and $[F]^{\vir}$ by a sign which therefore cancels in the above formula.

So this convention does \emph{not} follow \cite[Section 7]{OhThomas}, where a specific choice is made by splitting the $\sT$-weights of $N^{\vir}_{F/X_\Sigma}$ into two disjoint subsets $S\sqcup(-S)\subset\ft^\vee_\Q$ and orienting $N^{\vir}_{F/X_\Sigma}=N^S\oplus N^{-S}$ so that $N^S$ defines a \emph{positive} maximal isotropic as in \cite[Equation 63]{OhThomas}. (In fact \cite[Section 7]{OhThomas} uses $\sT=\C^*$ and takes $S$ to be the positive weights.) Doing that here would introduce annoying signs into Proposition \ref{cy4_normal_bundles}. So even for $\sT=\C^*$ what we call $[F]^{\vir}$ here may differ from $[F]^{\vir}$ in \cite{OhThomas} by a sign.
\end{remark}

The proof follows the perfect obstruction theory working very closely, again beginning with the abelian case $\sG=\sT$. As explained after Theorem \ref{th:abelian_JK_general} we can (and do) assume that $\langle \mu(F), \eta \rangle\ne0$ for components $F\subseteq X^\sT$.
As in Section \ref{wide} we may choose a simplicial cone $\Sigma = \cone(\psi_1, \dots, \psi_r)\subset\ft^\vee_\Q$ of full dimension such that
\begin{itemize}
\item $\eta\big|_{\Sigma \setminus \{0\}}\,>\,0$,
\item $\Sigma$ is \emph{$\eta$-wide}: it satisfies $\langle \mu(F), \eta \rangle>0\iff\mu(F) \in\Sigma$ for every component $F\subseteq X^\sT$,
\item $\eps$ is in minus the interior $-\mathring\Sigma$ of $\Sigma$, and
\item $\Sigma$ is transverse to $\Delta=\Delta^\sT(X,L)$ in the sense of Definition \ref{def:transversal}.
\end{itemize}\smallskip
Thus weak $\eta$-semiprojectivity and Proposition \ref{lem:polyhedron_structure} show $\Delta \cap \Sigma\subseteq\{\eta\geq 0\}\cap \Delta$ is bounded and
$\rk\big[H^0(\cO_X)^\sT\to H^0(\cO_{X^{\Sigma-\ss}})^\sT\big]<\infty$ because the restriction map factors through $H^0\(\cO_{X^{\lbrace\eta\ge0\rbrace-\ss}}\)^\sT$. Thus the algebraic cut $X_\Sigma$ is a strongly projective Deligne-Mumford stack by Theorem \ref{thm:structure_cut} with a $\sT$-equivariant oriented $(-2)$-shifted symplectic enhancement $\sT\acts(\sX_\Sigma,\omega\_\Sigma,o\_\Sigma)$ by \eqref{Xsig}.

By \eqref{classes_on_cut} the class $\alpha\in H^*_\sT(X)$ induces $\alpha\_\Sigma\in H^*_\sT(X_\Sigma)$ to which we apply the virtual localisation formula of \cite[Theorem 7.1]{OhThomas}. With $\alpha_p$ defined as in \eqref{alphap} this gives
\beq{OTloc}
\int_{[X_\Sigma]^{\vir}} \alpha\_\Sigma \= \sum_{p \in \mu(X_\Sigma^\sT)} \int_{[F_p]^{\vir}} \frac{\alpha_p}{\sqrt e^\sT\!\(N^{\vir}_{F_p/X_\Sigma}\)}
\eeq
using the orientation conventions of Remark \ref{orsi}.
To evaluate \eqref{OTloc} we describe the virtual normal bundles to the fixed loci. For $p \in \mu(X_\Sigma^\sT)$ describe the corresponding fixed locus as $F_p \cong \mu^{-1}(\tau_p)\div\,\wsigma\subseteq X_\sigma\subseteq X_\Sigma$ as in Proposition \ref{pro:fixed_locus}\;\eqref{cond:3}. Recall the $\sT$-equivariant line bundles of \eqref{eq:Nj_bundle},
$$
    N_j\ :=\ X^{\sigma-\ss} \times\_{\dsigma}\Psi^\diag_j\ \in\  \Pic^\sT\!\!\(X^{\sigma-\ss}/\,\dsigma\)\ \stackrel{\eqref{eq:algcutstrata}}\cong\ \Pic^{\sT\!}(X^\circ_\sigma).
$$
Using the $(-2)$-shifted symplectic structures on $X_\Sigma$ and $X_\sigma$ from \eqref{Xsig}, and their induced symmetric obstruction theories $\EE_{X_\Sigma},\,\EE_{X_\sigma}$, we have the following analogue of Lemma \ref{lem:normal_bundles}.

\begin{proposition}\label{cy4_normal_bundles} We have isomorphisms
\begin{eqnarray} \label{E4}
\EE_{X_\Sigma}|_{F_p} &\cong& \EE_{X\div\,\wsigma}\!\big|_{F_p}\ \oplus\ \bigoplus\nolimits_{j \,:\, \psi_j \notin \sigma}\(N_j^\vee\oplus N_j[2]\)\big|_{F_p}, \\
\label{N4} N^{\vir}_{F_p/X_\Sigma} &\cong& N^{\vir}_{F_p/X_\sigma}\ \oplus\ \bigoplus\nolimits_{j \,:\, \psi_j \notin \sigma}\(N_j\oplus N_j^\vee[2]\)\big|_{F_p}.
\end{eqnarray}
In \eqref{E4} we have $o\_\Sigma|\_{F_p}=o\_\sigma|\_{F_p}\otimes o\_{\oplus_jN_j^\vee}|\_{F_p}$ in the conventions of Footnote \ref{ofoot}.
The $\sT$-weights of the first summand lie in $\rspan(\sigma)\subset\ft^\vee_\Q$. In the second summand, the $\sT$-weight of $N_j \vert_{F_p}$ is the (possibly rational) projection of $-\psi_j$ to $\tau_p$ along $\rspan(\sigma)$, i.e. minus a rational generator of a 1-dimensional edge of $\Sigma\cap \tau_p$ emanating from $p$.
\end{proposition}

\begin{proof}
    Apply the diagram \eqref{eq:cy4_triangles_cut} to both $\Sigma$ and $\sigma$ to describe and compare the obstruction theories on $X_\Sigma$ and $X_\sigma$. Writing $\Sigma_\C=\sigma\_\C\oplus\bigoplus_{j:\psi_j \notin \sigma}\Psi^{\;\diag}_j$ by Notation\;\ref{sigma_rep}, this gives
\beq{EEEE}
\EE_{X_\Sigma} \big|_{X^\circ_\sigma}\ \cong\ \EE_{X_\sigma}\big|_{X^\circ_\sigma}\ \oplus\ \bigoplus\nolimits_{j \,:\, \psi_j \notin \sigma} \bigl(N_j^\vee \oplus N_j[2]\bigr).
\eeq
Substituting in the isomorphism from Lemma \ref{lem:restriction_open_stratum},
$$
\EE_{X_\sigma}\big|_{X^\circ_\sigma}\ \cong\ \EE_{X\div\,\wsigma}\!\big|_{X^\circ_\sigma}
$$
then restricting to $F_p$ gives \eqref{E4}. The orientations on the diagram  \eqref{eq:cy4_triangles_cut} were described in \eqref{o4}, giving
$$
o\_{\Sigma}\otimes o_{\Sigma^\vee_\C}^{-1} \= o\_{\sigma}\otimes o_{\sigma^\vee_\C}^{-1}\ \so\ o\_{\Sigma}\=o\_{\sigma}\otimes o\_{\oplus_jN_j^\vee}
$$
on $F_p$, because $o\_{\Sigma^\vee_\C} \= o\_{\sigma^\vee_\C}\otimes o\_{\oplus_jN_j^\vee}$.

Dualising \eqref{EEEE}, restricting to $F_p$ and taking $\sT$-moving parts gives \eqref{N4}.

The weights of $N^{\vir}_{F_p/X_\sigma}$ lie in $\rspan\sigma$ because $X_\sigma$ is entirely fixed by $\wSigma$ by Remark \ref{rem:action_factorisation}.
And the $\sT$-weight of $N_j \vert_{F_p}$ was already calculated in Lemma \ref{lem:normal_bundles}.
\end{proof}

This allows us to describe contributions of the fixed loci to \eqref{OTloc}. As usual we consider separately the 3 types:
\begin{enumerate}
\item \textcolor{darkgreen}{the special fixed locus} $F_0=X\div\;\sT$,
\item \textcolor{blue}{the old fixed loci} where $p\in\mathring\Sigma\cap\partial_0\Delta$, and
\item \textcolor{red}{the new fixed loci} where $p\in\mathring\sigma\cap\partial_{r-\dim\sigma}\Delta$ for $0<\dim\sigma<r$.
\end{enumerate}	

\begin{proposition}\label{pro:cy4_special_old_contributions}
The contribution of (1) \textcolor{darkgreen}{the special fixed locus} to \eqref{OTloc} is
\beq{eq:special_contribution_cy4}
     \int_{[F_0]^{\vir}}\frac{\alpha\_\Sigma|\_{F_0}}{\sqrt e^\sT\!\(N^{\vir}_{F_0/X_\Sigma}\)} \= \int_{[X\div\;\sT]^{\vir}}\frac{\alpha\_0}{\prod_{j=1}^r \(c_1(N_j)-\psi_j\)}\,.
\eeq
The contribution of (2) \textcolor{blue}{the old fixed locus} to \eqref{OTloc} is
\beq{eq:old_contribution_cy4}
      \sum_{p \in \mathring\Sigma\cap\partial_0\Delta} \int_{[F_p]^{\vir}} \frac{\alpha\_\Sigma|\_{F_p}}{\sqrt e^\sT\!\(N^{\vir}_{F_p/X_\Sigma}\)} \=
      \frac1{|\;\wSigma\!|}\sum_{\substack{F\subseteq X^\sT\\ \mu(F) \in \Sigma}}
        \int_{[F]^{\vir}} \frac{\alpha \vert_{F} }{\sqrt e^\sT\!\(N^{\vir}_{F/X}\)}\,,
\eeq
where on the right hand side the virtual class on $F$ is induced from $F\subseteq X^\sT\into X$.
\end{proposition}

Finally, to remove the (residues of the) contributions from (3) \textcolor{red}{the new fixed loci} to \eqref{OTloc} we need the following result about all fixed loci $F_p$ for $p \in \mu(X^\sT_\Sigma)$.

\begin{proposition}\label{pro:weights_normal_bundle_CY4}
The rational function
$$
        \int_{[F_p]^{\vir}} \frac{\alpha\_\Sigma|\_{F_p}}{\sqrt e^\sT\!\(N^{\vir}_{F_p/X_\Sigma}\)}\,:\,\ft_\Q
        \xymatrix{\ar@{-->}[r]& \ \Q}
$$
has poles only on hyperplanes of the form $\lbrace\gamma = 0\rbrace$, where $\gamma \in \ft_\Q^\vee$ is either
    \begin{enumerate}
        \item[(i)] in $\rspan(\sigma)$, or
        \item[(ii)] on a 1-dimensional edge of the simplicial cone $\Sigma \cap \tau_p$.
        \end{enumerate} 
\end{proposition}

Just as in Section \ref{section:loc_cut}, Proposition \ref{pro:weights_normal_bundle_CY4} follows immediately from Proposition \ref{cy4_normal_bundles}. And applying Proposition \ref{cy4_normal_bundles} to $p=0,\,\sigma=\{0\}$ gives \eqref{eq:special_contribution_cy4}: the fixed part of the obstruction theory \eqref{E4} is $\EE_{X\div\,\sT}$ with its orientation $o\_{X\div\,\sT}$, while dualising the moving part gives $N^{\vir}_{F_0/X_\Sigma}=\bigoplus_j\(N_j\oplus N_j^\vee[2]\)$ with orientation $o\_{\oplus_jN_j}$ with respect to which $\bigoplus_jN_j$ is positive. Therefore, by \cite[Equation 22]{OhThomas},
$$
\sqrt e^\sT\Big(\bigoplus\nolimits_j\(N_j\oplus N_j^\vee[2]\)\Big)\=e^\sT
\Big(\bigoplus\nolimits_jN_j\Big)\=\prod\nolimits_je^\sT(N_j)\=\prod\nolimits_j\(e(N_j)-\psi_j\).
$$
Finally we apply Proposition \ref{cy4_normal_bundles} to $p\in\mu(X^\sT)\cap\Sigma$ with $\sigma=\Sigma$. We find $(\EE_{X_\Sigma},o\_{\Sigma})|\_{F_p}$ is the same as $(\EE_X,o)|_F$ after quotienting by the finite group $\wSigma$ and using the isomorphism $F_p\cong F/\,\wSigma$. Combined with
$$
A_*(F)\,\cong\,A_*(F/\,\wSigma)\,=\,A_*(F_p), \qquad [F]^{\vir}\,\Mapsto\,\big|\wSigma\!\big|\cdot[F_p]^{\vir}
$$
by Proposition \ref{restriction_open_stratum_cy4}, and
$$
H^*(F)\,\cong\,H^*(F/\,\wSigma)\,=\,H^*(F_p), \qquad \alpha\vert_F\,\stackrel{\eqref{alphap}}\Mapsto\,\alpha\_\Sigma|\_{F_p}\,,
$$
this proves \eqref{eq:old_contribution_cy4}.
 
Thus we have all the results needed to apply Section \ref{section:JK_loc} again, taking JK residues to give Theorem \ref{th:JK_cy4} in the abelian case. Finally the abelianisation argument of Section \ref{section:NA_JK_loc} together with Proposition \ref{pro:compatibility_cy4} proves the general case from this.

    \section{$K$-theoretic JK localisation}\label{KKJK}
$K$-theoretic JK localisation computes, for $V\in K_0^\sG(X)$,
$$
\chi(X\div\sG,V_0)\,\text{ from }\,\chi\_\sG(X,V)\,\text{ from }\,V|_{X^\sT}.
$$
Here\footnote{We will impose conditions to ensure the $H^i(X,V)$ and $H^i(X\div\sG,V_0)$ are finite dimensional.} $\chi\_\sG(X,V)$ is the virtual $\sG$-representation $\sum_i(-1)^iH^i(X,V)$, and $V\mapsto V_0$ is the composition of restriction and descent,
$$
K_0^\sG(X) \,\To\, K_0^\sG(X^{\ss})\ \cong\ K_0(X\div\sG) \qquad\text{and}\qquad 
K^0_\sG(X) \,\To\, K^0_\sG(X^{\ss})\ \cong\ K^0(X\div\sG),
$$
so that $V_0$ pulls back to $V|_{X^{\ss}}$. In this Section we give a number of ways of computing
$$
\chi^{\vir}(X\div\sG,V_0)\,\text{ from }\,\chi^{\vir}_\sG(X,V)\,\text{ from }\,V\otimes\cO^{\vir}_{X^\sT}
$$
for $V\in K^0_\sG(X)$. Here $X$ has either (c) a perfect obstruction theory or (d) an equivariant oriented $(-2)$-shifted symplectic structure (numbering to be explained below). By Sections \ref{BFsec} and \ref{sec:derived_symp_red} these induce the same structures on $X\div\sG$ and $X^\sT$, and so also virtual structure sheaves $\cO^{\vir}$ \cite{FantechiGottsche} in case (c) and \emph{twisted virtual structure sheaves} $\widehat\cO^{\vir}$ \cite[Definition 5.9]{OhThomas} in case (d).
We let $\chi^{\vir}(V)$ denote $\chi(\cO^{\vir}\otimes V)$
in (c) and $\chi\(\widehat\cO^{\vir}\otimes V\)$ in case (d).

\subsection*{Assumptions}
Throughout this Section we fix
\begin{enumerate}
\item[(a)] an action $\sG\acts(X,L)$ on a polarised projective-over-affine Deligne-Mumford stack,
    \item[(b)] such that $\sT$-semistable points are $\sT$-stable, where $\sT\subseteq\sG$ is a maximal torus, and
    \item[(c)] either $X$ is endowed with a $\sG$-invariant perfect obstruction theory $(\EE,\phi)$,
\item[(d)] or with a lift $\sG\acts(\sX,\omega,o)$ to an oriented $(-2)$-shifted symplectic enhancement in the strong sense of \eqref{equivariant}, preserving the orientation $o$ (e.g. if $\sG$ is connected). 
\end{enumerate}
Here ``projective-over-affine" is meant in the weak sense if we are in case (c) and in the strong sense in case (d). 

Let $\cH \subset \ft_\Q$ be the hyperplane arrangement \eqref{eq:hyperplanes} defined by the set $\Gamma \subset \ft_\Q^\vee$ of $\sT$-weights of $N^{\vir}_{X^\sT/X}$. Then we choose $\eta \in \ft_\Q\setminus \cH$ such that
        \begin{itemize}
        \item $X$ is $\eta$-semiprojective (Definition \ref{etasp}).
    \end{itemize}
(See Remark \ref{*} for weakening this to weak $\eta$-semiprojectivity. See Section \ref{C*sec} for a further weakening in the presence of an additional $\C^*$ action.) Lemma \ref{check} proves these conditions imply $\sG$-semistable points are $\sG$-stable and $X\div\,\sT,\,X\div\sG$ are Deligne-Mumford stacks which are weakly projective in case (c) and strongly projective in case (d).

\subsection*{Constant terms}
For simplicity we begin with $\sG=\sT$ abelian, $X$ compact and $V\otimes L^{n\gg0}$ in place of $V$. By taking the trace of the virtual representation, $\chi\_\sT(X,V\otimes L^n)\in H^0(\cO_\sT)$ can be considered as an algebraic function on $\sT$ via Notation\;\ref{psipsi},
$$
\chi\_\sT(X,V\otimes L^n)\=\sum\nolimits_{\psi\in\ft^\vee_\Z}a_\psi t^\psi.
$$
For any such Laurent polynomial $f$ we let its \emph{constant term} $\CT(f)$ be the coefficient of $1$ in its Laurent expansion,
$$
	    \CT : H^0(\cO_\sT)\,\To\, \C, \qquad \sum\nolimits_{\psi \in \ft_\Z^\vee} a_\psi t^\psi \Mapsto a_0.
$$
Equivalently we can restrict $f$ to the maximal compact subgroup $\sT_\R \subset \sT$ and average over Haar measure $\frac1{(2\pi i)^r}\frac{d{\bf t}}{\bf t}\big|_{\sT_\R}$, removing the oscillating terms to leave the constant term\footnote{Many readers will prefer to apply Cauchy's residue theorem to prove the result.} 
\beq{Haar}
\CT(f)\=\frac1{(2\pi i)^r}\int_{\sT_\R}f\,\frac{d{\bf t}}{\bf t}\= \frac1{(2\pi i)^r}\int_{\sT_\R} f \, \frac{dt_1}{t_1} \wedge \dots \wedge \frac{dt_r}{t_r}\,.
\eeq
In both expressions we orient $\sT_\R$ so that $\CT(1)=1$, while in the second we have picked an arbitrary decomposition $\sT\cong(\C^*)^r$. We can now state the simplest of our results.

    \begin{proposition}\label{CTab}
        For $X$ projective, $V \in K^0_\sT(X)$ and $n\gg0$,
$$
                    \chi^{\vir}(X\div\sT, V_0 \otimes L^n_0) \= \CT\( \chi^{\vir}_\sT(X,V \otimes L^n)\).
$$
    \end{proposition}

\noindent We will illustrate each result with the following simple running example,
\beq{eg}
\sT\,:=\,\C^\ast\ \acts\ (X,L)\,:=\,\(\PP(t\oplus t^{-1}),\cO(1)\),
\eeq
where $t$ denotes the standard weight 1 irreducible representation of $\sT$.
The semistable locus is $\C^*\subset\PP^1$ acted on with weight 2, so $X\div\,\sT$ is the orbifold $\C^*/\,\sT=$\,pt$/(\Z/2\Z)$, over which $L_0$ is the nontrivial line bundle (i.e. the stabiliser $\Z/2\Z$ acts on the fibre as $\pm1$). We calculate
\beq{chiL}
\chi\_\sT(X,L^n)\=\Sym^n(t\oplus t^{-1})^\vee\=t^{-n}+t^{-(n-2)}+\cdots+t^{n-2}+t^n\=\frac{t^{-n}-t^{n+2}}{1-t^2}
\eeq
for $n\ge0$. So taking the coefficient of $t^0$ gives, for $n\ge0$,
\beq{10}
            \CT\(\chi\_\sT(X,L^n)\) \= \left\{\!\!\begin{array}{ll}
                1 & \text{if $n$ is even}\\0 & \text{if $n$ is odd}
            \end{array}\!\!\right\} \= \chi\(\!\pt\!/(\Z/2\Z), L^n_0\).
\eeq
\subsection*{Localisation} For $X$ noncompact (with $X\div\sT$ compact) $\chi^{\vir}_\sT(X,V \otimes L^n)$ will have poles on $\sT$, so $\CT$ is ill-defined. Even for $X$ compact, we would like to apply $K$-theoretic virtual localisation to the right hand side of Proposition \ref{CTab}. This expresses the regular function $\chi^{\vir}_\sT(X,V \otimes L^n)$ on $\sT$ as a sum of functions with poles on the hypertoric arrangement $\cH_\sT\subset\sT$ of Definition \ref{defHT} (closely related to the hyperplane arrangement $\cH_\C\subset\ft_\C$ of \eqref{eq:hyperplanes}).

Since the poles will appear on the contour $\sT_\R$ of \eqref{Haar} we perturb it by the 1-parameter subgroup $\wt\eta:\C^*\into\sT$ associated to $\eta\in \ft_\Q\setminus\cH$ (by exponentiating the primitive integral element of $\Q_{>0}\cdot\eta\subset\ft_\Q$) to define a renormalisation of $\CT$. That is, on the set $H^0(\cO_{\sT \setminus \cH_\sT})$ of rational functions on $\sT$ whose poles lie in $\cH_\sT$ we define the
        \textit{$\eta$-constant term functional}
 \beq{CTeta}
            \CT^\eta:H^0(\cO_{\sT \setminus \cH_\sT})\To \C, \qquad
            f\Mapsto \lim_{s\to0}\frac1{(2\pi i)^r}\int_{\wt\eta(s)\cdot\sT_\R}f\,\frac{d{\bf t}}{\bf t}\,.
\eeq
For an algebraic expression for $\CT^\eta(f)$ see Section \ref{poles}. Of course if $f$ has no poles on $\sT$ then $\CT^\eta(f)=\CT(f)$. Applying this, and $K$-theoretic virtual localisation, to Proposition \ref{CTab} will give our next result. It extends to noncompact $X$ which is $\eta$-semiprojective.
  
     \begin{theorem}[Abelian JK localisation I]\label{KJKab_n}
    For $\eta$-semiprojective $X$ and $V \in K^0_\sT(X)$,
\beq{eq:KJK_large_n}
                    \chi^{\vir}(X\div\sT, V_0 \otimes L^n_0) \= \sum_{\substack{F\subseteq X^\sT\\ \langle\mu(F), \eta \rangle>0}}\CT^{-\eta}\left( \chi^{\vir}_\sT\left(\!F,\,\frac{V \otimes L^n\vert_F}{\Lambda\udot\(N^{\vir}_{F/X}\)^\vee}\right)\right)
\eeq
for $n\gg0$.
In case {\rm(d)} we get the same formula with $\Lambda\udot\(N^{\vir}_{F/X}\)^\vee$ replaced by $\sqrt{\mathfrak e^\sT}(N^{\vir}_{F/X})$.
    \end{theorem}

\noindent We see how this works in our running example \eqref{eg}. The contribution of the fixed points $0=\PP(t^{-1})$ and $\infty=\PP(t)$ in $X=\PP(t\oplus t^{-1})$ to the localisation formula for $\chi\_\sT(X,L^n)$ are
\beq{egloc}
\frac{\cO(n)|\_{0}}{\Lambda\udot N_{0/X}^\vee}\=\frac{t^n}{1-t^{-2}}\=\frac{-t^{n+2}}{1-t^2} \qquad\text{and}\qquad
\frac{\cO(n)|\_{\infty}}{\Lambda\udot N_{\infty/X}^\vee}\=\frac{t^{-n}}{1-t^2}\,.
\eeq
Their sum reproduces the last term of \eqref{chiL} and has no poles, but individually they have poles at the roots of unity $\pm1$ on $\sT_\R=S^1\subset\C^*$ and at $t=0$ or $t=\infty$.  Taking $\eta>0$ in $\ft_\Q=\Q$, then $\wt\eta(s)$ pushes the contour $\sT_\R=S^1\subset\C^*$ inwards in \eqref{CTeta}, away from the poles at roots of unity, leaving us with the residue at $t=0$. Similarly choosing $-\eta$ instead pushes the contour outwards, so $\CT^{-\eta}$ computes minus the residue at $t=\infty$,
$$
\CT^{\eta}(f)\,=\,\Res_{\;t=0}\(f(t)\tfrac{dt}t\),\qquad\CT^{-\eta}(f)\,=\,-\Res_{t=\infty}\(f(t)\tfrac{dt}t\).
$$
Rewriting the second as $\Res_{u=0}\(f\(\tfrac1u\)\(-\tfrac{du}u\)\)$ and applying both to \eqref{egloc} gives
\begin{align*}
&\CT^{\eta}\left(\frac{-t^{n+2}}{1-t^2}\right)
\=0, && \CT^{\eta}\left(\frac{t^{-n}}{1-t^2}\right)
\=\left\{\!\!\begin{array}{ll}
                1 & \text{if $n$ is even,}\\0 & \text{if $n$ is odd,}
            \end{array}\right. \\
&\CT^{-\eta}\left(\frac{-t^{n+2}}{1-t^2}\right)
\=\left\{\!\!\begin{array}{ll} 1 & \text{if $n$ is even,}\\0 & \text{if $n$ is odd,}
\end{array}\right. &&
\CT^{-\eta}\left(\frac{t^{-n}}{1-t^2}\right)\=0.
\end{align*}
Summing either row recovers \eqref{10}, and in both cases we only need to include one term according to the sign of $\<\mu(F),\eta\>$, as in \eqref{eq:KJK_large_n}. (Notice $\mu(0)=1$ and $\mu(\infty)=-1$.) 

This also shows we could have removed either $0$ or $\infty$ (depending on the sign of $\eta$) and worked on $\C$, where $\chi\_\sT(\C,L^n)$ precisely equals the localisation contribution from the unique remaining fixed point. Hence the results hold also for noncompact spaces.

This also illustrates that $\CT^\eta\ne\CT^{-\eta}$ for functions with poles, reflecting the fact that the constant term in a Laurent series depends on whether we expand about $t=0$ or $t=\infty$.
  
The nonabelian version is very similar.
    \begin{theorem}[Nonabelian JK localisation I]\label{th:KJK_large_n}
        For $V \in K^0_\sG(X)$ and $n\gg0$,
$$
            \chi^{\vir}(X\div\sG, V_0 \otimes L^n_0) \= \frac1{|W|}\sum_{\substack{F\subseteq X^\sT\\ \langle\mu(F), \eta \rangle>0}}\CT^{-\eta}\left(\Lambda\udot(\fg_\C/\ft_\C)^\vee \cdot \chi^{\vir}_\sT\left(\!F,\,\frac{V \otimes L^n\vert_F}{\Lambda\udot\(N^{\vir}_{F/X}\)^\vee}\right)\right).
$$
Here $\fg_\C/\ft_\C$ carries the adjoint action so $\Lambda\udot(\fg_\C/\ft_\C)^\vee=\prod\nolimits_{i}(1-t^{\eps_i})$ with $\epsilon_i \in \ft_\Z^\vee$ the roots of $\sG$.
In case {\rm(d)} we get the same formula with $\Lambda\udot\(N^{\vir}_{F/X}\)^\vee$ replaced by $\sqrt{\mathfrak e^\sT}(N^{\vir}_{F/X})$.
    \end{theorem}

	\subsection*{Virtual $K$-theoretic JK localisation} Just as we could drop the $e^{c_1^\sT(L)}$ term in the cohomological JK localisation formula, so we would like to set $n=0$ in the above results. To do this we use the \emph{toric JK residue} $\TJK^\eta_\xi$, constructed from $\JK_\xi^\eta$ in Definition \ref{def:TJK} below. It is defined more generally than $\CT^{-\eta}$ but equals it on a class of functions,
$$
        \TJK^\eta_\xi(t^{n\psi}f)\=\CT^{-\eta}(t^{n\psi} f)
$$
for $\psi\in\ft^\vee_\Z\setminus\cH$ in the same chamber as $\xi$ and $n\gg0$. The result is then the following.

    \begin{theorem}\label{th:KJK}
        Assume $X$ is $\eta$-semiprojective and let $\epsilon \in \ft^\vee_\Q$ be sufficiently small that the segments $\{\mu(F) + s\epsilon\}_{s \in (0,1],\,F\subseteq X^\sT}$ are disjoint from $\cH^\vee$. Then for $V \in K^0_\sG(X)$,
$$
            \chi^{\vir}(X\div\sG, V_0) \= \frac1{|W|}\sum_{\substack{F\subseteq X^\sT\\ \langle\mu(F), \eta \rangle>0}}\TJK^\eta_{\mu(F)+\eps} \left(\Lambda\udot(\fg_\C/\ft_\C)^\vee \cdot\,\chi^{\vir}_\sT\!\left(\!F,\,\frac{V\vert_F}{\Lambda\udot\(N^{\vir}_{F/X}\)^\vee}\right)\right).
$$
In case {\rm(d)} we get the same formula with $\Lambda\udot\(N^{\vir}_{F/X}\)^\vee$ replaced by $\sqrt{\mathfrak e^\sT}(N^{\vir}_{F/X})$.
   \end{theorem}

\noindent In our running example \eqref{eg} we have expressed $\chi\(\!\pt\!/(\Z/2\Z), L^n_0\)$ as $-\Res_{t=\infty}f(t)\frac{dt}t$, where $f(t)=\frac{-t^{n+2}}{1-t^2}$ is the contribution \eqref{egloc} of the fixed point $0\in\PP^1$. By the residue theorem, this is minus the sum of the residues at $t=0$ and at $p=\pm1\in\sT$. The first vanishes, the latter we express cohomologically in terms of Chern characters (i.e. using Hirzebruch-Riemann-Roch) by using the exponential map $\ft_\C\ni u\mapsto pe^u=t\in\sT$ about $p\in\{\pm1\}$ to pullback to $\ft_\C$ and take residues at $u=0$. Since $\frac{dt}t=du$ we get
$$
-\sum_{p=\pm1}\Res_{u=0}f(pe^u)du\=-\Res_{u=0}\frac{e^{(n+2)u}}{1-e^{2u}}-\Res_{u=0}\frac{(-1)^ne^{(n+2)u}}{1-e^{2u}}\=\frac12+\frac{(-1)^n}2\,,
$$
recovering the answer 1 ($n$ even) and 0 ($n$ odd) of \eqref{10} even for $n=0$.

\begin{remark}\label{*}
The rest of this Section is devoted to proving the above results. Since Fourier series are so much simpler than Fourier (or Laplace) transforms,\footnote{For instance, the continuous analogue of the constant term for a function on $\ft_\C$ (instead of $\sT$) is a renormalisation of its Fourier or Laplace transform at $0$. Hence the definition \eqref{eq:global_residue} of the  cohomological $\res^{\eta}$ is much more complicated than the definition \eqref{CTeta} of the $K$-theoretic of $\CT^\eta$.} the proofs are much easier than their cohomological analogues. Since \cite[Appendix A]{AganagicOkounkov} already works within algebraic geometry\,---\,i.e. it does not use the moment map\,---\,we follow it closely, adapting only to incorporate the virtual structure.

Therefore we do not even need to use the cut $X_\Sigma$,
at the cost of replacing our assumption of weak $\eta$-semiprojectivity of $(X,L)$ by $\eta$-semiprojectivity of $X$. However we could, if we wanted, deduce the same results assuming only weak $\eta$-semiprojectivity by following the proofs in the cohomological setting closely:\footnote{For a hamiltonian $S^1$ action on a symplectic manifold $(X,\omega)$, Metzler \cite{Me} gave a $K$-theoretic analogue of Lerman's proof \cite{Lerman} of JK localisation using the symplectic cut.}
    \begin{enumerate}
        \item apply $K$-theoretic localisation to $V_\Sigma \otimes L_\Sigma^n \in K^0_\sT(X_\Sigma)$,
        \item apply $\CT^\eta$ \eqref{CTeta}, instead of $\res^\eta$ \eqref{eq:global_residue}, to the localisation formula\,---\,using $\CT^\eta(f)=\CT(f)=\CT^{-\eta}(f)$ for functions $f$ without poles in place of $\res^\eta(f)=\res^{-\eta}(f)$ \eqref{1}, and \eqref{eq:CT_vanishing} in place of Proposition \ref{pro:vanishing_property}\,---\,to deduce the abelian $K$-theoretic JK localisation formula of Theorem \ref{KJKab_n},
        \item apply Theorem \ref{th:martin_K} below to deduce the result for nonabelian $\sG$.
    \end{enumerate}
    \end{remark}

    \subsection{Global sections on quotients}
    Throughout we fix the assumptions at the beginning of this Section. In particular all $\sG$-semistable points are $\sG$-stable. The following result is really the key to everything which follows. It originates in the ``quantisation commutes with reduction" literature; see \cite[Remark 3.3\;(ii)]{Teleman} for instance. We follow and expand on \cite[Appendix A, Lemma 3]{AganagicOkounkov}.
    
    \begin{proposition}\label{pro:sections_on_quotient}
For $V \in \coh^\sG(X)$ and $n\gg0$,
        \begin{equation}\label{eq:sections_on_quotient}
	        H^0(X\div\sG, V_0 \otimes L_0^n)\ \cong\ H^0(X, V\otimes L^n)^\sG.
	    \end{equation}
	\end{proposition}
	
    \begin{proof}
We will prove this by showing both the restriction and descent maps
\beq{==}
H^0(X, V\otimes L^n)^\sG\ \To\ H^0\(X^{\ss}, V\otimes L^n|_{X^{\ss}}\)^\sG\ \To\ H^0(X\div\sG, V_0\otimes L_0^n)
\eeq
are isomorphisms. By the good quotient property $p_*\cO_{X^{\ss}}=\cO_{X\div\sG}$ of $p:X^{\ss}\to X\div\sG$ and the projection formula, the descent map $E\mapsto E_0$ can be described as $E\mapsto(p_*E|_{X^{\ss}})^\sG$. Thus the second arrow \eqref{==} is an isomorphism \emph{for all n} and
$$
            M\ :=\ \bigoplus\nolimits_{n\geq0}H^0\(X^{\ss},V \otimes L^n|_{X^{\ss}}\)^\sG\ \cong\ \bigoplus\nolimits_{n\geq0} H^0\(X\div\sG,V_0\otimes L_0^n\)
$$
is a finitely generated module over the finitely generated ring
$$
R\ :=\ \bigoplus\nolimits_{n\geq0}H^0\(X^{\ss},L^n|_{X^{\ss}}\)^\sG\ \cong\ \bigoplus\nolimits_{n\geq0} H^0\(X\div\sG,L_0^n\)
\stackrel{\eqref{lem:ss_on_ss_is_ss}}{\ \cong\ }\bigoplus\nolimits_{n\geq0}H^0(X,L^n)^\sG.
$$
This will help us prove the the restriction map \eqref{==} is also an isomorphism for $n\gg0$.\smallskip
 
\noindent\textbf{Sections of high degree extend.} 
So we can choose finitely many positive degree generators $s_1,\dots,s_k$ of $R=\bigoplus\nolimits_{n\geq0}H^0(X,L^n)^\sG$ and finitely many generators $m_i$ of $M= \bigoplus\nolimits_{n\geq0}H^0\(X^{\ss},V \otimes L^n|_{X^{\ss}}\)^\sG$.  The $s_j$ vanish on the unstable locus so, for any fixed $N$ there exists $d\gg0$ such that all homogeneous polynomials $p$ of degree $d$ in the $s_j$ vanish to order $N$ on the unstable locus. Taking $N\gg0$ then, all the
$$
p(s_1,\dots,s_k)\cdot m_i\,\text{ extend by zero to $X$.}
$$
But by finite generation, all elements of $R_{\ge d}$ are of the form $p(s_1,\dots,s_k)$ and all elements of $M$ of sufficiently high degree are linear combinations of the $p(s_1,\dots,s_k)\cdot m_i$. Thus the restriction map \eqref{==} is a surjection.
\smallskip
 
\noindent\textbf{Restriction is injective in high degree.}
Since $\bigoplus_{n\geq 0}H^0(X,V\otimes L^n)^\sG$ is also finitely generated as an $R$-module, so is kernel of the restriction map. Therefore, by the same argument, elements of the kernel of sufficiently high degree vanish on $X\setminus X^{\ss}$. Since they vanish on $X^{\ss}$ by definition, we find they are zero.
\end{proof}

Since our conditions imply $X\div\sG$ is projective, both sides of \eqref{eq:sections_on_quotient} are finite-dimensional. So by Serre vanishing, for $n\gg 0$,
    \begin{equation}\label{eq:chi_on_quotient}
        \chi(X\div\sG,V_0 \otimes L_0^n) \= \chi\_\sG(X,V \otimes L^n)^\sG\ :=\ \sum\nolimits_{j\geq0} (-1)^j \dim H^j\bigl(X, V \otimes L^n\bigr)^\sG.
    \end{equation}
 Beware that although $\chi\_\sG(X,V \otimes L^n)^\sG$ is well defined, $\chi\_\sG(X,V \otimes L^n)$ may not be, for instance if $X$ has an unstable noncompact irreducible component over which $\sG$ acts trivially and $L$ scales by a nontrivial character.

\subsection{Virtual structure sheaves}
Under condition (c) the $\sG$-equivariant perfect obstruction theory on $X$ endows $X\div \sG$ with a perfect obstruction theory by Lemma \ref{lem:inducingPotOnQuotient}.
Thus we get virtual structure sheaves $\cO^{\vir}_X\in K_0^\sG(X)$ and  $\cO^{\vir}_{X\div \sG} \in K_0(X\div\sG)$ by \cite{FantechiGottsche}.

Similarly under condition (d) the $\sG$-equivariant oriented $(-2)$-shifted symplectic structure on $X$ endows $X\div \sG$ with the same by \eqref{Xred}.
Thus we get the \emph{twisted virtual structure sheaves}
$$
\widehat\cO^{\vir}_X\,\in\,K^\sG_0\(X,\Z\big[\tfrac12\big]\)\ \text{ and }\ \widehat\cO^{\vir}_{X\div \sG}\,\in\,K_0\(X\div\sG,\Z\big[\tfrac12\big]\)
$$
of \cite[Definition 5.9]{OhThomas}. In the two cases (c),\,(d) we define $\chi^{\vir}(V)$ to be
$$
\chi(V\otimes\cO^{\vir})\ \text{ and }\ \chi^{\vir}(V)\ :=\ \chi(V\otimes\widehat\cO^{\vir})
$$
respectively. And in both cases they behave well under $p:X^{\ss}\to X\div\sG$, as follows.
    
    \begin{proposition}\label{pro:descent_Ovir}
       In case (c) we have $p^\ast \cO^{\vir}_{X\div\sG} = \cO^{\vir}_X\vert\_{X^{\ss}}\in K^\sG_0(X^{\ss})$, so
        \begin{equation*}
            \cO_{X\div\sG}^{\vir} \= (\cO_X^{\vir})\_0\ \in\ K_0(X\div\sG),
        \end{equation*}
        with the same for $\widehat\cO^{\vir}$ in case (d). In particular, since $X\div\sG$ is projective,
\begin{equation}\label{eq:chi_vir_on_quotient}
        \chi^{\vir}(X\div\sG,V_0 \otimes L_0^n) \= \chi^{\vir}_\sG(X,V \otimes L^n)^\sG \qquad \text{for }V \in K_\sG^0(X)\text{ and }n\gg0.
        \end{equation}
    \end{proposition}
    \begin{proof}
We first work under condition (c) using perfect obstruction theories. Applying \cite[Proposition 2.11]{Qu} to the composition
        \begin{equation*}
            X^{\ss} \rt{\,\,\,p\,\,\,} X\div \sG \rt{\,\,\,q\,\,\,} \pt, 
        \end{equation*}
  gives $\cO^{\vir}_X \vert\_{X^{\ss}} = (q\circ p)^! (\cO_{\pt})$, where $p^!,\,q^!$ are his virtual pullbacks. So $q^! (\cO_{\pt})  = \cO^{\vir}_{X\div\sG}$, while $p^!=p^*$ because $p$ is smooth and we endowed it with the trivial relative obstruction theory $\Omega_p$ in Lemma \ref{lem:inducingPotOnQuotient}. This gives the first claim, from which
the second follows by the projection formula for $p$ and the good quotient property $(p_\ast \cO_{X^{\ss}})^\sG = \cO_{X\div\sG}$.
Finally \eqref{eq:chi_vir_on_quotient} now follows by applying \eqref{eq:chi_on_quotient} to $V\otimes \cO_X^{\vir}$.\smallskip

Working with $(-2)$-shifted symplectic structures under condition (d), the proof is the same on replacing \cite[Proposition 2.11]{Qu} by \cite[Theorem B.3]{Pa3} and using the relative obstruction theory $\Omega_p\oplus\Omega_p[-2]$ on $p$. The only change is the twisting, for which we take determinants in the exact triangle of Lemma \ref{obstruction_theory_quotient_cy4}. This gives
$$
p^*K^{\vir}_{X\div\sG}\ \cong\ K^{\vir}_X\big|_{X^{\ss}}\otimes\det\fg\otimes\det\fg^\vee\ \cong\ K^{\vir}_X\big|_{X^{\ss}}
$$
in $K^0_\sG(X^{\ss})$. So descending and twisting by square roots as in \cite[Equation 133]{OhThomas}, we get the same results for twisted virtual structure sheaves
$$
p^\ast \widehat\cO^{\vir}_{X\div\sG} \= \widehat\cO^{\vir}_X\vert\_{X^{\ss}}\,\in\,K^\sG_0(X^{\ss})
\quad\text{so that}\quad \widehat\cO^{\vir}_{X\div\sG} \= \(\widehat\cO^{\vir}_X\)_0\,\in\,K_0(X\div\sG).\qedhere
$$
    \end{proof}

\begin{remark}[Coefficients]\label{cffs}
We have been a little loose about the coefficients of $K$-theory since they do not affect $\chi^{\vir}$, which is ultimately what we are interested in. To form the twisted structure sheaf we have to invert $2$ in $K_0$ \cite[Section 5.1]{OhThomas}. To form the \emph{equivariant} twisted structure sheaf we also have to admit square roots $t^{\psi/2}$ of our characters. This can be done formally, or by replacing $\sT$ by the $2^r$-fold $t\mapsto t^2$ cover $\widetilde\sT$ of $\sT$ (it is another copy of $\sT$ of course). Thus $\widetilde\ft_\Z=2\ft_\Z$ and $\widetilde\ft_\Z^{\,\vee}=\frac12\ft_\Z^\vee$.
\end{remark}

\subsection{Poles and the $\eta$-constant term}\label{poles}   We need the toric analogue of the hyperplane arrangement $\cH_\C \subset \ft_\C$ from \eqref{eq:hyperplanes}. In our application $\Gamma$ will be the set of weights of $N^{\vir}_{X^\sT/X}$.

\begin{definition}\label{defHT}
    Given a finite subset $\Gamma\subset\ft^\vee_\Z\setminus\{0\}$, we consider the union of the kernels of the corresponding characters, 
    \begin{equation}
        \cH_\sT\ :=\ \bigcup\nolimits_{\gamma \in \Gamma} \ker(\gamma) \ \subset \ \sT.
    \end{equation}
\end{definition}

Note if $\gamma$ is primitive then $\ker(\gamma)$ is a subtorus of $\sT$, but in general is the orbit of such under the action of a finite subgroup of $\sT$.

Note also that $\cH_\C$ is the hyperplane arrangement tangent to the toric arrangement $\cH_\sT$ at $1 \in \sT$. If every $\gamma$ is primitive, then $\cH_\sT$ equals the image of $\cH_\C$ through the exponential map $\exp : \ft_\C \to \sT$, but in general it is bigger.\medskip

We can now describe the action of $\eta$-constant term functional $\CT^\eta$ \eqref{CTeta}  on the ring $H^0(\cO_{\sT\setminus\cH_\sT})$ of rational functions on $\sT$ whose poles lie in $\cH_\sT$. It is sufficient to describe it on its generators (as a vector space)
$$
            g \=t^\nu \cdot\prod_{\gamma\in\Gamma\,\cup\,-\Gamma}
        \frac{1}{(1-t^{\gamma})^{d_\gamma}}\,, \quad \nu \in \ft^\vee_\Z,\ d_\gamma\geq 0.
$$
We expand the denominator \emph{in $\eta$-positive monomials} only\,---\,i.e. $t^\gamma$ with $\<\gamma,\eta\>>0$\,---\,using
\beq{eq:eta_expansion}
\frac1{(1-t^{-\gamma})^d}\=\frac{(-1)^dt^{d\gamma}}{(1-t^\gamma)^d} \quad\text{and}\quad \frac1{(1-t^\gamma)^d}\=1+dt^\gamma+\frac{d(d+1)}2t^{2\gamma}+\cdots
\eeq
We then multiply by $t^\nu$ and taking the constant term in the resulting Laurent series to recover $\CT^\eta(g)$.



In particular it follows that if $f \in H^0(\cO_{\sT \setminus \cH_\sT})$ is regular away from $\bigcup_{i=1}^\ell \ker(\gamma_i) \subset \sT$, where the $\gamma_i$ are $\eta$-normalised $\<\gamma_i,\eta\>>0$, then
    \begin{equation}\label{eq:CT_vanishing}
        \CT^\eta(t^{-n\xi} f) \= 0  \quad\text{for }\xi\not\in\cone(\gamma_1, \dots, \gamma_\ell)\text{ and }n\gg0.
    \end{equation}

\subsection{Formal characters and localisation}
We would like to define a formal character
\beq{eq:extended_Euler_char}
                \widehat{\chi}\_\sT(X,W)\ :=\ \sum\nolimits_{\psi \in \ft_\Z^\vee} \Bigl(\sum\nolimits_{k \geq 0}(-1)^k \dim H^k(X,W)_\psi\Bigr)\,t^\psi\ \in\ \prod\nolimits_{\psi \in \ft^\vee_\Z} \<t^\psi\>,
\eeq
for $W \in \coh^\sT(X)$, as a formal series with possibly infinitely many nonzero coefficients in every direction of $\ft^\vee_\Z$. This is well defined, with individual coefficients finite, by the next Lemma. It then extends by linearity to any $W\in K_0^\sT(X)$.

        \begin{lemma}\label{pro:chi_well_def}
        Suppose $\dim H^0(\cO_X)^\sT < \infty$ for a projective-over-affine DM stack $X$ acted on by $\sT$. Then, for $W \in \coh^\sG(X)$,
            \begin{equation*}
                 \dim \bigoplus\nolimits_i H^i(X,W)_\psi\ <\ \infty \quad\text{for every } \psi\,\in\,\ft^\vee_\Z.
            \end{equation*}
        \end{lemma}

        \begin{proof}
            By pushing down the projective morphism $\pi : X\to X_{\mathrm{aff}}$ we find that $H^i(X,W)$ is a finitely generated $\sT$-module over $H^0(\cO_X)$, so is a quotient
           \beq{eq:generators_of_M}
            \bigoplus\nolimits_{j=1}^k H^0(\cO_X)\otimes\Psi_j\,\Onto\,H^i(X,W)
           \eeq
for some choice of characters $\psi_j\in\ft^\vee_\Z$. Thus $H^i(X,W)_\psi$ is a quotient of $\bigoplus_{j=1}^k H^0(\cO_X)_{\psi-\psi_j}$, which is finite dimensional by Lemma \ref{trayn}.
        \end{proof}

Under our running assumptions, the formal series $ \widehat\chi\_\sT(X,W)$ \eqref{eq:extended_Euler_char} satisfies a  $K$-theoretic localisation formula.

        \begin{theorem}\label{thm:K_localisation}
Suppose $X$ is $\eta$-semiprojective. Then for $W\in K_0^\sT(X)$,
        \begin{itemize}
            \item The series \eqref{eq:extended_Euler_char} is bounded in the $\eta$-positive direction:
            \begin{equation}\label{eq:novikov_condition}
\big\{\psi \,:\, \widehat\chi\_\sT(X,W)_\psi\,\ne\,0\,\text{ and }\,\langle \psi, \eta \rangle\,>\,c \big\} \ \text{ is finite for all }c,
            \end{equation}
        \item In case {\rm (c)}, for $V\in K^0_\sT(X)$ the series $\widehat{\chi}\_{\sT}\(X,V\otimes\cO^{\vir}_X\)$ is the $(-\eta)$-positive expansion \eqref{eq:eta_expansion} of the rational function on $\sT$
\beq{locloc}
         \sum_{F \subseteq X^\sT} \chi\_\sT\!\left(\!F,\,\frac{V\vert_F\otimes\cO^{\vir}_F}{\Lambda\udot\(N_{F/X}^{\vir}\)^\vee} \right) \ \in\ H^0\(\cO_{\sT \setminus \cH_\sT}\).
\eeq
        \item In case {\rm(d)}, for $V\in K^0_\sT(X)$ the series $\widehat{\chi}\_{\sT}\(X,V\otimes\cO^{\vir}_X\)$ is the $(-\eta)$-positive expansion \eqref{eq:eta_expansion} of the rational function on $\widetilde\sT$ (cf. Remark \ref{cffs})
\beq{locloc2}
         \sum_{F \subseteq X^\sT} \chi\_\sT\!\left(\!F,\,\frac{V\vert_F\otimes\widehat\cO^{\vir}_F}{\sqrt{\mathfrak e^\sT}(N^{\vir}_{F/X})} \right) \ \in\ H^0\Big(\cO_{\widetilde\sT \setminus \cH_{\widetilde\sT}}\Big).
\eeq
        \end{itemize}
        \end{theorem}
        
        \begin{proof} Let $I:= \sqrt{0}\subset H^0(\cO_X)$ be the nilradical and write, noncanonically,
        \begin{equation}\label{eq:red_nilp}
            H^0(\cO_X)\ \cong\ H^0(\cO_X)\_{\red}\,\oplus\, \Bigl(\bigoplus\nolimits_{k\geq 1} I^k/I^{k+1}\Bigr)
        \end{equation}
as $\sT$-modules. Each $I^k/I^{k+1}$ is a finitely generated module over $H^0(\cO_X)_{\red}$ so the set of its weights is contained in a finite union of translates of the cone $\Delta(X_{\mathrm{aff}})$ of \eqref{Cdef}. Since only finitely many of the $I^k/I^{k+1}$ are nonzero the same is true for the set of weights of $H^0(\cO_X)$ by \eqref{eq:red_nilp}. And by \eqref{eq:generators_of_M} the same is also true for the weights of $H^i(X,W)$. Since $\Delta(X_{\mathrm{aff}})\subset \lbrace 0 \rbrace \cup \lbrace \eta<0\rbrace$ by $\eta$-semiprojectivity, this proves \eqref{eq:novikov_condition}. \medskip

The series satisfying \eqref{eq:novikov_condition} form a \textit{ring} which contains the image of the $H^0(\cO_\sT)$-module homomorphism $\widehat{\chi}\_\sT: K^\sT_0(X) \to \prod_{\psi} \langle t^\psi \rangle$. The element $1-t^{\gamma}$ is invertible in this ring, with
        \begin{equation*}
            \frac{1}{1-t^\gamma} \= \begin{cases}
                \sum_{k\geq 0}t^{k\gamma} & \text{if }\langle\gamma,\eta\rangle<0,\\
                -t^{-\gamma}\sum_{k\geq 0}t^{-k\gamma} & \text{if }\langle\gamma,\eta\rangle>0.
            \end{cases}
        \end{equation*}
        Thus $\widehat{\chi}\_\sT$ extends uniquely to the localisation 
        $K_0^\sT(X)\big[{(1-t^{\gamma})^{-1}}\big]_{\gamma \in \Gamma}$,             where $\Gamma \subset \ft_\Z^\vee$ is the set of weights of $\sT \curvearrowright N_{X^\sT/X}^{\vir}$. This is the ring which features in the localisation theorem \cite[Theorem 3.3]{Qu} for perfect obstruction theories,
     $$
            V \otimes \cO^{\vir}_X \=  i_\ast \sum\nolimits_{F \subseteq X^\sT\!} \left(\frac{V \vert_F \otimes \cO_F^{\vir}}{\Lambda\udot\(N_{F/X}^{\vir}\)^\vee}\right) \ \in\ K_0^\sT(X)\Big[\frac1{1-t^{\gamma}}\Big]_{\gamma \in \Gamma}\,.
    $$
Applying our extension of $\widehat{\chi}\_\sT$ to both sides gives \eqref{locloc}.

Similarly for oriented $(-2)$-shifted symplectic structures we have the localisation theorem
     $$
            V \otimes \widehat\cO^{\vir}_X \=  i_\ast \sum\nolimits_{F \subseteq X^\sT\!} \left(\frac{V \vert_F \otimes \widehat\cO_F^{\vir}}{\sqrt{\mathfrak e^\sT}(N^{\vir}_{F/X})}\right) \ \in\ K_0^\sT\(X\)\otimes\_{\Z[t^{\pm1}]}\Z\Big[\frac12,\,t^{\pm1/2},\,\frac1{1-t^{\gamma}}\Big]_{\gamma \in \Gamma}\,.
    $$
    (In \cite[Theorem 7.3]{OhThomas} this is stated with $\Q(t^{1/2})$ coefficients for simplicity, but the proof only uses the localisation above.)
        \end{proof}

    \subsection{Abelian $K$-theoretic JK localisation}
Both sides of \eqref{eq:KJK_large_n} are invariant under replacing $\sT$ by a cover $t\mapsto t^k$: the left-hand side by \eqref{eq:chi_vir_on_quotient}, the right-hand side because replacing $t$ by $t^k$ leaves the constant term $t^0$ unchanged. So to prove Theorem \ref{KJKab_n} we may assume that the $\sT$ action on $X^\sT$ is trivial so $K_0^\sT(X^\sT)=K_0(X^\sT)\otimes H^0(\cO_\sT)$ and the weights of $N^{\vir}_{X^\sT/X}$ are integral.

    \begin{proof}[Proof of Theorem \ref{KJKab_n}]
We first work with perfect obstruction theories in case (c). By \eqref{eq:chi_vir_on_quotient} the $T$-invariant part of $\widehat\chi\_\sT$, i.e. its $t^0$ coefficient, satisfies
$$
 \chi^{\vir}(X\div\sT,V_0 \otimes L_0^n) \= \widehat\chi^{\vir}_\sT(X,V \otimes L^n)^\sT \quad\text{for }n\gg0.
$$
By weak $\eta$-semiprojectivity we may apply the $K$-theoretic localisation formula of Theorem \ref{thm:K_localisation} to express $\widehat\chi\_\sT$ as the $(-\eta)$-positive expansion \eqref{eq:eta_expansion} of the rational function
$$
\sum_{F \subseteq X^\sT} \chi^{\vir}_\sT\!\left(\!F,\,\frac{V \otimes L^n\vert_F}{\Lambda\udot\(N^{\vir}_{F/X}\)^\vee}\right)\ \in\ H^0\(\cO_{\sT\setminus \cH_\sT}\).
$$
That is, applying \eqref{eq:eta_expansion} to $-\eta$ in place of $\eta$, we expand denominators in positive powers of only those $t^\gamma$ with $\<\gamma,\eta\><0$. On such expansions $\CT^{-\eta}$ takes the coefficient of $t^0$, so
\beq{eq:KJKab_large_n}
 \chi^{\vir}(X\div\sT,V_0 \otimes L_0^n) \= \sum_{F \subseteq X^\sT} \CT^{-\eta}\left(\chi^{\vir}_\sT\!\left(\!F,\,\frac{V \otimes L^n\vert_F}{\Lambda\udot\(N^{\vir}_{F/X}\)^\vee}\right)\right) \quad\text{for }n\gg0.
\eeq	
        It remains to prove that components $F \subseteq X^\sT$ with $\langle \mu(F), \eta \rangle \leq0$, contribute zero to the right-hand side of \eqref{eq:KJKab_large_n}. Write $L|_F$ as $\mathsf L \otimes\mu(F)$, where $\mathsf L$ has trivial $\sT$ action and $\mu(F)$ is the character, and write
$$
        \frac{V\vert_F}{\Lambda\udot\(N^{\vir}_{F/X}\)^\vee}\ \text{ as }\ \sum\nolimits_i E_i \otimes g_i\ \in\ K^0(F)\otimes H^0(\cO_{\sT\setminus \cH_\sT}).
$$
        Then the contribution is
\beq{eq:K_F_contribution}
         \CT^{-\eta}\left(   \chi^{\vir}_\sT\left(F,\,\frac{V \otimes L^n\vert_F}{\Lambda\udot\(N^{\vir}_{F/X}\)^\vee}\right)\right) \=   \sum_i \chi^{\vir}\(F,\, E_i|_F \otimes\mathsf L^n\)\, \CT^{-\eta}(t^{n\mu(F)}g_i).
\eeq
For $\langle \mu(F), \eta \rangle \leq0$ this vanishes for $n\gg0$ by \eqref{eq:CT_vanishing}.

Finally in case (d) the proof is the same on replacing \eqref{locloc} by \eqref{locloc2} and $\sT$ by $\widetilde\sT$.
    \end{proof}
    \subsection{Nonabelian $K$-theoretic JK localisation}
The nonabelian Theorem \ref{th:KJK_large_n} follows from the abelian Theorem \ref{KJKab_n} and the following abelianisation result (a $K$-theoretic analogue of Martin's Theorem \ref{martin}).

\begin{theorem}\label{th:martin_K}
Let $\op{Ad}:=X^{\sT-\ss} \times\_\sT (\fg_\C/\ft_\C)$ with $\sT$ acting on $\fg_\C/\ft_\C$ by the adjoint action as in Theorem \ref{martin}.
For $V \in K^\sG_0(X)$ and $n\gg0$,
    \begin{equation}\label{eq:martin_K_n}
        \chi(X\div \sG, V_0\otimes L_0^n) \= \frac{1}{\vert W \vert}\,\chi\left(X\div\sT,\, \Lambda\udot\op{Ad}^\vee \otimes V_0\otimes L_0^n\right).
    \end{equation}
\end{theorem}
\begin{proof}
For fixed $k\ge0$ the $\sG$-representation $H:=H^k(X, V)$ may be infinite dimensional. Nonetheless the Weyl formula
\begin{equation}\label{eq:G_T_inv}
    \dim H^\sG \= \frac1{|W|}\dim \Bigl(H\otimes \prod\nolimits_{i}(1-t^{\eps_i})\Bigr)^\sT
\end{equation}
still holds, where the right-hand side is the dimension of a virtual representation, and the product is over the roots $\eps_i$ of $\sG$. This is
because the multiplicities $\{\dim H_\psi\}_{\psi \in \ft^\vee_\Z}$ of the irreducible $\sT$-modules in $H$ are all finite by Lemma \ref{pro:chi_well_def} so we can apply the usual Weyl integration formula argument to a finite dimensional $\sG$-submodule of $H$ generated by $H^\sT\oplus\bigoplus_\eps H_{-\eps}$, where $\eps$ runs through all weights appearing in $\prod_i(1-t^{\eps_i})$ (i.e. all subset sums of roots). Taking the alternating sum over $k$ therefore gives
    \begin{equation}\label{eq:new_martin}
        \chi_\sG(X, V)^\sG \= \frac{1}{\vert W \vert}\,\chi_\sT\left(X,\, \Lambda\udot (\fg_\C/\ft_\C)^\vee \otimes V\right)^\sT.
    \end{equation}
Replacing $V$ by $V\otimes L^n$,
    \begin{align*}
        \chi(X\div\sG, V_0\otimes L_0^n) &\stackrel{\eqref{eq:chi_on_quotient}}{=} \dim H^0(X, V\otimes L^n)^\sG\\
        &\stackrel{\eqref{eq:G_T_inv}}{=} \tfrac1{|W|}H^0\bigl(X, \Lambda\udot(\fg_\C/\ft_\C)^\vee\otimes V\otimes L^n\bigr)^\sT\\
        &\stackrel{\eqref{eq:chi_on_quotient}}{=} \tfrac1{|W|}\chi\bigl(X\div\sT, \Lambda\udot\op{Ad}^\vee\otimes V_0\otimes L_0^n\bigr).\qedhere
    \end{align*}
\end{proof} 

\subsection{Quasi-polynomials}
We claim we can  substitute $n=0$ into \eqref{eq:martin_K_n} to give the more general abelianisation formula
    \begin{equation}\label{eq:martin_K2}
        \chi(X\div \sG, V_0) \= \frac{1}{\vert W \vert}\,\chi\left(X\div\sT,\, \Lambda\udot\op{Ad}^\vee \otimes  V_0\right).
    \end{equation}
This would be clear if both sides were polynomial in $n$, as $\chi(Y,V\otimes L^n)$ indeed is for $Y$ a projective scheme. But it need not be true for orbi-bundles $V,\,L$ on a projective Deligne-Mumford stack. Instead we push forwards to the coarse moduli scheme $|Y|$ and use the projection formula and the fact that $L^d$ descends for some $d$. The upshot is that if we restrict to $n$ with a fixed residue $[n]$ mod $d$ then $\chi(Y,V\otimes L^n)$ \emph{is} polynomial. Thus for all $n$ it is determined by $d$ polynomials and is called a \emph{quasi-polynomial}, satisfying the following definition with lattice $\Z\cdot[L]$. (Or we can let $L$ vary through the lattice $\Pic(Y)$ to see it is a quasi-polynomial in the more general sense.)

    \begin{definition}\label{qp}
		A ($d$-periodic) \textit{quasi-polynomial} on a lattice $L$ is a function $P : L \to \Q$ determined by polynomials $\lbrace p_\ell \in \Sym(L_\Q^\vee) : \ell \in L/dL \rbrace$  by
		\begin{equation*}
		    P(\psi) \= p\_{[\psi]}(\psi) \quad \text{ for every }\psi \in L. 
		\end{equation*}
	\end{definition} 

Since both sides of \eqref{eq:martin_K_n} are quasi-polynomials in $n$, their equality for $n\gg0$ holds for every $n$ and in particular for $n=0$, proving \eqref{eq:martin_K2}.

\subsection{Toric JK residues}
    Given a point $p \in \sT$, consider the (analytic) exponential map
 $$
        \exp_p : \ft_\C\, \To\, \sT, \qquad \psi\,\Mapsto\, p \cdot e^\psi
 $$
    giving analytic local coordinates around $p \in \sT$. Alternatively it can be described algebraically via the ring homomorphism 
$$
        \exp_p^\ast : H^0(\cO_{\sT \setminus \cH_\sT}) \,\To\, \widehat{R}_\Gamma, \qquad t^\psi \,\Mapsto\, t^\psi(p) e^\psi
$$
    pulling back rational functions on $\sT$ to the complete ring \eqref{eq:R_gamma}. Recall \cite[Definition 6]{DeConciniProcesi}.
    
	\begin{definition}\label{def:TJK}
		Given $\eta \in \ft_\Q \setminus \cH$ and $\xi \in \ft^\vee_\Q \setminus \cH^\vee$, the \textit{toric Jeffrey--Kirwan residue} is
$$
			\TJK^\eta_\xi : H^0(\cO_{\sT \setminus \cH_\sT}) \,\To\, \C, \qquad f \,\Mapsto\, \sum\nolimits_{p} \JK^\eta_\xi \bigl( \exp_p^\ast f \bigr),
$$
	the sum being over the finite number of isolated intersection points $p \in \sT$ of at least $r$ independent affine hypertori (translates of codimension 1 subtori of $\sT$) in $\cH_\sT$.
	\end{definition}
	
    \begin{remarks}
By Proposition \ref{JKpm}, $\TJK^\eta_\xi$ is a linear functional, depending on $\eta$ and $\xi$ only through the chambers of $\ft_\Q \setminus \cH$ and $\ft^\vee_\Q \setminus \cH^\vee$ they lie in. It doesn't depend on $\Gamma$, in the sense that for $f \in H^0(\cO_{\sT \setminus \cH_{\sT,1} \cap \cH_{\sT,2}})$ and $\eta \notin \cH_1 \cup \cH_2$, $\xi \notin \cH_1^\vee \cup \cH_2^\vee$, the residue $\TJK^\eta_\xi(f)$ is the same when computed with respect to $\Gamma_1$ or $\Gamma_2$. Finally, $\TJK^{-\eta}_{-\xi}=(-1)^r\TJK^\eta_\xi.\hfill\square$
    \end{remarks}


$\TJK^\eta_\xi$ is related to $\CT^{-\eta}$ by \cite[Remark 7.7\,(2) and Theorem 7.10]{DeConciniProcesi}. To explain this let $\Gamma_\eta:=\{\gamma\in\Gamma\cup(-\Gamma):\<\gamma,\eta\>>0\}$ be the set of $\eta$-normalised weights.

    \begin{proposition}[De Concini--Procesi]
		Given $\eta \in \ft_\Q \setminus \cH$ and $\xi \in \ft_\Q^\vee\setminus \cH^\vee$ the function 
		\begin{equation*}
		    \ft^\vee_\Z \,\To\, \Z, \qquad \psi \,\Mapsto\, \TJK^{\eta}_{\xi}\bigl(t^{\psi} f\bigr)
		\end{equation*}
        is a quasi-polynomial in the sense of Definition \ref{qp} for every $f \in H^0(\cO_{\sT \setminus \cH_\sT})$.
        Fixing $d_\gamma\in\Z_{\ge 0}$ for $\gamma \in \Gamma_\eta$, and $\psi$ in the closure of the chamber of $\ft_\Q^\vee \setminus \cH^\vee$ containing $\xi$,
        \begin{equation}\label{eq:CT_JK_DCP}
            \CT^{-\eta} \left(t^{\psi} \prod\nolimits_{\gamma \in \Gamma_\eta} \frac{1}{(1-t^{-\gamma})^{d_\gamma}}\right) \= \TJK^{\eta}_{\xi}\left(t^{\psi} \prod\nolimits_{\gamma \in \Gamma_\eta} \frac{1}{(1-t^{-\gamma})^{d_\gamma}}\right).
        \end{equation}
	\end{proposition}

But \emph{any} $f \in H^0(\cO_{\sT\setminus \cH_\sT})$ can be written as a linear combination of functions of the form 
$$
t^\nu \cdot\prod\nolimits_{\gamma \in \Gamma_\eta} 
        \frac{1}{(1-t^{\gamma})^{d_\gamma}}\ \text{ for some }\nu \in \ft^\vee_\Z\text{ and }d_\gamma\geq 0,
$$
and if an integral $\psi\in\ft_\Z^\vee$ is in the \textit{open} chamber of $\ft^\vee_\Q \setminus \cH^\vee$ containing $\xi$ then so is $\nu+n\psi$ for $n\gg0$. Thus \eqref{eq:CT_JK_DCP} gives
    \begin{equation}\label{eq:CT_JK}
        \CT^{-\eta}(t^{n\psi} f) \= \TJK^\eta_\xi(t^{n\psi}f) \qquad \text{for }n\gg0.
    \end{equation}

	\subsection{Proof of Theorem \ref{th:KJK}}
  We can finally prove Theorem \ref{th:KJK} (with $n=0$) from Theorem \ref{th:KJK_large_n} (with $n\gg0$) by using the fact that the relevant functions of $n$ are quasi-polynomials.
  
         First suppose  $\sG= \sT$ and $\mu(F) \notin \cH^\vee$ for every component $F\subseteq X^\sT$. The contribution of $F\subseteq X^\sT$ to the right hand side of the localisation formula \eqref{eq:KJKab_large_n} was computed in \eqref{eq:K_F_contribution},
$$
\CT^{-\eta}\left(\chi^{\vir}_\sT\!\left(\!F,\,\frac{V \otimes L^n\vert_F}{\Lambda\udot\(N^{\vir}_{F/X}\)^\vee}\right)\right)\=\sum_i \chi^{\vir}\(F,\, E_i|_F \otimes\mathsf L^n\)\, \CT^{-\eta}(t^{n\mu(F)}g_i)
$$
for $n\gg0$. But by \eqref{eq:CT_JK} this can be written as
$$
	    \sum\nolimits_{i=1}^\ell \chi^{\vir}\bigl(F,\, E_i|_F \otimes \mathsf L^n\bigr)\, \TJK^{\eta}_{\mu(F)}(t^{n\mu(F)}g_i) \= \TJK^\eta_{\mu(F)}\left( \chi^{\vir}_\sT\left(\!F,\, \frac{V \otimes L^n\vert_F}{\Lambda\udot\(N^{\vir}_{F/X}\)^\vee}\right)\right).
$$
    Thus \eqref{eq:KJKab_large_n}  has become, for $n \gg 0$,
$$
		    \chi^{\vir}\Bigl(X\div\sT, V_0 \otimes L_0^n \Bigr) \ = \sum_{\substack{F \subseteq X^\sT\\\
				\langle \mu(F), \eta \rangle >0}} \TJK^\eta_{\mu(F)}\left( \chi^{\vir}_\sT\left(\!F,\, \frac{V \otimes L^n\vert_F}{\Lambda\udot\(N^{\vir}_{F/X}\)^\vee}\right)\right).
$$
    Since both sides are quasi-polynomial in $n$ they coincide for $n=0$, which gives Theorem \ref{th:KJK} in this case. Under condition (d) the proof is the same on replacing $\Lambda\udot\(N^{\vir}_{F/X}\)^\vee$ by $\sqrt{\mathfrak e^\sT}\(N^{\vir}_{F/X}\)$ and $\sT$ by $\widetilde\sT$.

    If $\mu(F) \in \cH^\vee$ for some $F \subset X^\sT$, we can apply the argument above to the new linearisation $L^N \otimes \eps$, for $\eps \in \ft_\Z$ and $N\gg 0$. Finally, if $\sG$ is nonabelian, the claim follows from the abelian result combined with \eqref{eq:martin_K2}.
    
\subsection{Noncompact quotients}\label{C*sec}
In this Section we relax the $\eta$-semiprojective condition, allowing \emph{noncompact} quotients $X\div\sG$, in the presence of an additional $\C^*$ action on $X$ which allows us to replace $\chi^{\vir}\(X\div\sG,V_0\otimes L_0^n\)$ by the $\C^*$-character $\chi^{\vir}_{\C^*}\(X\div\sG,V_0\otimes L_0^n\)$. For instance in the abelian $\sG=\sT$ and perfect obstruction theory case we will write\footnote{The $L^n$s can then be removed using an affine JK residue; see \cite{Ontani} for details.}
$$
\chi^{\vir}_{\C^*}\(X\div\,\sT,V_0\otimes L_0^n\)\ \text{ in terms of }\ 
\chi^{\vir}_{\C^*\times\sT}\(X,V\!\otimes\!L^n\)\,=\!
\sum_{F \subseteq X^{\C^\ast \times \sT}} \chi^{\vir}_{\C^\ast \times \sT}\!\left(\!F,\, \frac{V\otimes L^n|_F}{\Lambda^\bullet\(N^{\vir}_{F/X}\)^\vee}\!\right)
$$
when the $\C^*$ action contracts $X$ onto $X^{\C^*}$ and the latter is $\eta$-semiprojective (so that $\wt\eta:\C^*\into\sT$ further contracts $X^{\C^*}$ onto a compact $X^{\C^*\times\sT}$ core as in Definition \ref{etasp}).
\medskip
        
So we again fix $\sG\acts(X,L,\EE,\phi)$ satisfying (a),\,(b)\,---\,and one of (c),\,(d)\,---\,as at the beginning of this Section. We replace the assumption that $X$ is $\eta$-semiprojective with the existence of an additional $\C^\ast$ action on $(X,L,\EE,\phi)$, commuting with the action of $\sG$, and an $\eta\in \ft_\Q\setminus\cH$\,---\,where as before $\cH \subset \ft_\Q$ is the hyperplane arrangement \eqref{eq:hyperplanes} defined by the set $\Gamma \subset \ft_\Q^\vee$ of $\sT$-weights of $N^{\vir}_{X^\sT/X}$\,---\,satisfying
    \begin{enumerate}
        \item[(e)] the $\C^\ast$-limits $\lim_{s\to 0}s\cdot x$ exist for every closed point $x\in X$ and
        \item[(f)] the limit locus $X^{\C^\ast}$ is $\eta$-semiprojective (Definition \ref{etasp}).
    \end{enumerate}
For an example consider the action of $\C^\ast \times (\sG=\C^*)$ on $\C^2$ by $(s,t)\cdot (x,y) := (tx, st^{-1}y)$; this satisfies (e) and (f) for $\eta>0$.

So $X^\sG$ and $X\div\sG$ need not be compact but $X^{\C^*\times\sG}$ is, and $X\div\sG$ is $\C^\ast$-semiprojective by Lemma \ref{lem:XdivG_semiproj} below. In particular $\chi^{\vir}_{\C^\ast}(X\div\sG, V_0\otimes L_0^n)\in\C(\!(s^{-1})\!)$ is a Laurent series in $s^{-1}$ by \eqref{eq:novikov_condition} of Theorem \ref{thm:K_localisation} applied to $\C^*\acts X\div\sG$. This Section is devoted to evaluating this Laurent series via Theorem \ref{th:Cast_loc} below.\medskip

We define the operator $\CT^{-\eta}_{-s}$ on rational functions $f:\C^*\times\sT\dashrightarrow \C$ by expanding them as Laurent series in $s^{-1}$ then applying the $(-\eta)$-constant term functional $\CT^{-\eta}$ of \eqref{CTeta} to each coefficient. We can write each $\CT^{-\eta}$ as an integral over a contour as in \eqref{CTeta} but to make this commute with the infinite sum over powers of $s^{-1}$ we need it to be absolutely convergent. Expanding the rational functions \eqref{eq:rational_factor} which arise shows that for $|s|>1$ this requires us to integrate over the countour $\eta(1+\epsilon)\cdot\sT_\R,\,0<\epsilon\ll1$, just outside $\sT_\R$ but inside any poles of these functions which are also outside $\sT_\R$. That is,
\beq{eq:integral_presentation}
\CT^{-\eta}_{-s}(f)(s)\=\frac1{(2\pi i)^r}\int_{\eta(1+\eps) \cdot\sT_\R} f(s,t)\; \frac{d{\mathbf t}}{\mathbf t} \qquad \text{for $|s|>1$ and $0<\eps \ll 1$}.
\eeq

    \begin{theorem}\label{th:Cast_loc}
  For $V \in K^0_{\C^\ast \times \sG}(X)$ and $n\gg0$ we have, in $\C(\!(s^{-1})\!)$,
\begin{eqnarray}\nonumber
  \chi^{\vir}_{\C^\ast}\(X\div\sG, V_0 \otimes L^n_0\) &=&
\frac1{|W|}\Big(\Lambda^\bullet(\fg_\C/\ft_\C)^\vee \cdot\chi_{\C^\ast \times \sT}^{\vir}(X,V\otimes L^n)\Big)^\sT \\ \label{eq:Cast_loc}
&=& \CT^{-\eta}_{-s}\left(\frac1{|W|}\prod\nolimits_{i}(1-t^{\eps_i})\cdot\sum_{F \subseteq X^{\C^\ast \times \sT}}\chi^{\vir}_{\C^\ast \times \sT}\!\left(\!F,\,\frac{V\otimes L^n|_F}{\Lambda^\bullet\(N^{\vir}_{F/X}\)^\vee} \right)\right)\!,
\end{eqnarray}
with $\epsilon_i \in \ft_\Z^\vee$ the roots of $\sG$.
In case {\rm(d)} we get the same formula with $\Lambda\udot\(N^{\vir}_{F/X}\)^\vee$ replaced by $\sqrt{\mathfrak e^{\C^\ast \times\sT}}\(N^{\vir}_{F/X}\)$.
    \end{theorem}
    
        In applications, it can be hard to compute \eqref{eq:Cast_loc} using \eqref{eq:integral_presentation} since it has many (divisors of) poles which restrict movements of the contour of integration. In practice this is solved by writing the rational function as a sum of summands with simpler poles. We can then evaluate \eqref{eq:integral_presentation} for each summand by shrinking or expanding the contour of integration. See \cite{Ontani} for an example. \medskip

In the rest of the Section we work towards proving Theorem \ref{th:Cast_loc}.
We identify the cocharacters of $\C^*\times\sT$ with $\Z\oplus\ft_\Z$, so that $(k,\eta)$ corresponds to the 1-parameter subgroup $\C^*\ni z\mapsto\(z^k, \widetilde\eta(z)\)\in\C^\ast \times \sT$.

    \begin{lemma}\label{lem:k_eta}
  Assumptions (e),\,(f) imply that $X$ is $(k,\eta)$-semiprojective for $k\gg0$.
    \end{lemma}
  
    \begin{proof}
    Since (e),\,(f) and semiprojectivity do not see nonreduced structure we may assume $X$ is reduced. Working one component at a time we also assume it is connected.
    
Fix $g\in H^0(\cO_X)\setminus\{0\}$ equivariant with respect to $\C^*\times\wt\eta$. We first claim that its weight $(\alpha,\beta)\in\Z^2$ has $\alpha\le0$ and whenever $\alpha=0$ we have $\beta<0$ unless $g$ is constant.

Choose $x\in X$ with $g(x)\ne0$. Then assumption (e) defines 
$x^\prime := \lim_{s\to 0}s\cdot x \in X^{\C^\ast}$ so
$$
        g(x^\prime) \= g\(\!\lim\nolimits_{s\to 0} s\cdot x\) \= \lim\nolimits_{s\to 0} \(s^{-\alpha} g(x)\)
$$
shows $\alpha\leq 0$. If $\alpha=0$ we use assumption (f) to define $x^{\prime\prime}:= \lim_{t\to 0}\eta(t)\cdot x^\prime \in X^{\C^\ast\times \widetilde\eta}$ and \begin{equation}
        g(x^{\prime\prime}) \= g\(\!\lim\nolimits_{t\to 0} \widetilde\eta(t) \cdot x^\prime\) \= \lim\nolimits_{t\to 0} \(t^{-\beta} g(x^\prime)) 
    \end{equation}
shows $\beta\leq 0$. Finally, if $(\alpha, \beta)=(0,0)$, we note that $g(x) = g(x^{\prime\prime})$ for every $x \in X$, so $g(X) = g(X^{\C^\ast \times \widetilde\eta})$ takes only finitely many values because $X^{\C^\ast \times \widetilde\eta}$ is projective. Since $X$ is connected $g$ is in fact constant, proving the claim. \medskip

Therefore there are finitely many nonconstant equivariant generators $g_1, \dots, g_n \in H^0(\cO_X)$ with weights $(\alpha_i,\beta_i)\in \Z^2$ as above. Thus we may choose $k\ll0$ such that the $(k,\eta)$-weights $\alpha_i k + \beta_i$ of the $g_i$ are all negative, so $X$ is $(k,\eta)$-semiprojective by Proposition \ref{pro:alg_char_eta}.
    \end{proof}

    \begin{lemma}\label{lem:XdivG_semiproj}
        $X\div\sG$ is $\C^\ast$-semiprojective.
    \end{lemma}    
    
    \begin{proof}
The $(\C^\ast\times \sG)$-equivariant composition
$$
            X \To X_{\mathrm{aff}}\,=\,\Spec H^0(\cO_X) \To \Spec\bigl(H^0(\cO_X)^\sG \bigr),
$$
is dominant because the pull back map on functions is injective. Since the domain $X$ is $(k,\eta)$-semiprojective by Lemma \ref{lem:k_eta}, so is the target by Proposition \ref{pro:eta_semiproj_dominant}. But the target is fixed by $\widetilde\eta(\C^\ast) \subset \sG$, so it is in fact $\C^\ast$-semiprojective. Since the GIT structure morphism 
$$
            |X\div\sG| \= \mathrm{Proj}\left(\bigoplus\nolimits_{n\geq 0} H^0(L^{n})^\sG\right) \To \Spec\bigl(H^0(\cO_X)^\sG \bigr)
$$
        is $\C^\ast$-equivariant and projective, Proposition \ref{pro:eta_semiproj_dominant} implies that $|X\div\sG|$ is also $\C^\ast$-semiprojective. Thus the same is true of $X\div\sG$.
    \end{proof}

\begin{lemma}\label{lem:Cast_expansion}
  For $V \in K^0_{\C^\ast \times \sT}(X)$ the series $\chi_{\C^\ast \times \sT}^{\vir}(X,V)\in\prod_{(a,\psi)\in\Z\oplus\ft^\vee_\Z} \<s^a\cdot t^\psi\>$ can be obtained from the rational function
\beq{ratfn}
  \sum_{F \subseteq X^{\C^\ast \times \sT}}\chi^{\vir}_{\C^\ast \times \sT}\!\left(\!F,\,\frac{V|_F}{\Lambda^\bullet\(N^{\vir}_{F/X}\)^\vee} \right)\ \text{ on }\C^*\times\sT
\eeq
by expanding in powers of $s^{-1}$ then taking the $(-\eta)$-positive expansion \eqref{eq:eta_expansion} of the coefficients. In case {\rm(d)} we get the same formula with $\Lambda\udot\(N^{\vir}_{F/X}\)^\vee$ replaced by $\sqrt{\mathfrak e^{\C^\ast \times\sT}}\(N^{\vir}_{F/X}\)$.
\end{lemma}

\begin{proof}
   Since $X$ is $(k,\eta)$-semiprojective by Lemma \ref{lem:k_eta}, the result is already given by Theorem \ref{thm:K_localisation} if we can show the $(-k,-\eta)$-positive expansion of \eqref{ratfn} is the same as expanding first in $s^{-1}$ and then taking the $(-\eta)$-positive expansion of the coefficients.
    
    Note the poles of \eqref{ratfn} all have denominators of the form
    \begin{equation}\label{eq:rational_factor}
        \frac{1}{(1-s^a t^\psi)^d}\,, \qquad d>0,
    \end{equation}
    where $(a,\psi)$ run over the $(\C^\ast \times \sT)$-weights of $N^{\vir}_{X^{\C^\ast\times \sT}/X}$.
By \eqref{eq:eta_expansion}, the $(-k,-\eta)$-positive expansion acts on \eqref{eq:rational_factor} according to the sign
    \begin{equation*}
       \op{sign}\Big\langle (a,\psi), (k,\eta)\Big\rangle\= \op{sign}(ak+ \langle \psi, \eta \rangle) \= \begin{cases}
           \op{sign}(a) & \text{if $a\neq 0$,}\\
            \op{sign} \langle \psi, \eta \rangle & \text{if $a=0$,}
        \end{cases}
    \end{equation*}
    for $k\gg0$. The first case says that when $a\ne0$ the $(-k,-\eta)$-positive expansion of \eqref{eq:rational_factor} is the expansion in negative powers of $s$ (i.e. in positive powers of $(s^a t^\psi)^{-1}$ when $a>0$ and of 
$s^at^\psi$ when $a<0$)\,---\,i.e. as a Laurent series in $s^{-1}$. And when $a=0$ \eqref{eq:rational_factor} is independent of $s$ and its $(-k,-\eta)$-positive expansion is its $(-\eta)$-positive expansion as a function on $\sT$.
\end{proof}
    
    \begin{proof}[Proof of Theorem \ref{th:Cast_loc}] 
Viewing \eqref{eq:sections_on_quotient} of Proposition \ref{pro:sections_on_quotient} as an isomorphism of $\C^\ast$-repre\-sentations gives, in $\C(\!(s^{-1})\!)$ and for $n\gg0$,
\beqa
        \chi_{\C^\ast}^{\vir}(X\div\sG,V_0 \otimes L_0^n) &=& \chi_{\C^\ast \times \sG}^{\vir}(X,V \otimes L^n)^\sG \\
&=& \sum_{m\ll\infty} \chi_{\C^\ast \times \sG}^{\vir}\bigl(X,V \otimes L^n \otimes s^{-m}\bigr)^{\C^\ast\times  \sG}\,  s^m \\
&=& \frac{1}{|W|}  \sum_{m\ll\infty} \chi_{\C^\ast \times \sT}^{\vir}\bigl(X, \Lambda\udot(\fg_\C/\ft_\C)^\vee\otimes V \otimes L^n \otimes s^{-m}\bigr)^{\C^\ast \times\sT}\,s^m \\
&=& \frac{1}{|W|} \chi_{\C^\ast \times \sT}^{\vir}\bigl(X, \Lambda\udot(\fg_\C/\ft_\C)^\vee\otimes V \otimes L^n\bigr)^{\sT},
\eeqa    
where the third equality is the abelianisation result \eqref{eq:new_martin} for $(\C^\ast \times \sG) \acts X$.

Thus we have a $\C^\ast$-equivariant version of \eqref{eq:new_martin}.
    By Lemma \ref{lem:Cast_expansion} the right hand side can be computed by expanding the contents of the brackets in \eqref{eq:Cast_loc}, taking the $(-\eta)$-positive expansion of the resulting coefficients, then taking the $\sT$-invariant parts. But this is precisely $\CT^{-\eta}_{-s}$.
    \end{proof}

    
\appendix
\section{Some facts from GIT}
We review a few standard facts from Geometric Invariant Theory for the action $\sG\acts(X,L)$ of a reductive algebraic group on a polarised Deligne-Mumford stack. 
The open locus of semistable points is
$$
X^{L-\ss}\=X^{\ss}\ :=\ \bigcup\nolimits_{n \geq 0}\  \bigcup\nolimits_{s \in H^0(X,L^n)^\sG} \ \big\{x\in X\ :\ s(x) \neq 0\big\}.
$$

\begin{remark}\label{rem:ss_on_closed}
    If $X$ is projective-over-affine and $Z\hookrightarrow X$ is a closed invariant substack then $Z^{\ss} = Z \cap X^{\ss}$ by Serre vanishing. This is not true for an arbitrary quasi-projective $X$ since sections of $L^n\vert_Z$ do not necessarily extend to sections on $X$. An example is given by $Z= \AA^1=\PP^1\setminus \lbrace p \rbrace\subset\PP^2 \setminus \lbrace p \rbrace=X$.
\end{remark}

We will always assume that stability\,=\,semistability for $\sG\acts(X,L)$, so that all points of $X^{\ss}$ are stable and hence have finite stabilisers. Therefore the stack quotient of $X^{\ss}$ by $\sG$ is a Deligne-Mumford stack, the \emph{GIT quotient stack}
$$
X\div\sG\ :=\ X^{\ss}/\sG.
$$
We denote the underlying GIT quotient \emph{scheme} by
$$
    |X\div\sG|\ :=\ \mathrm{Proj}\left(\bigoplus\nolimits_{n\geq 0} H^0(X,L^{n})^\sG\right)
$$
since it is the coarse moduli scheme of $X\div\sG$ (this is essentially the well-known fact that the scheme-theoretic GIT quotient is a \emph{geometric quotient} of the stable\,=\,semistable locus of $X$). In particular $\pi_*\;\cO_{X\div\sG}=\cO_{|X\div\sG|}$, where $\pi:X\div\sG\to|X\div\sG|$ is the projection. Of course
$\pi$ is an isomorphism if and only if all semistable points have trivial stabilisers (both as points of the Deligne-Mumford stack $X$ and under the action of $\sG$).

Recall our convention from Notation\;\ref{weak} that projective is meant in the weak sense of projective coarse moduli space.

\begin{proposition}\label{pro:proj_criterion}
    Let $\sG$ act on a polarised projective-over-affine Deligne-Mumford stack $(X, L)$, so that semistable points are stable. The following are equivalent:
    \begin{enumerate}
        \item the following restriction map has finite rank,
        \begin{align}\label{eq:restriction_map}
        H^0(X, \cO_X)^\sG \To H^0(X^{L-\ss}, \cO_{X^{L-\ss}})^\sG,
    \end{align}
        \item  $X\div\sG$ is projective,
        \item  $Y\div\sG$ is projective for every irreducible component $Y \subseteq X$,
        \item  $Y_{\red}\div\sG$ is projective for every irreducible component $Y_{\red}\subseteq X_{\red}$,
        \item $\dim H^0(\cO_Y)^\sG<\infty$ for every irreducible component $Y\subseteq X$ with $Y\cap X^{L-\ss}\ne\emptyset$,
        \item $H^0(\cO_{Y_{\red}})^\sG=\C$ for every irreducible component $Y_{\red}\subseteq X_{\red}$ with $Y_{\red}\cap X^{L-\ss}\ne\emptyset$.
    \end{enumerate}
\end{proposition}
\begin{proof}
    $(1)\iff(2)$ The pullback map $H^0\(|X\div\sG|, \cO_{|X\div\sG|}\)\into H^0(X^{L-\ss}, \cO_{X^{L-\ss}})^\sG$ is an isomorphism because $X^{L-\ss}\to|X\div\sG|$ is a \emph{good} quotient. Hence the projective morphism $p: |X\div\sG| \rightarrow \Spec\!\(H^0(X, \cO_X)^\sG\)$ factors through Spec of the restriction map \eqref{eq:restriction_map},
    \begin{align*}
        p^*\,:\,H^0(X, \cO_X)^\sG \To H^0\(|X\div\sG|, \cO_{|X\div\sG|}\)\ \cong\ H^0(X^{L-\ss}, \cO_{X^{L-\ss}})^\sG.
    \end{align*}
So $p$ factors through a projective surjective morphism $|X\div\sG| \to\Spec(\im p^*)$, meaning $|X\div\sG|$ is projective if and only if $\Spec(\im p^*)$ is, if and only if $\im p^*$ is finite dimensional. \\
    $(2)\so(3)\so(4)$ As closed subschemes, $|Y_{\red}\div \sG|\subseteq|Y\div \sG|\subseteq|X\div\sG|$ are also projective. \\
    $(4)\so(6)$ Since we have proved $(2)\so(1)$, the restriction map \eqref{eq:restriction_map} for $Y_{\red}$ has finite rank. But it is also injective because $Y_{\red}$ is irreducible, so $H^0(\cO_{Y_{\red}})^\sG$ is finite dimensional. It is nonzero because $Y$ intersects $X^{L-\ss}$, and has no zero divisors, so it is $\C$. \\
    $(6)\so(1)$ Finite dimensionality of the bottom left corner of the commutative diagram
\begin{equation*}
\begin{tikzcd}[row sep=15pt]
H^0(X_{\red},\cO_{X_{\red}})^\sG \ar[r]\ar[d, hook']& H^0\Big(X_{\red}^{L-\ss},\cO_{X^{L-\ss}_{\red}}\Big)^\sG \ar[d, hook', shorten <= -1.2mm, shorten >= -.7mm] \\
\bigoplus\limits_{Y:\,Y^{\ss}\ne\emptyset}H^0(Y_{\red},\cO_{Y_{\red}})^\sG \ar[r]& \bigoplus\limits_{Y:\,Y^{\ss}\ne\emptyset}H^0\(Y_{\red}^{L-\ss},\cO_{Y_{\red}^{L-\ss}}\)^\sG
\end{tikzcd}
\end{equation*}
implies the upper horizontal arrow has finite rank. Therefore the exact sequence $0\to\sqrt0_X\to\cO_X\to\cO_{X_{\red}}\to0$ gives
$$
\begin{tikzcd}[row sep=15pt, column sep=15pt]
0 \ar[r]& H^0\(X,\sqrt0_X\)^\sG \ar[r]\ar[d]& H^0(X,\cO_X)^\sG \ar[dr,"c"]\ar[r]\ar[d,"b"]& H^0\(X_{\red},\cO_{X_{\red}}\)^\sG \ar[d] \\
0 \ar[r]& H^0\(X^{L-\ss},\sqrt0_{X^{L-\ss}}\)^\sG \ar[r,"a"]& H^0\(X^{L-\ss},\cO_{X^{L-\ss}}\)^\sG \ar[r]& H^0\Big(X^{L-\ss}_{\red},\cO_{X^{L-\ss}_{\red}}\Big)^\sG
\end{tikzcd}
 $$
with $c$ of finite rank. So the quotient of $\im b$ by the nilpotent subring $(\im a\cap\im b)\subset\im b$ is finite dimensional. Therefore $\im b$ is a finitely generated ring whose reduction is finite dimensional, so is itself finite dimensional. \\
    $(5)\iff(6)$ Since $Y$ is projective over affine the composition $Y_{\red}\into Y\to Y_{\op{aff}}:=\Spec H^0(\cO_Y)$ (the affinisation of $Y$) is a projective morphism. Thus $H^0(\cO_{Y_{\red}})$ is finitely generated as a module over $H^0(\cO_Y)_{\red}$ (it seems likely they're the same, but we do not need this). The same then holds for $\sG$ invariants:
\beq{tinvts}
H^0(\cO_{Y_{\red}})^\sG\,\text{ is finitely generated as a module over its subring }\,H^0(\cO_Y)_{\red}^\sG
\eeq
by the second half of \cite[Lemma 2.3]{brion2011}. Therefore one is finite dimensional if and only if the other is. And since $H^0(\cO_Y)^\sG$ is Noetherian it is finite dimensional if and only if its reduction $H^0(\cO_Y)_{\red}^\sG$ is. So $\dim H^0(\cO_Y)^\sG<\infty\iff\dim H^0(\cO_{Y_{\red}})^\sG<\infty$. But, for $Y$ nonempty, $H^0(\cO_{Y_{\red}})^\sG$ is nonzero and without zero divisors, so has $\dim<\infty$ if and only if it equals $\C$.
\end{proof}

The divisibility assumption in the next result may not be necessary, but given the weak notion of ampleness we are using for line bundles on Deligne-Mumford stacks it is easiest to work with those which descend to the coarse moduli space $|X\div\sG|$.

\begin{lemma}\label{lem:ss_on_ss_is_ss}
 Let $\sG$ act on a polarised projective-over-affine Deligne-Mumford stack $(X, L)$, so that semistable points are stable. Then for $k$ sufficiently large and divisible,
$$
H^0(X^{\ss}, L^k\vert_{X^{\ss}})^\sG\ \cong\ H^0(X,L^k)^\sG.
$$
\end{lemma}

\begin{proof}
Via $p:X^{\ss}\to X\div\sG$ the line bundle $L$ descends to an (orbi-)line bundle $\cO(1)$ such that $p^*\cO(1)\cong L$. Pick a $d>0$ so that the stabiliser group of each point $x$ acts trivially on the fibre $\cO_x(d)$. Then $\cO(d)$ descends under $\pi:X\div\sG\to|X\div\sG|$. Then the idea is to show that $H^0(X^{\ss}, L^{dj}\vert_{X^{\ss}})^\sG\cong H^0(X,L^{dj})^\sG$ for $j\gg0$ by equating both with $H^0\(|X\div\sG|,\cO(dj)\)$.

Since $p$ is a \emph{good quotient} we know $\(p_*\;\cO_{X^{\ss}}\)^\sG\cong\cO_{X\div\sG}$, while $\pi_*\;\cO_{X\div\sG}=\cO_{|X\div\sG|}$ by construction. So by the projection formula for $\rho=\pi\circ p$,
\beqa
H^0\(X^{\ss},L^{dj}|_{X^{\ss}}\)^\sG  &=& H^0\Big(|X\div\sG|,\(\rho_*\;\rho^*\cO(dj)\)^\sG\Big) \\
&=& H^0\(|X\div\sG|,(\rho_*\;\cO_{X^{\ss}})^\sG\otimes\cO(dj)\) \\ &=& H^0\(|X\div\sG|,\cO(dj)\) \\
&=& H^0\(X,L^{dj}\)^\sG\ \text{ when }j\gg0.
\eeqa
To get the last equality recall that $|X\div\sG,\cO(d)|=\mathrm{Proj}\bigoplus\nolimits_{j\geq 0} H^0(X,L^{dj})^\sG$ and that for $j\gg0$ the $j$th piece of the graded ring coincides with global sections of $\cO(dj)$ on its Proj.
\end{proof}

\section{Moment polyhedra of projective-over-affine DM stacks}\label{appendix}

The moment polyhedron (Definition \ref{def:moment_polyhedron}) of a projective-over-affine polarised Deligne-Mumford stack with torus action is the union of the moment polyhedra of its irreducible components; these we study now. So let a torus $\sT$ act on an \emph{irreducible} polarised projective-over-affine Deligne-Mumford stack $(X,L)$. We let
\beq{Cdef}
\cC(X)\ :=\ \Delta\(X_{\mathrm{aff}}\)\ \subseteq\ \ft_\Q^\vee
\eeq
denote the moment polyhedron of the affinisation $X_{\op{aff}}:=\Spec H^0(\cO_X)$ of $X$\,---\,i.e. the polyhedral cone of positive rational linear combinations of $\sT$-weights of $H^0(\cO_X)_{\red}$ (the quotient of the ring $H^0(\cO_X)$ by its nilpotents). Firstly we claim this is the same as the cone made from $H^0(\cO_{X_{\red}})$ instead.

\begin{lemma}\label{newC}
The cone $\cC(X)=\Delta\(X_{\mathrm{aff}}\)$ also equals $\Delta\((X_{\red})_{\mathrm{aff}}\)\subseteq\ft_\Q^\vee$.
\end{lemma}

\begin{proof}
The inclusion $\cC(X)\subseteq\Delta\((X_{\red})_{\mathrm{aff}}\)$ is immediate from
$$
H^0(\cO_X)_{\red}\ \Into\ H^0(\cO_{X_{\red}}).
$$
Recall from \eqref{tinvts} that the latter is finitely generated as a module over the former, so we may pick equivariant generators with weights $\psi_1,\dots,\psi_k$. The $d$th powers of these generators are nonzero but their weights $d\psi_i\not\in\{\psi_1,\dots,\psi_k\}$ for $d\gg0$ (unless $\psi_i=0$), so $d\psi_i\in\cC(X)$. Therefore $\Delta\((X_{\red})_{\mathrm{aff}}\)\subseteq\cC(X)$.
\end{proof}

The following is a slight generalisation of \cite[Lemma 2.18]{PR}.

\begin{lemma}\label{lem:polyhedron_structure}
For $X$ irreducible, $\Delta(X,L)$ is a polyhedron of the form
\begin{equation}\label{eq:polytope}
    \Delta(X,L)\= P\,+\,\cC(X) \qquad \text{where $P$ is a polytope.}
\end{equation}
If moreover $\dim H^0(X,\cO_X)^\sT<\infty$ and there exists a $\eta \in \ft_\Q$ negative on $\cC(X) \setminus \lbrace 0 \rbrace$, then $P$ is the convex hull of $\mu(X^\sT)$ so
    \begin{equation}\label{eq:hull}
    \Delta(X,L)\=\mathrm{Hull}\Big(\mu\(X^\sT\)\Big)\,+\,\cC(X).
\end{equation}
\end{lemma}

\begin{proof}
Replacing $L$ by a power $L^d$ if necessary we may assume it descends to the coarse moduli space $|X|$ and then work with its reduction, because by Lemma \ref{newC} this does not change the cone $\cC(X)$, it only scales $\Delta$,
$$
\Delta(X,L) \= \tfrac1d\Delta\(X,L^d\) \= \tfrac1d\Delta\(|X|_{\red}, L^d\),
$$
and our assumption $\dim H^0(\cO_X)^\sT<\infty$ implies $\dim H^0(\cO_{|X|_{\red}})^\sT=\dim H^0(\cO_{X_{\red}})^\sT<\infty$ by \eqref{tinvts}.
So we assume $X$ is an integral scheme and $L$ is very ample, the restriction of $\cO(1)$ via a $\sT$-equivariant embedding $X\into\PP^m\times\Xaff$.

    We claim \eqref{eq:polytope} holds for $P:= \Delta(\pi\_{\PP^m}(X), \cO(1))$. The inclusion $\supseteq$ is clear by multiplying functions in $H^0(\cO_X)$ and sections of $\cO_{\PP^m}(n)$ pulled back to $X$. So we focus on $\subseteq$.

    Pick $\psi\in\Delta$, so $H^0(X,L^\ell)_{\ell\psi}\ne0$ for some $\ell>0$. We may assume $\ell$ is sufficiently large that the restriction map
$$
H^0(\PP^m, \cO(\ell)) \otimes H^0(X, \cO_X) \To H^0\(X, L^\ell\)
$$    
is onto. Therefore we can find a nonzero element of $H^0(X, L^\ell)_{\ell\psi}$ 
of the form
$$
        \pi_{\PP^m}^\ast s \otimes f \qquad \text{for } s \in H^0(\PP^m,\cO(\ell))\_{\nu\;}, \ \  f \in H^0(X, \cO_X)\_{\ell\psi-\nu}
$$
for some $\nu\in\ft_\Z^\vee$. Thus $\nu/\ell\in P$ and $\ell\psi-\nu\in\cC(X)$, so
\begin{equation*}
\psi\ =\ \frac\nu\ell+\frac{\ell\psi-\nu}\ell\ \in\ P\,+\,\cC(X),
\end{equation*}
completing the proof of \eqref{eq:polytope}.

By \cite[Theorem 1\;(iii)]{Brion}  we know that $P$ is the convex hull of $\mu\(\pi\_{\PP^m}(X)^\sT,\cO(1)\)$ so to prove \eqref{eq:hull} it remains to show 
$\pi\_{\PP^m}(X)^\sT\cong X^\sT$.

By assumption $H^0(X,\cO_X)^\sT\subset H^0(X,\cO_X)$ is a finite dimensional integral domain, so it is the constants $\C$. The cocharacter $\eta$ generates a 1-parameter subgroup $\C^*\subset\sT$ all of whose weights on $H^0(X,\cO_X)$ are strictly negative except for the zero weight space which is just the constants $\C\cong H^0(X,\cO_X)^\sT$. So using nonconstant $\sT$-equivariant generators $f_1,\dots,f_k\in H^0(X,\cO_X)$ to embed $\Xaff\into\C^k$, the only $\C^*$ fixed point is the origin (which lies in $\Xaff$ because it is closed in $\C^k$).\footnote{Alternatively, note the ideal of $\Xaff\into\C^k$ is generated by $\C^*$-equivariant functions of strictly negative weight, so they all vanish at $0\in\C^k$.} Thus $X\into\PP^m\times\C^k$ has $\sT$ fixed locus $X^\sT\times\{0\}$.
\end{proof}

        \begin{lemma}\label{trayn}
  Suppose $\dim H^0(\cO_X)^\sT < \infty$ for a projective-over-affine DM stack $X$ acted on by $\sT$. Then $X^\sT$ is projective and $\dim H^0(\cO_X)_{\psi}<\infty$ for every weight $\psi\in\ft^\vee_\Z.$
            \end{lemma}
            
            \begin{proof}
            By \cite[Lemma 2.3]{brion2011} $H^0(\cO_X)_{\psi}$ is finitely generated as a module over $H^0(\cO_X)^\sT$. Therefore the former is finite dimensional if the latter is.
            
The $\sT$-fixed locus in $X_{\mathrm{aff}}$ is $(X_{\mathrm{aff}})^\sT \cong \Spec\bigl(H^0(\cO_X)/J\bigr)$, where $J$ is the ideal generated by the weight spaces $H^0(\cO_X)_\psi$ for $\psi \neq 0$. In particular, there is a surjection $H^0(\cO_X)^\sT \onto H^0(\cO_X)/J$, so the target is finite dimensional and $(X_{\mathrm{aff}})^\sT$ is projective. By the following commutative diagram, with vertical arrows projective,
$$
\begin{tikzcd}[row sep=1.5em]
X^\sT \ar[d]\ar[r,hook]& X \ar[d] \\
X_{\op{aff}}^\sT \ar[r,hook]& X_{\op{aff}}
\end{tikzcd}
$$
this makes $X^\sT$ projective.
            \end{proof}
           
Suppose a subtorus $\sS$ of $\sT\acts(X,L)$ acts trivially on $X$. Then the next result notes that the $\sT$-moment polytope $\Delta^\sT(X,L)$ is a translate of the $\sT/\sS$-moment polytope
$$
\Delta^{\sT/\sS}(X,L)\ \subset\ (\ft_\Q/\fs_\Q)^\vee\ \subset\ \ft^\vee_\Q.
$$

\begin{lemma}\label{lem:trivial_act_from_polyhedron}
For any $\psi \in \Delta^\sT(X,L)$, the $\sS$ action on $L\otimes\Psi^{-1}$ is trivial and
    $$
    \Delta^\sT(X,L) \= \psi\,+\,\Delta^{\sT/\sS}\(X, L\otimes\Psi^{-1}\).
    $$
In particular $\sS$ acts trivially on an irreducible $X_{\red}$ if and only if $\Delta^{\sT}(X,L)$ is contained in a translate of $\sS^\perp:= \lbrace \psi \in\ft_\Q^\vee\ :\ \psi(\sS) = 1 \rbrace$. And $\sS\subset \sT$ is the biggest subtorus fixing all of $X_{\red}$ if and only if $\mu^\sT(X)$ is a polyhedron of full dimension in a translate of $\sS^\perp$. \hfill$\square$
\end{lemma}

We leave the elementary proof to the reader; for $X$ projective see \cite[Theorem 1\;(i)]{Brion}.
This allows us to show that $\partial_d\Delta\setminus\partial_{d-1}\Delta$ is locally a finite union of dimension $d$ polyhedra.

\begin{lemma} \label{lem:codimension}
For finitely many subtori $\sS \subseteq \sT$ there is an irreducible component $Y\subseteq X^\sS$ with
$$
        \codim\(\mu(Y)\subset\ft_\Q^\vee\)\ =\ \rank(\sS).
$$
For all other subtori $\sS \subseteq \sT$ and irreducible components $Y \subseteq X^\sS$,
$$
        \codim\(\mu(Y)\subset\ft_\Q^\vee\)\ >\ \rank(\sS) \quad\text{and}\quad
       \mu(Y)\ \subseteq\ \partial_{r-\rank(\sS)-1} \Delta.
$$
\end{lemma}

\begin{proof}
We can assume that $X$ is integral. Lemma \ref{lem:trivial_act_from_polyhedron} gives the inequality $\codim(\mu(Y))\ge s:=\rank(\sS)$. Suppose $Y \subseteq X^\sS$ is an irreducible component with $\codim(\mu(Y)) = s$. Let $\sT_Y^0\subseteq\sT_Y\subseteq\sT$ denote the identity component of its stabiliser group. Then $\sS \subseteq \sT_Y^0$ so 
        \begin{align*}
            \rank(\sT_Y^0)\ \le\ \codim(\mu(Y))\ =\  s\ \le\ \rank(\sT_Y^0).
        \end{align*}
Hence $\sS=\sT_Y^0$ is the identity component of the stabiliser group of a generic point of $Y$.

We claim there are only finitely many such subtori of $\sT$. It is enough to prove this for $\PP(V)\times W$, where $V,\,W$ are $\sT$-representations. Write them as sums of 1-dimensional representations $V=\oplus_iV_i,\,W=\oplus_jW_j$ with weights $\psi_i,\,\phi_j$ respectively. The identity component of the stabiliser group of a point $\([x],y\)$ with coordinates $(x_i),\,(y_j)$ is
$$
\sT_{([x],y)}^0\ =\ \bigcap\nolimits_{(i,i')\in I\times I}\ker(\psi_i\psi_{i'}^{-1})\ \cap\ \bigcap\nolimits_{j\in J}\ker(\phi_j),
$$
where $I=\{i\,:\,x_i\ne0\},\,J=\{j\,:\,y_j\ne0\}$. There are only finitely many such subtori.

Conversely, if the inequality $\codim(\mu(Y))\ge s$ is strict then the  polyhedron $\mu(Y)$ lies in an affine subspace $\psi+\mathfrak u_\Q^\vee\subseteq\ft_\Q^\vee$ of codimension $s+1$. This defines a subtorus $\mathsf U\subset\sT$ of rank $s+1$ which stabilises $Y$ by Lemma \ref{lem:trivial_act_from_polyhedron}. Thus $\mu(Y)\subseteq\partial_{r-s-1}\;\Delta$.
\end{proof}

\begin{lemma}\label{lem:interior_polytope}
Let $\sS\subseteq \sT$ be a subtorus of dimension $s$ acting trivially on $X$. Then $\partial^{\sm}_{r-s}\;\Delta = \partial_{r-s}\;\Delta \setminus \partial_{r-s-1}\;\Delta$ lies in the relative interior of $\Delta$.
\end{lemma}

\begin{proof}
We prove the claim for $\sS = \lbrace 1 \rbrace$ first, using the description \eqref{Deltass} of $\Delta\ni\psi$. If $L\otimes\Psi^{-1}$ has stable points then by the openness of stability the same would be true for nearby $\psi'$ (proved for $(\PP^m \times \AA^n,\cO(1))$ by an application of the Hilbert--Mumford criterion in \cite[Theorem 3.3]{GulbrandsenHalleHulek}, for instance). Thus if $\psi$ is in the topological boundary it admits only strictly semistable points. Some of these must have stabilisers of $\dim>0$ by \cite[Remark 8.3]{Kirwan}, so $\psi\in\partial_{r-1}\;\Delta$.

For a general $\sS$ this shows that $\partial^{\sm}_{r-s}\,\Delta^{\sT/\sS}(X,L)$ is in the interior of $\partial_{r-s}\,\Delta^{\sT/\sS}(X,L)$. Since $\Delta^\sT(Y,L) = \psi+\Delta^{\sT/\sS}(X,L)$ by Lemma \ref{lem:trivial_act_from_polyhedron}, this gives the result.
\end{proof}

\bigskip \noindent {\tt{r.ontani@imperial.ac.uk}} \\
\noindent {\tt{richard.thomas@imperial.ac.uk}} \medskip

\noindent Department of Mathematics, Imperial College London, London SW7 2AZ, United Kingdom

\end{document}